\documentclass{amsart}

\usepackage[utf8]{inputenc}
\usepackage[dvips]{graphicx}
\usepackage{amssymb}
\usepackage{latexsym}
\usepackage{multirow}
\usepackage{multicol}
\usepackage{array}
\usepackage{epstopdf}
\usepackage{ragged2e}
\usepackage{graphicx}
\usepackage{hyperref}
\usepackage{enumerate}
\usepackage{float}

\newtheorem{theorem}{Theorem}[section]

\newtheorem{definition}[theorem]{Definition}
\newtheorem{example}[theorem]{Example}

\theoremstyle{definition}

\newtheorem{thm}{Theorem}[section]
\newtheorem{lem}[thm]{Lemma}
\newtheorem{cor}[thm]{Corolary}
\newtheorem{prop}[thm]{Proposition}
\newtheorem{defn}[thm]{Definition}
\newtheorem{obs}[thm]{Observation}
\newtheorem{ejem}[thm]{Example}

\numberwithin{equation}{section}

\title[The Dynamics of Solvable Subgroups of $\text{PSL}\left(3,\mathbb{C}\right)$]{\sc{The Dynamics of Solvable Subgroups of $\text{PSL}\left(3,\mathbb{C}\right)$} } 

\author{Mauricio Toledo-Acosta}
\address{CIMAT, Jalisco s.n, Valenciana, C. P. 36023, Guanajuato, Gto. México.}
\email{gerardo.toledo@cimat.mx}

\thanks{Partially supported by grants of projects PAPPIT UNAM IN101816, PAPPIT UNAM IN110219, CONACYT 282937, FORDECYT 265667} 

\subjclass{Primary 37F99, Secondary 30F40, 20H10, 57M60}

\begin{document}

\maketitle

\begin{abstract}
In this work we study and provide a full description, up to a finite index subgroup, of the dynamics of solvable complex Kleinian subgroups of $\text{PSL}\left(3,\mathbb{C}\right)$. These groups have \emph{simple} dynamics, contrary to strongly irreducible groups. Because of this, we propose to define elementary subgroups of $\text{PSL}\left(3,\mathbb{C}\right)$ as solvable groups. We show that triangular groups can be decomposed in four layers, via the semi-direct product of four types of elements, with parabolic elements in the inner most layers and loxodromic elements in the outer layers. It is also shown that solvable groups, up to a finite index subgroup, act properly and discontinuously on the complement of either a line, two lines, a line and a point outside of the line, or a pencil of lines passing through a point. These results are another step towards the completion of the study of elementary subgroups of $\text{PSL}\left(3,\mathbb{C}\right)$.
\end{abstract}

\maketitle

\section*{Introduction}

Kleinian groups are discrete subgroups of $\text{PSL}\left(2,\mathbb{C}\right)$, the group of biholomorphic automorphisms of the complex projective line $\mathbb{CP}^1$, acting properly and disconti-nuously on a non-empty region of $\mathbb{CP}^1$. Kleinian groups have been thoroughly studied since the end of the 19th century by Lazarus Fuchs, Felix Klein, Henri Poincar\'e (who named them after Klein in \cite{poincare1883memoire}) and many others. Kleinian groups have played a major role in several fields of mathematics, such as Riemann surfaces and Teichm\"uller theory, automorphic forms, holomorphic dynamics, conformal and hyperbolic geometry, etc. For a detailed study of Kleinian groups, see \cite{maskit}, \cite{taniguchi}. In \cite{seade2001actions}, Jos\'e Seade and Alberto Verjovsky introduced the notion of complex Kleinian groups, which are discrete subgroups of $\text{PSL}\left(n+1,\mathbb{C}\right)$ acting properly and discontinuously on an open invariant subset of $\mathbb{CP}^n$. In the last decade, there has been a great effort to complete the study of the dynamics of complex Kleinian groups in dimension 2 (see for example \cite{bcn11cuatrolineas,bcn14unalinea,bcn16,bgn18,ckg_libro,cs2014}). In this regard, one of the pieces left to complete is the full description of the discrete subgroups of $\text{PSL}\left(3,\mathbb{C}\right)$ which have \emph{simple} dynamics. These subgroups of $\text{PSL}\left(3,\mathbb{C}\right)$ are the analogous of elementary subgroups of $\text{PSL}\left(2,\mathbb{C}\right)$. However, there is no \emph{standard} definition of elementary subgroups of $\text{PSL}\left(3,\mathbb{C}\right)$. \\

In the context of Kleinian groups, \emph{elementary groups} are discrete subgroups of $\text{PSL}\left(2,\mathbb{C}\right)$ such that the limit set is a finite set, in which case the limit set has 0, 1, or 2 points. These groups have \emph{simple} dynamics. In complex dimension 2, the Kulkarni limit set is either a finite union of complex lines (1, 2 or 3) or it contains an infinite number of complex lines. On the other hand, the limit set either contains a finite number of lines in general position (1, 2, 3 or 4) or it contains infinitely many lines in general position \cite{bcn16}. Therefore, one could define elementary groups in complex dimension 2 as discrete subgroups of $\text{PSL}\left(3,\mathbb{C}\right)$ such that the Kulkarni limit set contains a finite number of lines, or discrete subgroups of $\text{PSL}\left(3,\mathbb{C}\right)$ such that the Kulkarni limit set contains a finite number of lines in general position. However, this definition would not be useful in complex dimension greater than 2 since there is no similar classification of the Kulkarni limit set in this case. \\


In this work, we propose to define elementary complex Kleinian subgroups of $\text{PSL}\left(3,\mathbb{C}\right)$ as solvable groups, which also generalizes the notion of elementary groups of $\text{PSL}\left(2,\mathbb{C}\right)$. We will show that solvable groups exhibit \emph{simple} dynamics contrary to the \emph{rich} dynamics of strongly irreducible discrete subgroups of $\text{PSL}\left(3,\mathbb{C}\right)$, which have been extensively studied \cite{bcn2011,ckg_libro}. The study of solvable groups will help to complete the classification and understanding of the dynamics of all complex Kleinian groups of $\text{PSL}\left(3,\mathbb{C}\right)$. This work is a step towards the implementation of a Sullivan's dictionary between the theory of iteration of functions in several complex variables and the dynamics of complex Kleinian groups.\\

The main purpose of this paper is to provide a precise description of discrete solvable subgroups of $\text{PSL}\left(3,\mathbb{C}\right)$ and their dynamics, up to a finite index subgroup. This paper generalizes and builds upon the work done in \cite{ppar}, in which the authors study complex Kleinian groups whose elements are parabolic. To make these generalizations, additional and new techniques were developed in this work.\\

Before stating the main theorem of this paper, it is necessary to introduce some notation. We denote by $\text{PSL}\left(3,\mathbb{C}\right)$ the group of biholomorphic automorphisms of the complex projective plane $\mathbb{CP}^2$. We denote the upper triangular subgroup of $\text{PSL}\left(3,\mathbb{C}\right)$ by
	$$U_{+}=\left\{\left[\begin{array}{ccc}
	a_{11} & a_{12} & a_{13} \\
	0 & a_{22} & a_{23} \\
	0 & 0 & a_{33}
	\end{array}\right]\,\middle|\,a_{11}a_{22}a_{33}=1,\text{ }a_{ij}\in\mathbb{C}\right\}.$$
Let $\lambda_{12},\lambda_{23},\lambda_{13}:(U_{+},\cdot)\rightarrow(\mathbb{C}^{\ast},\cdot)$ be the group morphisms given by
	$$\begin{array}{ccc}
	\lambda_{12}\left(\left[\alpha_{ij}\right]\right)=\alpha_{11}\alpha^{-1}_{22}, &
	\lambda_{23}\left(\left[\alpha_{ij}\right]\right)=\alpha_{22}\alpha^{-1}_{33}, &
	\lambda_{13}\left(\left[\alpha_{ij}\right]\right)=\alpha_{11}\alpha^{-1}_{33}.	
	\end{array}$$

If $\Gamma\subset U_{+}$ is a group, we simplify the notation by writing $\text{Ker}\left(\lambda_{ij}\right)$ instead of $\text{Ker}\left(\lambda_{ij}\right)\cap \Gamma$. In this paper we show the following theorem

\begin{thm}\label{thm_main_solvable}
Let $\Gamma\subset\text{PSL}\left(3,\mathbb{C}\right)$ be a triangularizable complex Kleinian group such that its Kulkarni limit set does not consist of exactly four lines in general position. Then, there exists a non-empty open region $\Omega_\Gamma\subset \mathbb{CP}^2$ such that
	\begin{enumerate}[(i)]
	\item $\Omega_\Gamma$ is the maximal open set where the action is proper and discontinuous.
	\item $\Omega_\Gamma$ is homeomorphic to one of the following regions: $\mathbb{C}^2$, $\mathbb{C}^2\setminus\{0\}$, $\mathbb{C}\times\left(\mathbb{H}^+\cup\mathbb{H}^-\right)$ or $\mathbb{C}\times\mathbb{C}^\ast$.
	\item $\Gamma$ is finitely generated and $\text{rank}(\Gamma)\leq 4$.
	\item The group $\Gamma$ can be written as
		$$\Gamma = \Gamma_p\rtimes\langle \eta_1 \rangle\rtimes...\rtimes\langle \eta_m \rangle\rtimes\langle\gamma_1\rangle\rtimes...\rtimes\langle\gamma_n\rangle$$
	where $\Gamma_p$ is the subgroup of $\Gamma$ consisting of all the parabolic elements of $\Gamma$. The elements $\eta_1,...,\eta_m$ are loxo-parabolic elements such that $\lambda_{12}(\eta_1),...,\lambda_{12}(\eta_m)$ generate the group $\lambda_{12}\left(\text{Ker}(\lambda_{23})\right)$. The elements $\gamma_1,...,\gamma_n$ are strongly loxodromic and complex homotheties such that $\lambda_{23}(\eta_1),...,\lambda_{23}(\eta_m)$ generate the group $\lambda_{23}\left(\Gamma\right)$.

	\item The group $\Gamma$ leaves a full flag invariant. 
	\end{enumerate}
\end{thm}

Our ultimate goal is to prove the previous theorem for solvable subgroups of $\text{PSL}\left(3,\mathbb{C}\right)$ satisfying the hypothesis of Theorem \ref{thm_main_solvable}. Every solvable subgroup of $\text{PSL}\left(3,\mathbb{C}\right)$ contains a finite index subgroup conjugate to some subgroup of $U_+$. As a consequence of this, a fundamental step towards this goal is to prove Theorem \ref{thm_main_solvable}. This is the first main result proved in this paper. The second result we prove in this paper is Theorem \ref{thm_main_gen}, which is a partial generalization of Theorem \ref{thm_main_solvable} for solvable subgroups of $\text{PSL}\left(3,\mathbb{C}\right)$. 

\begin{thm}\label{thm_main_gen}
Let $\Gamma\subset\text{PSL}\left(3,\mathbb{C}\right)$ be a solvable complex Kleinian group such that its Kulkarni limit set does not consist of exactly four lines in general position. Let $\Gamma_0\subset\Gamma$ be a virtually triangularizable finite index subgroup. If $\Gamma_0$ is commutative then there exists a non-empty open region $\Omega_\Gamma\subset \mathbb{CP}^2$ such that

\begin{enumerate}[(i)]
\item $\Omega_\Gamma$ is the maximal open set where the action is proper and discontinuous.
\item $\Omega_\Gamma$ is homeomorphic to one of the following regions: $\mathbb{C}^2$, $\mathbb{C}^2\setminus\{0\}$, $\mathbb{C}\times\left(\mathbb{H}^+\cup\mathbb{H}^-\right)$ or $\mathbb{C}\times\mathbb{C}^\ast$.
\item Up to a finite index subgroup, the group $\Gamma$ leaves a full flag invariant. 
\end{enumerate}
\end{thm}

The paper is organized as follows: In Section \ref{sec_background} we quickly review the necessary facts and definitions. In Section \ref{sec_non_commutative_triangular}, we give a dynamic description of non-commutative discrete triangular subgroups of $\text{PSL}\left(3,\mathbb{C}\right)$ and prove a decomposition theorem for these groups, via semi-direct products. In Section \ref{sec_commutative_triangular}, we characterize and describe the dynamics of commutative discrete triangular subgroups of $\text{PSL}\left(3,\mathbb{C}\right)$. Finally, we prove Theorems \ref{thm_main_solvable} and \ref{thm_main_gen} in Section \ref{sec_proof_main_thm}.

\section{Notation and Background} \label{sec_background}

\subsection{Complex Kleinian groups}

The complex projective plane $\mathbb{CP}^2$ is defined as
	$$\mathbb{CP}^2=(\mathbb{C}^{3}\setminus\{0\})/\mathbb{C}^\ast,$$
where $\mathbb{C}^\ast:=\mathbb{C}\setminus\{0\}$ acts by the usual scalar multiplication. Let $\left[\;\right]:\mathbb{C}^3\setminus\{0\}\rightarrow\mathbb{CP}^2$ be the quotient map. We denote the projectivization of the point $x=(x_1,x_2,x_3)\in\mathbb{C}^3$ by $[x]=\left[x_1:x_2:x_{n+1}\right]$. We denote by $e_1,e_2,e_3$ the projectivization of the canonical base of $\mathbb{C}^3$.\\

Let $\mathcal{M}_{3}\left(\mathbb{C}\right)$ be the group of all square matrices of size $3\times 3$ with complex coefficients, and let $\text{SL}\left(3,\mathbb{C}\right)\subset\mathcal{M}_{3}\left(\mathbb{C}\right)$ (resp. $\text{GL}\left(3,\mathbb{C}\right)$) be the subgroup of matrices with determinant equal to $1$ (resp. not equal to $0$). The group of biholomorphic automorphisms of $\mathbb{CP}^2$ is given by 
	$$\text{PSL}\left(3,\mathbb{C}\right)\cong\text{PGL}\left(3,\mathbb{C}\right):=\text{GL}\left(3,\mathbb{C}\right)/\{\text{scalar matrices}\}.$$
If $g\in\text{PSL}\left(3,\mathbb{C}\right)$ (resp. $z\in\mathbb{CP}^2$), we denote by $\mathbf{g}\in\text{SL}\left(3,\mathbb{C}\right)$ any of its lifts (resp. $\mathbf{z}\in\mathbb{C}^3$). We denote again by $\left[\;\right]:\mathbb{C}^3\setminus\{0\}\rightarrow\mathbb{CP}^2$ the quotient map. We denote by $\text{Fix}(g)\subset\mathbb{CP}^2$ the set of fixed points of an automorphism $g\in\text{PSL}\left(3,\mathbb{C}\right)$.\\

Now we describe a helpful construction used to reduce the action of a group $\Gamma\subset\text{PSL}\left(3,\mathbb{C}\right)$ on $\mathbb{CP}^2$ to the action of a subgroup of $\text{PSL}\left(2,\mathbb{C}\right)$ on a complex line in $\mathbb{CP}^2$, thus, simplifying the study of the dynamics of $\Gamma$ (see Chapter 5 of \cite{ckg_libro} for details). Consider a subgroup $\Gamma\subset\text{PSL}\left(3,\mathbb{C}\right)$ acting on $\mathbb{CP}^2$ with a global fixed point $p\in\mathbb{CP}^2$. Let $\ell\subset\mathbb{CP}^2\setminus\{p\}$ be a projective complex line. We define the projection $\pi=\pi_{p,\ell}:\mathbb{CP}^2\rightarrow\ell$ given by $\pi(x)=\ell\cap\overleftrightarrow{p,x}$. This function is holomorphic, and it allows us to define the group homomorphism 
	$$\Pi=\Pi_{p,\ell}:\Gamma\rightarrow\text{Bihol}(\ell)\cong\text{PSL}\left(2,\mathbb{C}\right)$$
given by $\Pi(g)(x)=\pi(g(x))$ for $g\in\Gamma$ (Lemma 6.11 of \cite{cs2014}). To simplify the notation we will write $\text{Ker}(\Gamma)$ instead of $\text{Ker}(\Pi)\cap \Gamma$.\\

Consider $M\in\mathcal{M}_{3}\left(\mathbb{C}\right)$. Let $\text{Ker}(M)$ be its kernel, and consider its projectivization $\left[\text{Ker}(M)\setminus\{0\}\right]$. Then $M$ induces a well defined map $\left[ M\right]:\mathbb{CP}^2\setminus \left[\text{Ker}(M)\setminus\{0\}\right]\rightarrow \mathbb{CP}^2$ given by $\left[M \right](z)=\left[M\mathbf{z}\right]$. We say that $M$ is a quasi-projective map. We define the space of quasi-projective maps $\text{QP}\left(3,\mathbb{C}\right)$ as
	$$\text{QP}\left(3,\mathbb{C}\right)=\left(\mathcal{M}_{3}\left(\mathbb{C}\right)\setminus\{0\}\right)\bigg/ \mathbb{C}^\ast.$$
This space is the closure of $\text{PSL}\left(3,\mathbb{C}\right)$ in the space $\mathcal{M}_{3}\left(\mathbb{C}\right)$. Therefore every sequence of elements in $\text{PSL}\left(3,\mathbb{C}\right)$ converge to an element of $\text{QP}\left(3,\mathbb{C}\right)$ (see Section 3 of \cite{cano2010equicontinuity} or Section 7.4 of \cite{ckg_libro}). We denote by $\text{Diag}(a_1,a_2,a_3)$ the diagonal element of $\text{QP}\left(3,\mathbb{C}\right)$ with diagonal entries $a_1,a_2,a_3$.\\

As in the case of automorphisms of $\mathbb{CP}^1$, we classify the elements of $\text{PSL}\left(3,\mathbb{C}\right)$ in three classes: elliptic, parabolic and loxodromic. However, unlike the classical case, there are several subclasses in each case (see Section 4.2 of \cite{ckg_libro}). We now give a quick summary of the subclasses of elements we will be using.

\begin{defn}
An element $g\in\text{PSL}\left(3,\mathbb{C}\right)$ is said to be:
	\begin{itemize}
	\item \emph{Elliptic} if it has a diagonalizable lift in $\text{SL}\left(3,\mathbb{C}\right)$ such that every eigenvalue has norm 1.
	\item \emph{Parabolic} if it has a non-diagonalizable lift in $\text{SL}\left(3,\mathbb{C}\right)$ such that every eigenvalue has norm 1.
	\item \emph{Loxodromic} if it has a lift in $\text{SL}\left(3,\mathbb{C}\right)$ with an eigenvalue of norm distinct of 1. Furthermore, we say that $g$ is:
		\begin{itemize}
		\item \emph{Loxo-parabolic} if it is conjugate to an element $h\in\text{PSL}\left(3,\mathbb{C}\right)$ such that
			$$\mathbf{h}=\left(
			\begin{array}{ccc}
 			\lambda & 1 & 0 \\
 			0 & \lambda & 0 \\
 			0 & 0 & \lambda^{-2}
			\end{array}\right),\;|\lambda|\neq 1.$$
		\item A \emph{complex homothety} if it is conjugate to an element $h\in\text{PSL}\left(3,\mathbb{C}\right)$ such that $\mathbf{h}=\text{Diag}\left(\lambda,\lambda,\lambda^{-2}\right)$, with $|\lambda|\neq 1$.
		\item A \emph{rational (resp. irrational) screw} if it is conjugate to an element $h\in\text{PSL}\left(3,\mathbb{C}\right)$ such that $\mathbf{h}=\text{Diag}\left(\lambda_1,\lambda_2,\lambda_3\right)$, with $\left| \lambda_1 \right|=\left|\lambda_2 \right| \neq \left|\lambda_3\right|$ and $\lambda_1\lambda_2^{-1}=e^{2\pi i \theta}$ with $\theta\in \mathbb{Q}$ (resp. $\theta\in\mathbb{R}\setminus\mathbb{Q}$).
		\item \emph{Strongly loxodromic} if it is conjugate to an element $h\in\text{PSL}\left(3,\mathbb{C}\right)$ such that $\mathbf{h}=\text{Diag}\left(\lambda_1,\lambda_2,\lambda_3\right)$, where the elements $\{|\lambda_1|,|\lambda_2|,|\lambda_3|\}$ are pairwise different.
		\end{itemize}		 
	\end{itemize}
\end{defn}

We say that $z\in\mathbb{C}\setminus\{1\}$ is a rational rotation (resp. irrational rotation) if $z=e^{2\pi i \theta}$ for some $\theta\in\mathbb{Q}$ (resp. $\theta\in\mathbb{R}\setminus\mathbb{Q}$). An element $\gamma\in\Gamma$ is called a \emph{torsion element} if it has finite order. The group $\Gamma$ is a \emph{torsion-free group} if the only torsion element in $\Gamma$ is the identity. 

In the classical case of Kleinian groups, there are several characterizations for the concept of limit set, and all of them coincide \cite{taniguchi}. In the complex setting there is no unique definition of limit set (for a detailed discussion and examples see Section 3.1 of \cite{ckg_libro}). In \cite{kulkarni}, Kulkarni introduced a notion of limit set which works in a very general setting (Section 3.3. of \cite{ckg_libro}). 

\begin{defn}\label{defn_kulkarni}
Let $\Gamma$ be a discrete subgroup of $\text{PSL}\left(3,\mathbb{C}\right)$ acting on $\mathbb{CP}^2$. Let us define the following sets:
	\begin{itemize}
	\item Let $L_0(\Gamma)$ be the closure of the set of points in $\mathbb{CP}^2$ with infinite group of isotropy.
	\item Let $L_1(\Gamma)$ be the closure of the set of accumulation points of orbits of points in $\mathbb{CP}^2\setminus L_0(\Gamma)$.
	\item Let $L_2(\Gamma)$ be the closure of the set of accumulation points of orbits of compact subsets of $\mathbb{CP}^2\setminus\left(L_0(\Gamma)\cup L_1(\Gamma)\right)$. 
	\end{itemize}
We define the \emph{Kulkarni limit set} of $\Gamma$ as $ \Lambda_{\text{Kul}}(\Gamma):=L_0(\Gamma)\cup L_1(\Gamma)\cup L_2(\Gamma)$.
The \emph{Kulkarni region of discontinuity} of $\Gamma$ is defined as $\Omega_{\text{Kul}}(\Gamma):=\mathbb{CP}^n\setminus\Lambda_{\text{Kul}}(\Gamma)$.
\end{defn}

The action of $\Gamma$ on $\Omega_{\text{Kul}}(\Gamma)$ is proper and discontinuous (Section 3.2 of \cite{ckg_libro}).

\begin{defn}
The \emph{equicontinuity region} for a familiy $\Gamma$ of automorphisms of $\mathbb{CP}^2$, denoted $\text{Eq}(\Gamma)$, is defined to be the set of points $z\in\mathbb{CP}^2$ for which there is an open neighborhood $U$ of $z$ such that $\Gamma$ restricted to $U$ is a normal family. 
\end{defn}

The following propositions will be very useful to compute the Kulkarni limit set of a group in terms of the quasi-projective limits of sequences in the group.

\begin{prop}[Proposition 7.4.1 of \cite{ckg_libro}]\label{prop_convergencia_qp}
Let $\{\gamma_k\}\subset\text{PSL}\left(3,\mathbb{C}\right)$ be a sequence of distinct elements, then there is a subsequence of $\{\gamma_k\}$, still denoted by $\{\gamma_k\}$, and a quasi-projective map $\gamma\in\text{QP}\left(3,\mathbb{C}\right)$ such that $\gamma_k\overset{k\rightarrow\infty}{\longrightarrow}\gamma$ uniformly on compact sets of $\mathbb{CP}^2\setminus\left[\text{Ker}(\gamma)\right]$.
\end{prop}

\begin{prop}[Proposition 2.5 of \cite{clu2017}]\label{prop_descripcion_Eq}
Let $\Gamma\subset\text{PSL}\left(3,\mathbb{C}\right)$ be a group, we say $\gamma\in\text{QP}\left(3,\mathbb{C}\right)$ is a limit of $\Gamma$, in symbols $\gamma\in\text{Lim}(\Gamma)$, if there is a sequence $\{\gamma_m\}\subset\Gamma$ of distinct elements satisfying $\gamma_m\rightarrow\gamma$. Then
	$$\text{Eq}(\Gamma)=\mathbb{CP}^2\setminus\overline{\underset{\gamma\in\text{Lim}(\Gamma)}{\bigcup}\text{Ker}(\gamma)}.$$
\end{prop}

\begin{prop}[Corollary 2.6 of \cite{clu2017}]\label{prop_eq_in_kuld}
Let $\Gamma\subset\text{PSL}\left(3,\mathbb{C}\right)$ be a discrete group, then $\Gamma$ acts properly discontinuously on $\text{Eq}(\Gamma)$. Moreover, $\text{Eq}(\Gamma)\subset\Omega_{\text{Kul}}(\Gamma)$.
\end{prop}

\subsection{Solvable groups}

\begin{definition}
Let $G$ be a group. The derived series $\{G^{(i)}\}$ of $G$ is defined inductively as
	$$\begin{array}{cc}
	G^{(0)}=G, & G^{(i+1)}=\left[G^{(i)},G^{(i)}\right].
	\end{array}$$
One says that $G$ is \emph{solvable} if for some $n\geq 0$, we have $G^{(n)}=\{\mathit{id}\}$. 
\end{definition}

The following are examples of solvable subgroups of $\text{PSL}\left(3,\mathbb{C}\right)$. These groups will be used later as they appear as subgroups of control groups.

\begin{example}\mbox{}
	\begin{itemize}
	\item Denote by $\text{Rot}_\infty\subset\text{PSL}\left(2,\mathbb{C}\right)$ the group of all rotations around the origin, then the \emph{infinite dihedral group} $\text{Dih}_\infty=\langle \text{Rot}_\infty,\;z\mapsto -z\rangle$ is solvable.
	\item The special orthogonal group,
	$$\text{SO}(3)=\left\{\left[\begin{array}{cc}
	a & -c\\
	c & \overline{a}	
	\end{array}\right]\,\middle|\,|a|^2+|c|^2=1\right\}\subset \text{PSL}\left(2,\mathbb{C}\right)$$
	is not solvable.
	\end{itemize}
\end{example}

The next theorem is known as the \emph{topological Tits alternative} (Theorem 1 of \cite{tits} or Theorem 1.3 of \cite{tits2}).

\begin{thm}\label{teo_tits2}
Let $K$ be a local field and let $\Gamma\subset\text{GL}(n,K)$ be a subgroup. Then either $\Gamma$ contains an open solvable group, or $\Gamma$ contains a dense free subgroup.  
\end{thm}

\begin{thm}[Theorem 3.6 of \cite{wehrfritz2012infinite}]\label{teo_solvable_triangularizables}
Let $G\subset\text{GL}\left(3,\mathbb{C}\right)$ be a discrete solvable subgroup, then $G$ is virtually triangularizable. 
\end{thm}

\section{Non-commutative triangular groups}\label{sec_non_commutative_triangular}

In this section, we provide a dynamic description of non-commutative upper triangular discrete subgroups of $\text{PSL}\left(3,\mathbb{C}\right)$. We also give a decomposition theorem of non-commutative subgroups of $U_+$ (Theorems \ref{thm_descomposicion_caso_noconmutativo} and \ref{thm_descomposicion_caso_noconmutativo2}). Subsections \ref{subsec_restrictions} and \ref{subsec_core} consist of results needed for the proofs of the main theorems of this section.

\subsection{Restrictions on the elements of a non-commutative group}
\label{subsec_restrictions}
In this subsection we study the restrictions that irrational screws, irrational ellipto-parabolic elements and complex homotheties impose on the groups they belong to. We show that if a discrete subgroup $\Gamma\subset U_+$ contains an irrational screw, an irrational ellipto-parabolic element or a complex homothety, then $\Gamma$ has to be commutative.\\

These technical results, along with Subsection \ref{subsec_core} will be important to prove the main theorems in Subsections \ref{subsec_decomposition} and \ref{subsec_consequences}. Also, these propositions generalize the results presented in Chapter 5 of \cite{tesisadriana}.\\

The following class of elements will be used often. So, for the sake of simplicity, let us denote

$$h_{x,y}=\left[\begin{array}{ccc}
	1 & 0 & x\\
	0 & 1 & y\\
	0 & 0 & 1\\
	\end{array}\right],\; x,y\in\mathbb{C}.$$

We will use the following argument several times throughout the paper. We state it as Lemma \ref{lem_bloques_conmutativos}, whose proof is a straightforward calculation. 

\begin{lem}\label{lem_bloques_conmutativos}
Let $\alpha=\left[\alpha_{ij}\right]$, $\beta\left[\beta_{ij}\right]$ $\in U_+\setminus\{\mathit{id}\}$ be two distinct elements. We denote by $\alpha_1$ and $\alpha_2$ (resp. $\beta_1$ and $\beta_2$) the upper left and lower right $2\times 2$ blocks of $\alpha$ (resp. $\beta$),

$$
\alpha_1 = \left[
\begin{array}{cc}
\alpha_{11} & \alpha_{12} \\
0 & \alpha_{22}
\end{array}
\right],\;\;\;
\alpha_2 = \Pi(\alpha)=\left[
\begin{array}{cc}
\alpha_{22} & \alpha_{23} \\
0 & \alpha_{33}
\end{array}
\right].
$$
If we regard $\alpha_i$ and $\beta_i$ as elements of $\text{PSL}(2,\mathbb{C})$, then $\text{Fix}(\alpha_1)=\{e_1,\left[\alpha_{12}:\alpha_{22}-\alpha_{11}\right] \}$ and $\text{Fix}(\alpha_2)=\{e_1,\left[\alpha_{23}:\alpha_{33}-\alpha_{22}\right] \}$, analogously for $\text{Fix}(\beta_1)$ and $\text{Fix}(\beta_2)$. Also, 
$$\text{Fix}(\alpha_i)=\text{Fix}(\beta_i)\;\Leftrightarrow \left[\alpha_i,\beta_i \right]=\mathit{id}.$$ 
\end{lem}
	
We first consider the case of irrational screws.

\begin{prop}\label{prop_IS_discreto_invariante}
Let $\gamma\in U_+$ be an irrational screw given by 
	$$\gamma=\text{Diag}(\beta^{-2}e^{-6\pi i \theta},\beta e^{4\pi i \theta},\beta e^{2\pi i \theta}),$$ 
for some $|\beta|\neq 1$ and $\theta\in\mathbb{R}\setminus\mathbb{Q}$. Let $\alpha\in U_+\setminus\langle\gamma\rangle$. If $\langle\alpha,\gamma\rangle$ is discrete, then $\alpha$ is diagonal.
\end{prop}

\begin{proof}
Let $\gamma,\alpha\in U_+$ be as in the statement of the proposition, and assume, without loss of generality, that $|\beta|>1$. Let us denote $\alpha=\left[\alpha_{ij}\right]$, and suppose that $\alpha$ is not diagonal. Since $\theta\in\mathbb{R}\setminus\mathbb{Q}$, there is a sequence such that $e^{2 k \pi i \theta}\rightarrow 1$. Consider the sequence $\{\mu_k:=\gamma^{-k}\alpha\gamma^{k}\}\subset\langle \alpha,\gamma \rangle$. Since $\alpha$ not diagonal, then $\alpha_{12}\neq 0$, $\alpha_{13}\neq 0$ or $\alpha_{23}\neq 0$. Therefore, $\{\mu_k\}$ is a sequence of distinct elements such that 
	$$\mu_k\rightarrow
	\left[\begin{array}{ccc}
	\alpha_{11} & 0 & 0 \\
	0 & \alpha_{22} & \alpha_{23}\\
	0 & 0 & \alpha_{33}\\
	\end{array}	
	\right]\in \text{PSL}\left(3,\mathbb{C}\right).$$ 
Then $\langle \alpha,\gamma \rangle$ is not discrete since it contains a sequence of distinct elements $\{\mu_k\}$ converging to an element of $\text{PSL}\left(3,\mathbb{C}\right)$. 
\end{proof}

We have the following immediate consequence of Proposition \ref{prop_IS_discreto_invariante}.

\begin{cor}\label{cor_IS_conmutativo}
Let $\Gamma\subset U_+$ be a discrete subgroup. Let $\gamma\in\Gamma$ be an irrational screw as in Proposition \ref{prop_IS_discreto_invariante}, then $\Gamma$ is commutative.   
\end{cor}

\begin{cor}\label{cor_IS_RotInf_prohibidas}
Let $\Gamma\subset U_+$ be a discrete subgroup such that $\Sigma=\Pi(\Gamma)$ is non-discrete and $\overline{\Sigma}=\text{Rot}_\infty$. Then $\Gamma$ is commutative.
\end{cor}

\begin{proof}
The elements of $\Sigma$ have the form $\Pi(\gamma)=\text{Diag}\left( e^{2\pi i \theta},e^{-2\pi i \theta}\right)$. Then every element of $\Gamma\setminus\{\mathit{id}\}$ is diagonalizable and therefore, it is an irrational screw. Then $\Gamma$ contains only screws. If there were a rational screw $\gamma\in\Gamma$ then $\lambda_{23}(\Gamma)$ would not be a torsion-free group. Then $\Gamma$ contains only irrational screws, it follows from Corollary \ref{cor_IS_conmutativo} that $\Gamma$ is commutative.
\end{proof}

The following is an immediate consequence of Lemma 5.10 of \cite{ppar}.

\begin{prop}\label{prop_EPI_conmutativo}
Let $\Gamma\subset U_+$ be a discrete group containing an irrational ellipto-parabolic element $\gamma$ with one of the following two forms:
$$\begin{array}{cc}
\gamma = \left[\begin{array}{ccc}
	e^{-4\pi i \theta} & \beta & \gamma \\
	0 & e^{2\pi i \theta} & \mu \\
	0 & 0 & e^{2\pi i \theta}\\
	\end{array}\right], &
	\gamma = \left[\begin{array}{ccc}
	e^{2\pi i \theta} & \beta & \gamma \\
	0 & e^{2\pi i \theta} & \mu \\
	0 & 0 & e^{-4\pi i \theta}\\
	\end{array}\right]\\
	\mu\neq 0 & \beta\neq 0
\end{array},$$
with $\theta\in\mathbb{R}\setminus\mathbb{Q}$. Then $\Gamma$ is commutative. 
\end{prop}

\begin{lem}\label{lem_control_nodiscreto_descartar_grL_vacio}
Let $\Gamma\subset U_+$ be a torsion-free, non-commutative, discrete group such that $\Sigma=\Pi(\Gamma)$ is not discrete. Then $\Lambda_{\text{Gr}}(\Sigma)\neq \emptyset$.
\end{lem}

\begin{proof}
Let us suppose that $\Lambda_{\text{Gr}}(\Sigma)=\emptyset$. According to Theorem 1.14 of \cite{cs2014}, we have two possibilities:
	\begin{enumerate}[(i)]
	\item $\overline{\Sigma}=\text{SO}(3)$. Since $\Gamma$ is solvable, $\overline{\Sigma}$ is solvable. However $\text{SO}(3)$ is not solvable. 
	\item $\overline{\Sigma}=\text{Dih}_\infty$ or $\overline{\Sigma}=\text{Rot}_\infty$. Observe that $\Sigma\cap \text{Rot}_\infty\neq \emptyset$, otherwise $\Sigma$ would be discrete. Let $\Pi(\gamma)\in\Sigma\cap \text{Rot}_\infty$, then $\Pi(\gamma)=\text{Diag}\left(e^{2\pi i \theta}, e^{-2\pi i \theta}\right)$. Since $\lambda_{23}(\Gamma)$ is a torsion-free group, it follows $\theta\in\mathbb{R}\setminus\mathbb{Q}$, and then $\gamma$ is an irrational screw. Conjugating $\gamma$ by a suitable element of $\text{PSL}\left(3,\mathbb{C}\right)$ and applying Corollary \ref{cor_IS_RotInf_prohibidas}, $\Gamma$ would be commutative. 
	\end{enumerate}
	These contradictions verify the lemma.
\end{proof}

Now, we study complex homotheties. For the rest of the paper, we will say that $\gamma\in\text{PSL}\left(3,\mathbb{C}\right)$ is a type I complex homothety if, up to conjugation, $\gamma=\text{Diag}(\lambda^{-2},\lambda,\lambda)$ for some $\lambda\in\mathbb{C}^\ast$ with $|\lambda|\neq 1$.

\begin{prop}\label{prop_HC_discreto_invariante}
Let $\gamma\in \text{PSL}\left(3,\mathbb{C}\right)$ be a type I complex homothety. Let $\alpha\in U_{+}\setminus\langle\gamma\rangle$ be an element such that $\alpha$ is neither a complex homothety nor a screw, then the group $\langle\alpha,\gamma\rangle$ is discrete if and only if $\alpha$ leaves $\Lambda_{\text{Kul}}(\gamma)$ invariant. 
\end{prop}

\begin{proof}
We have that $\Lambda_{\text{Kul}}(\gamma)=\{e_1\}\cup\overleftrightarrow{e_2,e_3}$ (Proposition 4.2.23 of \cite{ckg_libro}). Denote $\alpha=\left[\alpha_{ij}\right]$, a straightforward calculation shows that $\alpha$ leaves $\overleftrightarrow{e_2,e_3}$ invariant if and only if $\alpha_{12}=\alpha_{13}=0$.\\

First we prove that if $\langle\alpha,\gamma\rangle$ is discrete then $\alpha$ leaves $\Lambda_{\text{Kul}}(\gamma)$ invariant. Assume that $\alpha$ doesn't leave $\Lambda_{\text{Kul}}(\gamma)$ invariant, this means that $|\alpha_{12}|+|\alpha_{13}|\neq 0$. Let $\{g_n\}\subset\langle\alpha,\gamma\rangle$ be the sequence of distinct elements given by $g_n:=\gamma^n \alpha \gamma^{-n}$. Clearly, 
	$$g_n\rightarrow \left[\begin{array}{ccc}
	\alpha_{11} & 0 & 0\\
	0 & \alpha_{22} & \alpha_{23} \\
	0 & 0 & \alpha_{33}\\
	\end{array}\right]\in\text{PSL}\left(3,\mathbb{C}\right).$$
Therefore $\langle\alpha,\gamma\rangle$ is not discrete.\\

Now we prove that if $\alpha$ leaves $\Lambda_{\text{Kul}}(\gamma)$ invariant, then $\langle\alpha,\gamma\rangle$ is discrete. We rewrite $\gamma=\text{Diag}(\lambda^{-3},1,1)$. Suppose that $\langle\alpha,\gamma\rangle$ is not discrete, then there exists a sequence of distinct elements $\{w_k\}\subset\langle\alpha,\gamma\rangle$ such that $w_k\rightarrow \mathit{id}$. Since $\alpha$ leaves $\Lambda_{\text{Kul}}(\gamma)$ invariant, then
	$$\alpha=\left[\begin{array}{ccc}
	\alpha_{11} & 0 & 0\\
	0 & \alpha_{22} & \alpha_{23} \\
	0 & 0 & \alpha_{33}\\
	\end{array}\right],$$
and therefore $\left[\alpha,\gamma\right]=\mathit{id}$. Then the sequence $w_k$ can be expressed as 
	\begin{equation}\label{eq_dem_prop_HC_discreto_invariante_1}
	w_k=\gamma^{i_k}\alpha^{j_k}.
	\end{equation}
Observe that $h\gamma h^{-1}=\gamma$ for any $h\in\text{PSL}\left(3,\mathbb{C}\right)$ of the form
	$$h=\left[\begin{array}{ccc}
	h_{11} & 0 & 0\\
	0 & h_{22} & h_{23} \\
	0 & 0 & h_{33}\\
	\end{array}\right].$$
Furthermore, there exists an element $h\in\text{PSL}\left(3,\mathbb{C}\right)$ with the previous form such that $h\alpha h^{-1}$ has one of the following forms:
	\begin{enumerate}
	\item 
		$$h\alpha h^{-1}=\text{Diag}\left(\alpha_{22}^{-1}\alpha_{33}^{-1},\alpha_{22},\alpha_{33}\right)$$
	if $\Pi(\alpha)$ is loxodromic. In this case, from the previous equation and (\ref{eq_dem_prop_HC_discreto_invariante_1}) it follows $w_k=\text{Diag}\left(\lambda^{-3i_k}\alpha_{22}^{-j_k}\alpha_{33}^{-j_k},\alpha_{22}^{j_k},\alpha_{33}^{j_k}\right)\rightarrow \mathit{id}$. Therefore $\alpha_{22}^{j_k}$, $\alpha_{33}^{j_k}\rightarrow 1$, but then we cannot have $\lambda^{-3i_k}\alpha_{22}^{-j_k}\alpha_{33}^{-j_k}\rightarrow 1$, since $|\lambda|\neq 1$.
	\item $$h\alpha h^{-1}=\left[\begin{array}{ccc}
	\alpha_{11} & 0 & 0\\
	0 & 1 & 1\\
	0 & 0 & 1\\
	\end{array}\right],$$
	if $\Pi(\alpha)$ is parabolic. Then we have
	$$w_k=\left[\begin{array}{ccc}
	\lambda^{-3i_k}\alpha_{11}^{j_k} & 0 & 0\\
	0 & 1 & 1 \\
	0 & 0 & 1\\
	\end{array}\right]\rightarrow \mathit{id}.$$
	But this cannot happen.	
	\item $h\alpha h^{-1}=\text{Diag}\left(\alpha_{11},\alpha_{22},\alpha_{22}\right)$, with $|\alpha_{22}|=1$, if $\Pi(\alpha)$ is elliptic. But in this case $\alpha$ is a complex homothety or a screw.
	\end{enumerate}	 
Therefore neither of the three cases can occur and thus, $\langle\alpha,\gamma\rangle$ is discrete.
\end{proof}

From the proof of the previous proposition, we have the following immediate consequence.

\begin{cor}\label{cor_forma_afin_HC}
If $\Gamma\subset U_+$ is a discrete subgroup, and $\Gamma$ contains a type I complex homothety as in the previous proposition, then every element of $\Gamma$ has the form
	\begin{equation}\label{eq_cor_forma_afin_HC}
	\alpha= \left[\begin{array}{ccc}
	\alpha_{11} & 0 & 0\\
	0 & \alpha_{22} & \alpha_{23} \\
	0 & 0 & \alpha_{33}\\
	\end{array}\right].
	\end{equation} 
\end{cor}

\begin{prop}\label{prop_HC_conmutativo}
Let $\Gamma\subset U_+$ be a discrete group containing a type I complex homothety. If the control group $\Pi(\Gamma)$ is discrete, then both conclusions hold:
	\begin{enumerate}
	\item $\Pi(\Gamma)$ is purely parabolic or purely loxodromic.
	\item $\Gamma$ is commutative.
	\end{enumerate}
\end{prop}

\begin{proof}
Using Corollary \ref{cor_forma_afin_HC} we know that every element of $\Gamma$ has the form given by (\ref{eq_cor_forma_afin_HC}). Denote the control group by $\Sigma=\Pi(\Gamma)$:
\begin{enumerate}
\item If $\Sigma$ contains an elliptic element $\Pi(\alpha)=\text{Diag}\left(e^{2\pi i \theta_1},e^{2\pi i \theta_2}\right)$, then $\alpha= \text{Diag}\left(e^{-2\pi i (\theta_1+\theta_2)}, e^{2\pi i \theta_1}, e^{2\pi i \theta_2}\right)\in\Gamma$ is elliptic. If $\alpha$ has infinite order, this contradicts that $\Gamma$ is discrete and if $\alpha$ has finite order, this contradicts that $\Gamma$ is torsion-free. Therefore, $\Sigma$ has no elliptic elements.  	  
\item If $\Sigma$ has a parabolic element $\Pi(\alpha)$ then $\Sigma$ is purely parabolic. Otherwise, there exists $\beta\in\Gamma$ such that $\Pi(\beta)$ is loxodromic and then, using the J\o rgensen inequality, it would follow that $\langle \Pi(\alpha),\Pi(\beta)\rangle \subset\Sigma$ is not discrete. Then $\Gamma$ is commutative, since every element of $\Gamma$ has the form
	$$\alpha= \left[\begin{array}{ccc}
	\alpha_{11} & 0 & 0\\
	0 & 1 & \alpha_{23} \\
	0 & 0 & 1\\
	\end{array}\right].$$
\item If $\Sigma$ has a loxodromic element $\Pi(\alpha)$, then $\Sigma$ is purely loxodromic (see the previous case). Let $\beta\in\Gamma\setminus\{\alpha\}$, then $\Pi(\beta)$ is loxodromic. Since both $\Pi(\alpha)$ and $\Pi(\beta)$ are upper triangular elements, they share at least one fixed point. We have two cases:
	\begin{enumerate}
	\item If $\Pi(\alpha)$ and $\Pi(\beta)$ share exactly one fixed point, then they have the form
		$$\begin{array}{cc}
		\Pi(\alpha)=\left[\begin{array}{cc}
	\alpha_{22} & \alpha_{23}\\
	0 & \alpha_{33} \\
	\end{array}\right], & \Pi(\beta)=\left[\begin{array}{cc}
	\beta_{22} & \beta_{23}\\
	0 & \beta_{33} \\
	\end{array}\right]
		\end{array}$$
	with $\alpha_{23},\beta_{23}\neq 0$ and $[\alpha_{23}:\alpha_{33}-\alpha_{22}]\neq[\beta_{23}:\beta_{33}-\beta_{22}]$. Then $[\Pi(\alpha),\Pi(\beta)]\neq \mathit{id}$ is parabolic, contradicting that $\Sigma$ is purely loxodromic. 	\item If $\text{Fix}(\Pi(\alpha))=\text{Fix}(\Pi(\beta))$, then $\Pi(\alpha)$ and $\Pi(\beta)$ commute (Lemma \ref{lem_bloques_conmutativos}). Since every element of $\Gamma$ has the form given by (\ref{eq_cor_forma_afin_HC}), it follows that $\alpha$ and $\beta$ commute, and thus, $\Gamma$ is commutative.
\end{enumerate} 	
	\end{enumerate}
\end{proof}

\begin{prop}\label{prop_HC_conmutativo_control_no_discreto}
Let $\Gamma\subset U_+$ be a discrete group such that $\Sigma=\Pi(\Gamma)$ is not discrete and $\Lambda_{Gr}(\Sigma)=\mathbb{S}^1$. Then $\Gamma$ cannot contain a type I complex homothety.
\end{prop}

\begin{proof}
Suppose that $\Gamma$ contains such a complex homothety. Corollary \ref{cor_forma_afin_HC} states that each element in $\Gamma$ has the form given in (\ref{eq_cor_forma_afin_HC}). Since $\Lambda_{Gr}(\Sigma)=\mathbb{S}^1$, then $\Gamma$ is a non-elementary, non-discrete group, and since $\Lambda_{Gr}(\Sigma)$ is the closure of fixed points of loxodromic elements of $\Sigma$, then there are an infinite number of loxodromic elements in $\Sigma$ sharing exactly one fixed point. Let $f\in\Sigma$ such that $\Pi(f)$ is loxodromic and $f=\text{Diag}\left(\alpha^{-1},\alpha,1\right)$, with $|\alpha|<1$. Let $\gamma_1,\gamma_2\in\Gamma$ be two elements such that $\Pi(\gamma_1),\Pi(\gamma_2)$ are loxodromic and share exactly one fixed point, also $\gamma_i\neq f$, for $i=1,2$. It follows that $\gamma:=\left[\gamma_1,\gamma_2\right]\neq \mathit{id}$ is a parabolic element. Let $\{f^k \gamma f^{-k}\}\subset\Gamma$ be the sequence given by $f^k \gamma f^{-k}=h_{0,y}$, where $y=\alpha^k b_0$ and $b_0$ is the entry 23 of $\gamma$. Then $f^k \gamma f^{-k}\rightarrow\mathit{id}$, contradicting that $\Gamma$ is discrete. If $|\alpha|>1$, we consider the sequence $\{f^{-k} \gamma f^{k}\}$ instead. This concludes the proof.
\end{proof}

\begin{prop}\label{prop_HC_conmutativo_control_no_discreto_grenlim_2ptos}
Let $\Gamma\subset U_+$ be a non-commutative discrete group such that $\Sigma=\Pi(\Gamma)$ is not discrete and $|\Lambda_{\text{Gr}}(\Sigma)|=2$. Then $\Gamma$ cannot contain a type I complex homothety.
\end{prop}

\begin{proof}
Suppose that $\Gamma$ contains such a complex homothety. Then, by Corollary \ref{cor_forma_afin_HC}, each element in $\Gamma$ has the form given by (\ref{eq_cor_forma_afin_HC}). If $|\Lambda_{\text{Gr}}(\Sigma)|=2$, then, up to conjugation, every element of $\Sigma$ has the form $\text{Diag}(\beta,\delta)$ for some $\beta,\delta\in\mathbb{C}^\ast$. Then, using (\ref{eq_cor_forma_afin_HC}), it follows that each element in $\Gamma$ is diagonal and hence, $\Gamma$ would be commutative.
\end{proof}

\begin{prop}\label{prop_HC_conmutativo_control_no_discreto_grenlim_1pto}
Let $\Gamma\subset U_+$ be a non-commutative discrete group such that $\Sigma=\Pi(\Gamma)$ is not discrete and $|\Lambda_{\text{Gr}}(\Sigma)|=1$. Then $\Gamma$ cannot contain a type I complex homothety.
\end{prop}

\begin{proof}
Suppose that there exists such a complex homothety $\gamma\in\Gamma$. Corollary \ref{cor_forma_afin_HC} implies that each element in $\Gamma$ has the form given by (\ref{eq_cor_forma_afin_HC}). Since $\Sigma=\Pi(\Gamma)$ is not discrete and $|\Lambda_{\text{Gr}}(\Sigma)|=1$, then $\Sigma$ is a dense subgroup of $\text{Epa}\left(\mathbb{C}\right)$ containing parabolic elements (Theorem 2.14 of \cite{cs2014}). Then every element of $\Sigma$ has the form 
	$$\Pi(\mu)= \left[\begin{array}{cc}
	a & b \\
	0 & a^{-1} \\
	\end{array}\right], \text{ for }|a|=1\text{ and }b\in\mathbb{C}^\ast.$$
This implies that every element of $\Gamma$ has the form
	$$\mu= \left[\begin{array}{ccc}
	\alpha^{-2} & 0 & 0\\
	0 & \alpha a & \alpha b \\
	0 & 0 & \alpha a^{-1} \\
	\end{array}\right]\text{, for some }\alpha\in\mathbb{C}^\ast.$$
We have two possibilities:
\begin{enumerate}[(1)]
\item If every element of $\Sigma$ is parabolic, then $\Gamma$ would be commutative. This is because every element of $\Gamma\setminus\{\mathit{id}\}$ has the form
	$$\mu= \left[\begin{array}{ccc}
	\alpha^{-2} & 0 & 0\\
	0 & \alpha & \alpha b \\
	0 & 0 & \alpha \\
	\end{array}\right]\text{, for some }\alpha\in\mathbb{C}^\ast.$$
\item If there is an elliptic element $\Pi(\gamma)\in\Sigma$, then $\Pi(\gamma)= \text{Diag}\left(e^{2\pi i \theta}, e^{-2\pi i \theta}\right)$, for some $\theta\not\in\mathbb{Z}$. Then 
	\begin{equation}\label{eq_dem_prop_HC_conmutativo_control_no_discreto_grenlim_1pto_1}
	\gamma= \text{Diag}\left(\beta^{-2},\beta e^{2\pi i \theta},\beta e^{-2\pi i \theta}\right)\text{, for some }\beta\in\mathbb{C}^{\ast}.
	\end{equation}		
	Observe that, if $\theta\in\mathbb{Q}$, then $\lambda_{23}(\gamma)=e^{4\pi i \theta}$ would be a torsion element in $\lambda_{23}(\Gamma)$. This contradictions implies that, if there exists an elliptic element in $\Sigma$, then $\theta\in\mathbb{R}\setminus\mathbb{Q}$. On the other hand, since $\Sigma$ is not discrete, there is a sequence of distinct elements $\{\Pi(\mu_k)\}\subset\Sigma$ such that $\Pi(\mu_k)\rightarrow\mathit{id}$. We have two cases:
		\begin{enumerate}[(i)]
		\item If $\{\Pi(\mu_k)\}$ contains an infinite number of parabolic elements, then consi-dering an adequate subsequence, we can assume that $\{\Pi(\mu_k)\}$ is a sequence of distinct parabolic elements. Let us denote
			$$\Pi(\mu_k)= \left[\begin{array}{cc}
	1 & b_k \\
	0 & 1 \\
	\end{array}\right]$$
where $\{b_k\}\subset\mathbb{C}^{\ast}$ is a sequence of distinct elements such that $b_k\rightarrow 0$. Then
	$$\mu_k = \left[\begin{array}{ccc}
	\alpha_k^{-2} & 0 & 0\\
	0 & \alpha_k & \alpha_k b_k \\
	0 & 0 & \alpha_k\\
	\end{array}\right]\text{, for some }\{\alpha_k\}\subset \mathbb{C}^{\ast}.$$
Let $\{\xi_k\}\subset\Gamma$ be the sequence of distinct elements given by $\xi_k = \left[\gamma,\mu_k\right]= h_{0,y}$, where $y=b_k\left(1- e^{-4\pi i \theta}\right)$ and $\gamma$ is an elliptic element in $\Sigma$ given by (\ref{eq_dem_prop_HC_conmutativo_control_no_discreto_grenlim_1pto_1}). Then $\xi_k\rightarrow \mathit{id}$, contradicting that $\Gamma$ is discrete.
		\item  If $\{\Pi(\mu_k)\}$ contains only a finite number of parabolic elements, then we can assume that the whole sequence $\{\Pi(\mu_k)\}$ is made up of irrational elliptic elements. We denote
		$$\Pi(\mu_k)= \left[\begin{array}{cc}
	e^{2\pi i \theta_k} & b_k \\
	0 & e^{-2\pi i \theta_k} \\
	\end{array}\right],$$
	with $b_k\rightarrow 0$. Since $\{\theta_k\}\subset\mathbb{R}\setminus\mathbb{Q}$, we can pick an adequate subsequence of $\{\Pi(\mu_k)\}$, still denoted in the same way, such that $\theta_k\rightarrow 0$ by distinct elements $\{\theta_k\}$. Let $\{\Pi(\sigma_k)\}\subset\Sigma$ be the sequence of distinct elements given by $\Pi(\sigma_k)= \left[\Pi(\mu_k),\Pi(\mu_{k+1})\right]$. This sequence is made up of distinct parabolic elements of $\Sigma$, and since $\theta_k\rightarrow 0$, we have that $\Pi(\sigma_k)\rightarrow\mathit{id}$. Applying the same argument as in the previous case (i), we get a contradiction. 
		\end{enumerate}
\end{enumerate}
\end{proof}

Together, Lemma \ref{lem_control_nodiscreto_descartar_grL_vacio} and Propositions \ref{prop_HC_conmutativo}, \ref{prop_HC_conmutativo_control_no_discreto}, \ref{prop_HC_conmutativo_control_no_discreto_grenlim_2ptos}, \ref{prop_HC_conmutativo_control_no_discreto_grenlim_1pto} imply the following conclusion.

\begin{cor}\label{cor_HC_no_hay_en_no_conmutativos}
Let $\Gamma\subset U_+$ be a non-commutative, torsion-free discrete subgroup, then $\Gamma$ cannot contain a type I complex homothety.  
\end{cor}

\begin{prop}\label{prop_kernel_finito_kernel_trivial}
Let $\Gamma\subset U_+$ be a non-commutative, torsion-free discrete group. If $\text{Ker}(\Gamma)$ is finite, then $\text{Ker}(\Gamma)=\{\mathit{id}\}$.
\end{prop}

\begin{proof}
Let $\Gamma\subset U_+$ be a non-commutative, torsion-free discrete group. Let $\gamma\in\text{Ker}(\Gamma)$, and assume that $\gamma\neq\mathit{id}$. Then $\gamma$ has the form
$$\gamma = \left[\begin{array}{ccc}
	\alpha^{-2} & x & y \\
	0 & \alpha & 0 \\
	0 & 0 & \alpha\\
	\end{array}\right],$$
for some $x,y\in\mathbb{C}$ and $\alpha\in\mathbb{C}^\ast$ such that either $|x|+|y|\neq 0$ or $\alpha\neq 1$. If $|x|+|y|\neq 0$, we can form the sequence of distinct elements $\{\gamma^k\}\subset \text{Ker}(\Gamma)$ contradicting that $\text{Ker}(\Gamma)$ is finite. If $\alpha\neq 1$, we have two possibilities: $|\alpha|\neq 1$ or $\alpha=e^{2\pi i \theta}$ for some $\theta\not\in\mathbb{Z}$. The former implies that $\gamma$ is a complex homothety, thus contradicting Corollary \ref{cor_HC_no_hay_en_no_conmutativos}. The latter implies that either $\alpha$ is an irrational rotation (contradicting that $\Gamma$ is discrete) or $\lambda_{12}(\gamma)$ is a torsion element (contradicting that $\Gamma$ is a torsion-free group). Therefore, $\text{Ker}(\Gamma)=\{\mathit{id}\}$. 
\end{proof}

\begin{lem}\label{lem_parte_parab_no_discreta}
Let $\Sigma$ be a non-discrete, upper triangular subgroup of $\text{PSL}\left(2,\mathbb{C}\right)$ such that $\Lambda_{Gr}(\Sigma)=\mathbb{S}^1$. Then the parabolic part of $\Sigma$ is a non-discrete group.
\end{lem}

\begin{proof}
Let $\Sigma_p$ be the parabolic part of $\Sigma$. Let $g\in\Sigma$ be a loxodromic element such that $g=\text{Diag}\left(\alpha,\alpha^{-1}\right)$ for some $|\alpha|<1$. Let $h_1,h_2\in\Sigma$ be two loxodromic elements such that $\text{Fix}(h_1)\neq\text{Fix}(h_2)$ and $\text{Fix}(h_i)\neq\text{Fix}(g)$. Let $h=\left[h_1,h_2\right]$, then $h\neq\mathit{id}$ is a parabolic element in $\Sigma$. Hence $f_k:=g^k h g^{-k}\rightarrow\mathit{id}$, and therefore $\Sigma$ is not discrete. The fact that one can take the elements $g,h_1,h_2$ is a consequence of $\Lambda_{Gr}(\Sigma)=\mathbb{S}^1$.
\end{proof}

\subsection{The core of a group}\label{subsec_core}

In this subsection we define and study an important purely parabolic subgroup of a complex Kleinian group $\Gamma$ which determines the dynamics of $\Gamma$. This subgroup will be called the \emph{core} of $\Gamma$. We state and prove several technical results which will play an important role in the proofs of the main theorems of Subsections \ref{subsec_decomposition} and \ref{subsec_consequences}. Let us define 
	$$\text{Core}(\Gamma)=\text{Ker}(\Gamma)\cap\text{Ker}(\lambda_{12})\cap\text{Ker}(\lambda_{23}).$$
	
The following proposition describes the elements of the $\text{Core}(\Gamma)$.

\begin{prop}\label{prop_forma_core}
The elements of $\text{Core}(\Gamma)$ have the form
	$$g_{x,y}=\left[\begin{array}{ccc}
	1 & x & y \\
	0 & 1 & 0 \\
	0 & 0 & 1\\
	\end{array}\right],$$
for some $x,y\in \mathbb{C}$.
\end{prop}

\begin{proof}
Let $\gamma=\left[ \gamma_{ij} \right]\in\text{Core}(\Gamma)$. Since $\gamma\in \text{Ker}(\Gamma)$, it follows that $\gamma_{23}=0$ and $\gamma_{22}=\gamma_{33}=\gamma_{11}^{-2}$. Therefore, $\gamma$ has the form:
$$\gamma=\left[\begin{array}{ccc}
	\alpha^{-2} & \gamma_{12} & \gamma_{13} \\
	0 & \alpha & 0 \\
	0 & 0 & \alpha\\
	\end{array}\right],$$
for some $\alpha\in\mathbb{C}^\ast$.  Since $\gamma\in \text{Ker}(\lambda_{12})$, it follows that $\alpha^{-2} = \alpha$, and  
$$\gamma=\left[\begin{array}{ccc}
	\alpha & \gamma_{12} & \gamma_{13} \\
	0 & \alpha & 0 \\
	0 & 0 & \alpha\\
	\end{array}\right]=
	\left[\begin{array}{ccc}
	1 & \alpha^{-1}\gamma_{12} & \alpha^{-1}\gamma_{13} \\
	0 & 1 & 0 \\
	0 & 0 & 1\\
	\end{array}\right].$$
Therefore, $\gamma=g_{\alpha^{-1}\gamma_{12},\alpha^{-1}\gamma_{13}}$. It is straightforward to verify that $g_{x,y}\in\text{Core}(\Gamma)$ for any $x,y\in\mathbb{C}$.
\end{proof}

We will use the notation of Proposition \ref{prop_forma_core} for the rest of the work. It is straightforward to verify that $\Lambda_{\text{Kul}}\left(\text{Core}(\Gamma)\right)=\bigcup_{g_{x,y}\in\text{Core}(\Gamma)} \overleftrightarrow{e_1,[0:-y:x]}$. We denote this pencil of lines by $\mathcal{C}(\Gamma)=\Lambda_{\text{Kul}}\left(\text{Core}(\Gamma)\right)$.

\begin{prop}\label{prop_cono_invariante}
Let $\Gamma\subset U_+$ be a discrete group, then every element of $\Gamma$ leaves $\mathcal{C}(\Gamma)$ invariant.
\end{prop}

\begin{proof}
For $g_{x,y}\in\text{Core}(\Gamma)$, denote by $\ell_{x,y}=\overleftrightarrow{e_1,[0:-y:x]}\subset \mathcal{C}(\Gamma)$ the line determined by the element $g_{x,y}$. Let $\gamma=\left[\gamma_{ij}\right]\in\Gamma$, observe that
		$$\gamma g_{x,y}\gamma^{-1}=g_{\frac{\gamma_{11}}{\gamma_{22}}x, \frac{\gamma_{11}}{\gamma_{22}\gamma_{33}}(\gamma_{22} y - \gamma_{23} x)}\in \text{Core}(\Gamma).$$ 
	This element determines the line $\ell_{\gamma_{33}x,\gamma_{22}y-\gamma_{23}x}$. Therefore, this is a line in $\mathcal{C}(\Gamma)$. On the other hand, a direct calculation shows that $\gamma(\ell_{x,y})=\ell_{\gamma_{33}x,\gamma_{22}y-\gamma_{23}x}$. This proves that $\gamma$ leaves $\mathcal{C}(\Gamma)$ invariant.
\end{proof}

We have shown that every element $\gamma\in\Gamma$ moves the line $\ell_{x,y}$ to the line $\gamma(\ell_{x,y})$ according to the proof of the last proposition, also, the line $\overleftrightarrow{e_1,e_2}$ is fixed by every element of $\Gamma$. In particular, loxodromic elements leave $\mathcal{C}(\Gamma)$ invariant. This inva-riance imposes strong restrictions on these loxodromic elements.\\

We say that the discrete group $\Gamma\subset U_+$ is \emph{conic} if $\overline{\mathcal{C}(\Gamma)}$ is a cone homeomorphic to the complement of $\mathbb{C}\times\left(\mathbb{H}^+\cup\mathbb{H}^-\right)$. If $\mathcal{C}(\Gamma)$ is a line, we say that $\Gamma$ is \emph{non-conic}. In \cite{ppar}, conic groups are called \emph{irreducible}, and non-conic groups, \emph{reducible}. 

\begin{prop}\label{prop_kernel_igual_core}
Let $\Gamma\subset U_+$ be a non-commutative discrete group such that one of the following hypothesis hold:
	\begin{itemize}
	\item Its control group $\Sigma=\Pi(\Gamma)$ is discrete.
	\item $\Sigma$ is not discrete and $\Lambda_{\text{Gr}}(\Sigma)=\mathbb{S}^1$.
	\item $\Sigma$ is not discrete and $|\Lambda_{\text{Gr}}(\Sigma)|=2$.
	\end{itemize}
Then $\text{Ker}(\Gamma)=\text{Core}(\Gamma)$.
\end{prop}

\begin{proof}
We only have to prove that $\text{Ker}(\Gamma)\subset\text{Core}(\Gamma)$. Let $\gamma\in\text{Ker}(\Gamma)$, then
	$$\gamma=\left[\begin{array}{ccc}
	\alpha^{-2} & \gamma_{12} & \gamma_{13} \\
	0 & \alpha & 0 \\
	0 & 0 & \alpha\\
	\end{array}\right]\text{, for some }\alpha\in\mathbb{C}^\ast$$
	\begin{itemize}
	\item If $|\alpha|\neq 1$, then $\gamma$ is a complex homothety. Using the hypotheses and Propositions \ref{prop_HC_conmutativo}, \ref{prop_HC_conmutativo_control_no_discreto} and \ref{prop_HC_conmutativo_control_no_discreto_grenlim_2ptos}, $\Gamma$ would be commutative.
	\item If $|\alpha|=1$ but $\alpha\neq 1$, then $\alpha=e^{2\pi i \theta}$ for some $\theta\in\mathbb{R}\setminus\mathbb{Q}$ (otherwise, $\lambda_{12}(\Gamma)$ would not be a torsion-free group). Then $\gamma$ is a irrational ellipto-parabolic element, and then $\Gamma$ would be commutative (Proposition \ref{prop_EPI_conmutativo}).
	\end{itemize}
This contradictions imply that $\alpha=1$ and therefore, $\gamma\in\text{Core}(\Gamma)$.
\end{proof}

The proof of the following corollary is similar to the proof of the previous proposition.

\begin{cor}\label{cor_diagonal_es_1}
Under the hypotheses of the previous proposition, if 
	$$\gamma=\left[\begin{array}{ccc}
	\alpha^{-2} & \gamma_{12} & \gamma_{13} \\
	0 & \alpha & \gamma_{23} \\
	0 & 0 & \alpha\\
	\end{array}\right]\in\Gamma,$$
then $\alpha=1$. 
\end{cor}

\begin{prop}\label{prop_control_nodiscreto_conjlim_2puntos_irreducible}
Let $\Gamma\subset U_+$ be a non-commutative, discrete group such that $|\Lambda_{\text{Gr}}(\Pi(\Gamma))|=2$. Let $\ell$ be a line passing through $e_1$ such that $\ell\neq \overleftrightarrow{e_1,e_2}$, $\ell\neq \overleftrightarrow{e_1,e_3}$ and $\ell\subset\mathcal{C}(\Gamma)$, then $\Gamma$ is conic.  
\end{prop}

\begin{proof} 
If $|\Lambda_{\text{Gr}}(\Pi(\Gamma))|=2$, then up to conjugation, every element of $\Sigma$ has the form $\gamma=\text{Diag}\left(\beta,\delta\right)$, where $|\beta|\neq |\delta|$. Then every element of $\Gamma$ has the form
	\begin{equation}\label{eq_dem_prop_control_nodiscreto_conjlim_2puntos_irreducible_1}
	\gamma=\left[\begin{array}{ccc}
	\gamma_{22}^{-1}\gamma_{33}^{-1} & \gamma_{12} & \gamma_{13} \\
	0 & \gamma_{22} & 0 \\
	0 & 0 & \gamma_{33}\\
	\end{array}\right]\text{, with }|\gamma_{22}|\neq |\gamma_{33}|.
	\end{equation}
Let us assume that $\Gamma$ is not conic, with $\mathcal{C}(\Gamma)=\ell$. Let $[0:-y:x]\in\overleftrightarrow{e_2,e_3}$ such that $\ell=\ell_{x,y}$, then, by hypothesis, $x\neq 0$ and $y\neq 0$. Since $|\text{Ker}(\Gamma)|=\infty$, using Proposition \ref{prop_kernel_igual_core} it follows that $\text{Ker}(\Gamma)=\text{Core}(\Gamma)$. Since $\text{Ker}(\Gamma)$ is a normal subgroup of $\Gamma$, so is $\text{Core}(\Gamma)$. Therefore, if $\gamma=\left[\gamma_{ij}\right]\in\Gamma$, then $\gamma g_{x,y} \gamma^{-1} = h_{x',y'}\in\text{Core}(\Gamma)$ where $x'=x\gamma_{22}^{-2}\gamma_{33}^{-1}$ and $y'=y\gamma_{22}^{-1}\gamma_{33}^{-2}$. This element $\gamma g_{x,y} \gamma^{-1}$ determines the same line $\ell_{x,y}$ in $\mathcal{C}(\Gamma)$. Then 
	$$\left[-y\gamma_{22}^{-1}\gamma_{33}^{-2}:x\gamma_{22}^{-2}\gamma_{33}^{-1}\right]=\left[-y:x\right],$$
which means that $\gamma_{22}=\gamma_{33}$, contradicting (\ref{eq_dem_prop_control_nodiscreto_conjlim_2puntos_irreducible_1}). This contradiction proves the proposition.	
\end{proof}

\begin{prop}\label{prop_proyectiv_adit}
Let $W\subset\mathbb{C}^2$ be a non-empty and $\mathbb{R}$-linearly independent set, consider $\ell=\overline{\left[\text{Span}_\mathbb{Z}(W)\setminus\{0\}\right]}$. Then:
	\begin{enumerate}[(i)]
	\item If $W$ has exactly one point or has exactly two points, which are $\mathbb{C}$-linearly dependent, then $\ell$ is a single point.
	\item If $W$ has exactly two points, such that they are $\mathbb{C}$-linearly independent, then $\ell$ is a real line in $\mathbb{CP}^1$.
	\item If $W$ contains more that two points, then $\ell=\mathbb{CP}^1$.
	\end{enumerate}
\end{prop}

\begin{proof}
If $|W|=1$ or $W$ has exactly two points, which are $\mathbb{C}$-linearly dependent, it follows trivially that $\ell$ is a single point. If $W$ has exactly two points, such that they are $\mathbb{C}$-linearly independent, we can assume that $W=\{(1,0),(0,1)\}$, then $\left[\text{Span}_\mathbb{Z}(W)\setminus\{0\}\right]\cong \mathbb{Q}$ and therefore $\ell=\hat{\mathbb{R}}$, which is a line in $\mathbb{CP}^1$. Finally, if $W$ contains more that two points, we can assume that $W=\{(1,0),(0,1),(w_1,w_2)\}$, where $w_1,w_2\neq 0$. Then 
	$$\ell = \overline{\left\{\left[k+nw_1:m+nw_2\right]\,\middle|\,k,m,n\in\mathbb{Z}\right\}} = \left\{\left[r+sw_1:t+sw_2\right]\,\middle|\,r,s,t\in\mathbb{R}\right\}.$$
Clearly $\left[0:1\right]\in\ell$. Let $z\in\mathbb{C}$ such that $\text{Im}(z)\neq 0$, denoting
	$$r_0=\frac{\text{Im}(w_2)-\text{Re}(z)\text{Im}(w_1)}{\text{Im}(z)},\;\; s_0=1,\;\;t_0=z(r+w_1)-w_2,$$
it follows that $\left[1:z\right]=\left[r_0+s_0w_1:t_0+s_0w_2\right]$. If $\text{Im}(z)=0$, clearly $\left[1:z\right]\in\ell$. Therefore $\ell=\mathbb{CP}^1$.
\end{proof}

\subsection{Decomposition of non-commutative discrete groups of $U_+$}\label{subsec_decomposition}

In this subsection we state and prove the main theorem of this section. For the sake of clarity, we will divide the theorem in two parts (Theorems \ref{thm_descomposicion_caso_noconmutativo} and \ref{thm_descomposicion_caso_noconmutativo2}).

\begin{thm}\label{thm_descomposicion_caso_noconmutativo}
Let $\Gamma\subset U_{+}$ be a non-commutative, torsion-free, complex Kleinian group, then $\Gamma$ can be written in the following way
	$$\Gamma = \text{Core}(\Gamma)\rtimes\langle\xi_1\rangle\rtimes...\rtimes\langle\xi_r\rangle\rtimes\langle\eta_1\rangle\rtimes...\rtimes\langle\eta_m\rangle\rtimes\langle\gamma_1\rangle\rtimes...\rtimes\langle\gamma_n\rangle$$
	where
	$$\begin{array}{rl}
	\lambda_{23}(\Gamma)=\langle\lambda_{23}(\gamma_1),...,\lambda_{23}(\gamma_n)\rangle, & n=\text{rank}\left(\lambda_{23}(\Gamma)\right).\\
	\lambda_{12}\left(\text{Ker}(\lambda_{23})\right)=\langle\lambda_{12}(\eta_1),...,\lambda_{12}(\eta_m)\rangle, & m=\text{rank}\left(\lambda_{12}(\text{Ker}(\lambda_{23})\right).\\
	\Pi\left(\text{Ker}(\lambda_{12})\cap\text{Ker}(\lambda_{23})\right)=\langle\Pi(\xi_1),...,\Pi(\xi_r)\rangle, & r=\text{rank}\left(\Pi\left(\text{Ker}(\lambda_{12})\cap\text{Ker}(\lambda_{23})\right)\right).
	\end{array}$$
	
\end{thm}
  
\begin{proof}
We will divide the proof in three parts.\\

\underline{Part I. Decomposition of $\Gamma$ in terms of $\text{Ker}(\lambda_{23})$.} Let $\Gamma\subset U_{+}$ be a torsion-free complex Kleinian group. Since $\Gamma$ is triangular, it is finitely generated (see \cite{auslander}), and therefore $\lambda_{23}(\Gamma)$ is finitely generated. Let $n=\text{rank}\left(\lambda_{23}(\Gamma)\right)$, and let $\{\tilde{\gamma}_1,...,\tilde{\gamma}_n\}\subset\mathbb{C}^\ast$ be a generating set for $\lambda_{23}(\Gamma)$. We choose elements $\gamma_1,...,\gamma_n\in\Gamma$ such that $\gamma_i\in\lambda_{23}^{-1}(\tilde{\gamma}_i)$. Observe that 
	\begin{equation}\label{eq_thm_descomposicion_caso_noconmutativo_1}
	\Gamma=\langle\text{Ker}(\lambda_{23}),\gamma_1,...,\gamma_n\rangle.
	\end{equation}
Furthermore, we will prove that 
	\begin{equation}\label{eq_thm_descomposicion_caso_noconmutativo_2}
	\Gamma=\left(\left(\text{Ker}(\lambda_{23})\rtimes\langle\gamma_1\rangle\right)\rtimes...\right)\rtimes\langle\gamma_n\rangle.
	\end{equation}
Since $\lambda_{23}$ is a group homomorphism, $\text{Ker}(\lambda_{23})$ is a normal subgroup of $\Gamma$, and therefore it is a normal subgroup of $\langle\text{Ker}(\lambda_{23}),\gamma_1\rangle$. Now, assume that $\text{Ker}(\lambda_{23})\cap\langle\gamma_1\rangle$ is not trivial, then there exist $p\in\mathbb{Z}$ such that $\gamma_1^p\in\text{Ker}(\lambda_{23})$. If $\gamma_1=\left[a_{ij}\right]$, then the previous assumption means that either $a_{22}=a_{33}$ or $a_{22}^p=a_{33}^p$. In the latter case, this means, without loss of generality that 
	\begin{equation}\label{eq_thm_descomposicion_caso_noconmutativo_3}
	a_{22}a^{-1}_{33}=\omega,
	\end{equation}	 
where $\omega$ is a $p$-th root of the unity, with $p>1$. On the other hand, $\lambda_{23}(\Gamma)\subset\mathbb{C}^{\ast}$ is a torsion-free group, it follows from (\ref{eq_thm_descomposicion_caso_noconmutativo_3}) that $a_{22}a^{-1}_{33}$ is a torsion element of $\lambda_{23}(\Gamma)$; thus $a_{22}=a_{33}$. If this is the case, then $\gamma_1\in\text{Ker}(\lambda_{23})$, contradicting that $\lambda_{23}(\gamma_1)=\tilde{\gamma}_1$ belongs to a generating set for $\lambda_{23}(\Gamma)$. Thus, we can form the semi-direct product $\text{Ker}(\lambda_{23})\rtimes\langle\gamma_1\rangle$.\\ 

Now we verify that we can form the semi-direct product $\left(\left(\text{Ker}(\lambda_{23})\rtimes\langle\gamma_1\rangle\right)\right)\rtimes\langle\gamma_2\rangle$. First, we verify that $\text{Ker}(\lambda_{23})\rtimes\langle\gamma_1\rangle$ is a normal subgroup of $\langle\text{Ker}(\lambda_{23})\rtimes\langle\gamma_1\rangle,\gamma_2\rangle$. Let $g\in\text{Ker}(\lambda_{23})$ and consider arbitrary elements $g\gamma_1^n\in \text{Ker}(\lambda_{23})\rtimes\langle\gamma_1\rangle$ and $\gamma_2^m\in\langle\gamma_2\rangle$, then
	$$
	\gamma_2^m \left(g \gamma_1^n\right) \gamma_2^{-m} = \underbrace{\left(\gamma_2^m g \gamma_2^{-m}\right)}_{\text{Ker}(\lambda_{23})\triangleleft \Gamma}\gamma_2^m \gamma_1^n\gamma_2^{-m}= g_1\cdot \left(\gamma_2^m \gamma_1^n\gamma_2^{-m}\right)$$
where $g_1=\gamma_2^m g \gamma_2^{-m}\in\text{Ker}(\lambda_{23})$. Since $\lambda_{23}\left(g_1\gamma_2^m \gamma_1^n\gamma_2^{-m}\right)=\lambda_{23}\left(\gamma_1^n\right)$, we know that there exists $g_2\in\text{Ker}(\lambda_{23})$ such that $g_1\gamma_2^m \gamma_1^n\gamma_2^{-m}=g_2 \gamma_1^n$. Thus 
	$$\gamma_2^m \left(g \gamma_1^n\right) \gamma_2^{-m}=g_2 \gamma_1^n\in \text{Ker}(\lambda_{23})\rtimes\langle\gamma_1\rangle.$$
Therefore $\text{Ker}(\lambda_{23})\rtimes\langle\gamma_1\rangle$ is a normal subgroup of $\langle\text{Ker}(\lambda_{23})\rtimes\langle\gamma_1\rangle,\gamma_2\rangle$.\\

Now, we verify that $\left(\text{Ker}(\lambda_{23})\rtimes\langle\gamma_1\right)\rangle\cap \langle\gamma_2\rangle=\{\mathit{id}\}$. Assume that there exist $p,q\in\mathbb{Z}$ such that $\gamma_2^p = g\gamma_1^q$ for some $g\in\text{Ker}(\lambda_{23})$ with $\gamma_2^p\neq\mathit{id}$, then 
	\begin{equation}\label{eq_thm_descomposicion_caso_noconmutativo_4}
	\gamma_2^p\gamma_1^{-q}\in \text{Ker}(\lambda_{23})
	\end{equation}		 	
If we write $\gamma_1=\left[\alpha_{ij}\right]$ and $\gamma_2=\left[\beta_{ij}\right]$, then (\ref{eq_thm_descomposicion_caso_noconmutativo_4}) yields 
	\begin{equation}\label{eq_thm_descomposicion_caso_noconmutativo_5}
	\beta_{22}\alpha_{22}^{-\frac{q}{p}}\left(\beta_{33}\alpha_{33}^{-\frac{q}{p}}\right)^{-1}=\omega,
	\end{equation}
where $\omega$ is a $p$-th root of the unity. Then (\ref{eq_thm_descomposicion_caso_noconmutativo_5}) gives a torsion element in the torsion-free group $\lambda_{23}(\Gamma)$. This contradiction proves that  
	$$\left(\text{Ker}(\lambda_{23})\rtimes\langle\gamma_1\right)\rangle\cap \langle\gamma_2\rangle=\{\mathit{id}\}.$$
This verifies that we can define the semi-direct product $\left(\text{Ker}(\lambda_{23})\rtimes\langle\gamma_1\rangle\right)\rtimes\langle\gamma_2\rangle$. Analogously, we can form the semi-direct product $\left(\left(\text{Ker}(\lambda_{23})\rtimes\langle\gamma_1\rangle\right)\rtimes...\right)\rtimes\langle\gamma_n\rangle$. For the sake of clarity, we will just write $\text{Ker}(\lambda_{23})\rtimes\langle\gamma_1\rangle\rtimes...\rtimes\langle\gamma_n\rangle$ instead. Using (\ref{eq_thm_descomposicion_caso_noconmutativo_1}) we have proven (\ref{eq_thm_descomposicion_caso_noconmutativo_2}). \\

\underline{Part II. Decompose $\text{Ker}(\lambda_{23})$ in terms of $\text{Ker}(\lambda_{12})$.} Now, consider the restriction $\lambda_{12}:\text{Ker}(\lambda_{23})\rightarrow \mathbb{C}^\ast$ still denoted by $\lambda_{12}$. Again, $\lambda_{12}\left(\text{Ker}(\lambda_{23})\right)$ is finitely generated, and let $m$ be its rank. Let $\{\tilde{\eta}_1,...,\tilde{\eta}_m\}$ be a generating set, we choose elements $\eta_i\in\lambda_{12}^{-1}\left(\tilde{\eta}_i\right)$. Denote $A=\text{Ker}(\lambda_{12})\cap \text{Ker}(\lambda_{23})$, and observe that $\text{Ker}(\lambda_{23})=\langle A,\eta_1,...,\eta_m\rangle$. Every element of $A$ has the form
	$$\left[\begin{array}{ccc}
	1 & x & y \\
	0 & 1 & z\\
	0 & 0 & 1\\ 
	\end{array}
	\right].$$
Since $A$ is the kernel of the morphism $\lambda_{12}$, $A$ is a normal subgroup of $\langle A,\eta_1\rangle$.\\ 

Now suppose that $A\cap\langle \eta_1\rangle$ is not trivial, let $\eta_1^p\in A$ with $\eta_1^p\neq\mathit{id}$. Denote $\eta_1=\left[a_{ij}\right]$, since $\eta_1^p\in A$ it must hold $a_{11}^p=a_{22}^p=a_{33}^p$. This means that $a_{11}a_{22}^{-1}=\omega$, where either $\omega\neq 1$ is a $p$-th root of unity or $\omega=1$. In the former, $a_{11}a_{22}^{-1}$ is a torsion element in $\lambda_{12}\left(\text{Ker}(\lambda_{23})\right)$, which is a torsion-free group. If $\omega=1$ then $a_{11}=a_{22}$ and since $\eta_1\in\text{Ker}(\lambda_{23})$, then $a_{22}=a_{33}$. It follows that $a_{11}=a_{22}=a_{33}$ and thus, $\eta_1\in A$ which contradicts that $\eta_1$ is part of the generating set. All this proves that $A\cap\langle \eta_1\rangle= \emptyset$. This guarantees that we can form the semi-direct product $A\rtimes\langle \eta_1\rangle$.\\

Now we verify that we can make the semi-direct product with $\langle \eta_2\rangle$. The same argument used in part I to prove normality when we added $\gamma_2$ to $\text{Ker}(\lambda_{23})\rtimes\langle \gamma_1\rangle$ can be applied in the same way now to prove that $A\rtimes\langle\eta_1\rangle$ is a normal subgroup of $\langle A\rtimes\langle \eta_1\rangle,\eta_2\rangle$ and that $\left(A\rtimes\langle\eta_1\rangle\right)\cap\langle\eta_2\rangle$ is trivial. Using this argument for $\eta_3,...,\eta_m$ we get,
	\begin{equation}\label{eq_thm_descomposicion_caso_noconmutativo_8}
	\text{Ker}(\lambda_{23})=A\rtimes\langle\eta_1\rangle\rtimes...\rtimes\langle\eta_m\rangle.
	\end{equation}
	
\underline{Part III. Decompose $A$ in terms of $\text{Ker}(\Gamma)$.} Consider the restriction $\Pi:A\rightarrow\text{PSL}\left(2,\mathbb{C}\right)$. As before, $\text{Core}(\Gamma)$ it is a normal subgroup of $\Gamma$. Since $A$ is solvable, it is finitely generated and so is $\Pi(A)\subset\text{PSL}\left(2,\mathbb{C}\right)$; denote by $r$ its rank.  Let $\{\tilde{\xi}_1,...,\tilde{\xi}_r\}\subset\Pi(A)$ be a generating set for $\Pi(A)$, choose $\xi_i\in\tilde{\xi}_i^{-1}\subset A$, then $A=\langle\text{Core}(\Gamma),\xi_1,...,\xi_r\rangle$. Observe that $\text{Core}(\Gamma)$ is a normal subgroup of $\langle\text{Core}(\Gamma),\xi_1\rangle$. Assume that $\text{Core}(\Gamma)\cap\xi_1$ is not trivial, then $\xi_1^p\in\text{Core}(\Gamma)$ for some $p\in\mathbb{Z}\setminus\{0\}$. Since $\xi_1\in A$, we write
	$$\xi_1=\left[\begin{array}{ccc}
	1 & x & y \\
	0 & 1 & z\\
	0 & 0 & 1\\ 
	\end{array}
	\right].$$
Then 
	$$\xi^p_1=\left[\begin{array}{ccc}
	1 & x' & y' \\
	0 & 1 & pz\\
	0 & 0 & 1\\ 
	\end{array}
	\right]\text{, for some }x',y'\in\mathbb{C}.$$
Observe that $\xi_1^p\in\text{Core}(\Gamma)$ if and only if $z=0$, which contradicts that $\tilde{\xi}_1$ is a generator of $\Pi(A)$. Then we can form the semi-direct product $\text{Core}(\Gamma)\rtimes\langle\xi_1\rangle$. Observe that $\Pi(A)$ is a commutative subgroup of $\text{PSL}\left(2,\mathbb{C}\right)$, then we can apply the same argument we used when we formed the semi-direct product with $\lambda_2$ to conclude that $\text{Core}(\Gamma)\rtimes\langle\xi_1\rangle$ is a normal subgroup of $\langle\text{Core}(\Gamma)\rtimes\langle\xi_1\rangle,\xi_2\rangle$. This concludes the proof of the first part of the theorem.
\end{proof}

Before proving Theorem \ref{thm_descomposicion_caso_noconmutativo2}, we state some results that will be used in its proof.

\begin{prop}[Chapter 2 of \cite{TPR}]\label{prop_forma_subgrupos_aditivos}
The closed subgroup $H\subset\mathbb{R}^n$ is additive if and only if
	$$H\cong \mathbb{R}^p\oplus\mathbb{Z}^q$$
with non-negative integers $p,q$ such that $p+q\leq n$. 
\end{prop}

\begin{thm}[Theorem 1 of \cite{kapovich02}]\label{thm_obdim_1}
Let $\Gamma$ be a group acting properly and discontinuously on a contractible manifold of dimension $m$, then $\text{obdim}(\Gamma)\leq m$.
\end{thm}

In the statement of the previous Theorem, $\text{obdim}(\Gamma)$ is called the \emph{obstructor dimension}. It satisfies the following properties (see Corollary 27 of \cite{kapovich02} and \cite{bestvina02} respectively):
	\begin{itemize}
	\item If $\Gamma=H\rtimes Q$ with $H$ and $Q$ finitely generated and $H$ weakly convex, then
		$$\text{obdim}(\Gamma)\geq \text{obdim}(H)+\text{obdim}(Q).$$
	\item If $\Gamma=\mathbb{Z}^n$, then $\text{obdim}(\mathbb{Z}^n) = n$.		
	\end{itemize}

We will use the following notation in Theorems \ref{thm_obdim_2}, \ref{thm_descomposicion_caso_noconmutativo2} and Corollary \ref{cor_descomposicion_bloques_noconm}. For a non-commutative, torsion-free, complex Kleinian group $\Gamma\subset U_{+}$ we denote

\begin{align*}
k &= \text{rank}\left(\text{Core}(\Gamma)\right), \\
r &= \text{rank}\left(\Pi\left(\text{Ker}(\lambda_{12})\cap\text{Ker}(\lambda_{23})\right)\right), \\
m &= \text{rank}\left(\lambda_{12}(\text{Ker}(\lambda_{23})\right), \\
n &= \text{rank}\left(\lambda_{23}(\Gamma)\right).
\end{align*}

Some of this notation is taken from Theorem \ref{thm_descomposicion_caso_noconmutativo}. Using the previous properties of the obstructor dimension and Theorem \ref{thm_descomposicion_caso_noconmutativo}, we have the following reformulation of Theorem \ref{thm_obdim_1}.

\begin{thm}\label{thm_obdim_2}
Let $\Gamma\subset U_{+}$ be a non-commutative, torsion-free, complex Kleinian group acting properly and discontinuously on a simply connected domain $\Omega\subset\mathbb{CP}^2$, then $k+r+m+n\leq 4$.
\end{thm}

The strategy to prove Theorem \ref{thm_descomposicion_caso_noconmutativo2} will be to find a simply connected domain $\Omega\subset\mathbb{CP}^2$ where $\Gamma$ acts properly and discontinuously, and then apply Theorem \ref{thm_obdim_2}. In some cases, we will write the explicit decomposition of $\Gamma$ and verify that $\text{rank}(\Gamma)\leq 4$.\\


\begin{thm}\label{thm_descomposicion_caso_noconmutativo2}
Let $\Gamma\subset U_{+}$ be a non-commutative, torsion-free, complex Kleinian group, then $\text{rank}(\Gamma)\leq 4$. Furthermore, $k+r+m+n\leq 4$.
\end{thm}

\begin{proof}
We denote $\Sigma=\Pi(\Gamma)$. Observe that
	\begin{equation}\label{eq_thm_descomposicion_rango_caso_noconmutativo_1}
	\text{rank}(\Gamma)\leq k+r+m+n.
	\end{equation}

The proof of this theorem will require several cases and subcases, each one of them will have a unique label so it can be referenced when needed. For the sake of clarity, we will only give details for the most representative cases, for further details, see \cite{mitesis}.\\
 
We will divide the proof in three main cases:
	\begin{enumerate}[(i)]
	\item $\Sigma$ is discrete, and $\text{Ker}(\Gamma)$ is finite.
	\item $\Sigma$ is discrete, and $\text{Ker}(\Gamma)$ is infinite.
	\item $\Sigma$ is not discrete.	\\
	\end{enumerate}

	\begin{enumerate}[(i)]
	\item Assume that \textbf{$\Sigma$ is discrete and $\text{Ker}(\Gamma)$ is finite}. If $|\Lambda(\Sigma)|\neq2$, let 
		$$\Omega = \left(\bigcup_{z\in\Omega(\Sigma)} \overleftrightarrow{e_1,z}\right)\setminus \{e_1\}.$$
	Using Theorem 5.8.2 of \cite{ckg_libro} we know that $\Gamma$ acts properly and discontinuously on $\Omega$. If $|\Lambda(\Sigma)|=0,1$ or $\infty$ then each connected component of $\Omega$ is simply connected, since they are respectively homeomorphic to $\mathbb{CP}^2$, $\mathbb{C}^2$ or $\mathbb{C}\times\mathbb{H}$ respectively. Using Theorem \ref{thm_obdim_2} it follows $k+r+m+n\leq 4$.\\
		
	 If $|\Lambda(\Sigma)|=2$, then $\Sigma$ is elementary. Therefore, it is a cyclic group generated by a loxodromic element, and then $\Sigma\cong\mathbb{Z}$. On the other hand, $\text{Ker}(\Gamma)=\{\mathit{id}\}$ (Proposition \ref{prop_kernel_finito_kernel_trivial}), then $\Pi:\Gamma\rightarrow\Sigma$ is a group isomorphism, and then $\Gamma\cong \Sigma\cong\mathbb{Z}$. Therefore, $\text{rank}(\Gamma)=1$. \\
	
	\item Now, let us assume that \textbf{$\Sigma$ is discrete and $\text{Ker}(\Gamma)$ is infinite}. Then, $\text{Core}(\Gamma)$ is infinite (Proposition \ref{prop_kernel_igual_core}), which means that there exist elements $g_{x,y}\in\text{Core}(\Gamma)$, with $g_{x,y}\neq\mathit{id}$. Denote $\mathcal{B}(\Gamma)=\pi\left(\mathcal{C}(\Gamma)\setminus\{e_1\}\right)$, then $\overline{\mathcal{B}(\Gamma)}\cong\mathbb{S}^1$ or it is a single point (Proposition \ref{prop_proyectiv_adit}). On the other hand, consider a sequence of distinct elements $\{\mu_k\}\subset \Gamma$. Since $\Sigma$ is discrete, the sequence $\{\Pi(\mu_k)\}$ is either constant or converges, by distinct elements, to a quasi-projective map in $\text{QP}\left(2,\mathbb{C}\right)$. In the latter case, $\Sigma$ is elementary and then, we have three possibilities (Theorem 1.6 of \cite{series}):
		\begin{itemize}
		\item[{\scriptsize \textbf{[e1]}}] $\Sigma=\langle h\rangle$, with $h$ a loxodromic element. 
		\item[{\scriptsize \textbf{[e2]}}] $\Sigma=\langle h\rangle$, with $h$ a parabolic element. 
		\item[{\scriptsize \textbf{[e3]}}] $\Sigma=\langle g,h\rangle$, with $g,h$ parabolic elements with different translation directions.	
		\end{itemize}		 
		
		Depending on whether $\Gamma$ is conic or not, we have the two cases:\\
		
		\begin{itemize}
		\item[{\scriptsize \textbf{[con]}}] $\Gamma$ is conic, then $\mathcal{B}(\Gamma)\cong\mathbb{S}^1$. Define $\Omega(\Gamma)=\mathbb{CP}^2\setminus \overline{\mathcal{C}(\Gamma)}$. We will verify that $\Gamma$ acts properly and discontinuously on $\Omega(\Gamma)$ using an argument which will be used several times through this proof. Let $K\subset \Omega(\Gamma)$ be a compact set, denote $\mathcal{K}=\left\{\gamma\in\Gamma\,\middle|\,\gamma(K)\cap K\neq \emptyset\right\}$, and assume that $|\mathcal{K}|=\infty$. Then we can write $\mathcal{K}=\{\gamma_1,\gamma_2,...\}$. Consider the sequence of distinct elements $\{\gamma_k\}\subset\Gamma$, and the sequence $\{\Pi(\gamma_k)\}\subset\Sigma$. Let $\sigma\in\text{QP}\left(2,\mathbb{C}\right)$ be the quasi-projective limit of $\{\Pi(\gamma_k)\}$, then $\text{Ker}(\sigma),\text{Im}(\sigma)\subset\mathcal{B}(\Gamma)$.\\ 
		By definition, $\pi(K)\cap\mathcal{B}(\Gamma)=\emptyset$, and then Proposition \ref{prop_convergencia_qp} implies that $\Pi(\gamma_k)(\pi(K))$ accumulates on $e_1$ (observe that we are considering this $e_1$ as a point in $\mathbb{CP}^1\cong \overleftrightarrow{e_2,e_3}$, but we are actually referring to $\pi(e_2)$). On the other hand, since $|\mathcal{K}|=\infty$, then $|\left\{\alpha\in\Sigma\,\middle|\,\alpha(\pi(K))\cap \pi(K)\neq \emptyset\right\}|=\infty$. This contradicts the fact that $\Pi(\gamma_k)(\pi(K))$ accumulates on $e_1$. Therefore $|\mathcal{K}|<\infty$, and then $\Gamma$ acts properly and discontinuously on each connected component of $\Omega(\Gamma)$. Since $\Omega(\Gamma)\cong \mathbb{C}\times\left(\mathbb{H}\cup\mathbb{H}^-\right)$, each connected component of $\Omega(\Gamma)$ is simply connected. Theorem \ref{thm_obdim_2} yields $k+r+m+n\leq 4$.
		
		\item[{\scriptsize \textbf{[n-con]}}] $\Gamma$ is not conic, then $\mathcal{B}(\Gamma)=\{e_1\}$ or $\mathcal{B}(\Gamma)=\{e_2\}$ (Proposition \ref{prop_control_nodiscreto_conjlim_2puntos_irreducible}). We have two possibilities: If $\Sigma$ is of the form \small \textbf{[e2]} or \small \textbf{[e3]}, then $\Lambda(\Sigma)=\{e_1\}$ and, as in case \small \textbf{[con]}, it follows that $\Gamma$ acts properly and discontinuously on $\Omega(\Gamma)=\mathbb{CP}^2\setminus\overleftrightarrow{e_1,e_2}\cong \mathbb{C}^2$. If $\Sigma$ has the form \small \textbf{[e1]}, then $\Sigma$ is cyclic. Since $\text{Core}(\Gamma)$ is not conic, then $\text{Core}(\Gamma)\cong \mathbb{Z}$. Since $\pi$ is a group morphism, then $\Gamma\cong \text{Ker}(\Gamma)\rtimes \text{Im}(\Gamma)=\text{Core}(\Gamma)\rtimes \Sigma\cong\mathbb{Z}\rtimes\mathbb{Z}$. In both cases, we have $\text{rank}(\Gamma)\leq 4$.\\					
		\end{itemize}
		
		Now, let us assume that $\{\Pi(\mu_k)\}$ is a constant sequence. Then there exists a sequence $\{g_k\}\subset\ker(\Gamma)$ such that $\gamma_k=\gamma_0 g_k$. It follows that $\{g_k\}\subset\text{Core}(\Gamma)$ (Proposition \ref{prop_kernel_igual_core}). Let $\tau\in\text{QP}\left(3,\mathbb{C}\right)$ be the quasi-projective limit of $\{g_k\}$, since $\text{Im}(\tau)\nsubseteq \text{Ker}(\tau)$, then $\gamma_0 g_k=\gamma_k\rightarrow \gamma_0\tau$. It is straightforward to verify that $\text{Im}(\tau)=\text{Im}(\gamma_0\tau)$ and $\text{Ker}(\tau)=\text{Ker}(\gamma_0 \tau)$. We now have two possibilities, depending on whether $\Gamma$ is conic or not.
		
		If $\Gamma$ is conic, then $\text{Core}(\Gamma)=\langle g_{x_1,y_1},g_{x_2,y_2}\rangle$ for some $x_1,x_2,y_1,y_2\in\mathbb{C}$. Using Lemma \ref{prop_convergencia_qp}, $\Gamma$ acts properly and discontinuously on each component of $\Omega=\mathbb{CP}^2\setminus\overline{\mathcal{C}(\Gamma)}\cong \mathbb{C}\times\left(\mathbb{H}^+\cup\mathbb{H}^-\right)$. If $\Gamma$ is not conic, then $\text{Core}(\Gamma)=\langle g_{x_0,y_0}\rangle$ for some $x_0,y_0\in\mathbb{C}$, and $\mathcal{C}(\Gamma)$ is a line. As in previous cases, $\Gamma$ acts properly and discontinuously on each connected component of $\Omega = \mathbb{CP}^2\setminus\mathcal{C}(\Gamma)$. In both cases, Theorem \ref{thm_obdim_2} completes the proof for the case when $\{\Pi(\mu_k)\}$ is a constant sequence.
			
	\item Assume that \textbf{$\Sigma$ is not discrete}. Let $\{\gamma_k\}\subset\Gamma$ be a sequence of distinct elements, denote $\gamma_k=\left[\gamma_{ij}^{(k)}\right]$. Depending on whether the sequence $\{\Pi(\gamma_k)\}\subset\Sigma$ converges or not, we have two cases:
		\begin{itemize}
		\item[{\scriptsize \textbf{[Conv]}}] The sequence $\{\Pi(\gamma_k)\}\subset\Sigma$ converges to some $\alpha\in\text{PSL}\left(2,\mathbb{C}\right)$ where
				\begin{equation}\label{eq_thm_descomposicion_noconmutativo_conv_forma_limqp}
				\alpha=\left[\begin{array}{cc}
				\gamma_{22} & \gamma_{23} \\
				0 & 1\\ 
				\end{array}\right]\text{, for some }\gamma_{22}\in\mathbb{C}^\ast				
				\end{equation}									 
		such that $\gamma_{22}^{(k)}\left(\gamma_{33}^{(k)}\right)^{-1}\rightarrow \gamma_{22}$ and $\gamma_{23}^{(k)}\left(\gamma_{33}^{(k)}\right)^{-1}\rightarrow \gamma_{23}$. Consider the sequence $\{\gamma_{33}^{(k)}\}\subset\mathbb{C}^{\ast}$, we have three cases:
			\begin{enumerate}
			\item $\gamma_{33}^{(k)}\longrightarrow \gamma_{33}\in\mathbb{C}^{\ast}$. Since $\gamma_{11}^{(k)}\gamma_{22}^{(k)}\gamma_{33}^{(k)}=1$, and considering (\ref{eq_thm_descomposicion_noconmutativo_conv_forma_limqp}), it follows that	$\gamma_{11}^{(k)}\left(\gamma_{33}^{(k)}\right)^{-1}\longrightarrow \gamma_{22}^{-1}\gamma_{33}^{-3}$. Observe that 
			$$\gamma_{12}^{(k)}\left(\gamma_{33}^{(k)}\right)^{-1}\rightarrow \infty\text{ or }\gamma_{13}^{(k)}\left(\gamma_{33}^{(k)}\right)^{-1}\rightarrow \infty.$$ 
			Otherwise, $\gamma_k$ would converge to an element of $\text{PSL}\left(3,\mathbb{C}\right)$, contradicting that $\Gamma$ is discrete. Thus,	
				$$
				\gamma_k\longrightarrow\tau_{a,b}:=\left[\begin{array}{ccc}
				0 & a & b \\
				0 & 0 & 0\\
				0 & 0 & 0\\ 
				\end{array}\right],\;\;\; \begin{array}{l}
				\text{for some }a,b\in\mathbb{C},\\
				|a|+|b|\neq 0
				\end{array}
				$$
				Observe that $\text{Ker}(\tau_{a,b}) = \overleftrightarrow{e_1,\left[0:-b:a\right]}$ and $\text{Im}(\tau_{a,b}) =\{e_1\}$. For each quasi-projective limit $\tau_{a,b}\in\text{QP}\left(3,\mathbb{C}\right)$, we consider the point $\left[0:-b:a\right]\in\overleftrightarrow{e_2,e_3}$ and the horocycle determined by this point and $e_2$, then we consider the pencil of lines passing through $e_1$ and each point of the horocycle. Denote by $\Omega$ the complement of this pencils of lines in $\mathbb{CP}^2$. Each connected component of $\Omega$ is simply connected. As a consequence of Proposition $\ref{prop_convergencia_qp}$, the action of $\Gamma$ in $\Omega$ is proper and discontinuous. Theorem \ref{thm_obdim_2} concludes this case.\\
	 
			\item $\gamma_{33}^{(k)}\rightarrow \infty$. From (\ref{eq_thm_descomposicion_noconmutativo_conv_forma_limqp}), it follows that $\gamma_{22}^{(k)}\rightarrow\infty$. Given that $\gamma_{11}^{(k)}\gamma_{22}^{(k)}\gamma_{33}^{(k)}=1$ for all $k\in\mathbb{N}$, then $\gamma_{11}^{(k)}\rightarrow 0$ and $\gamma_{11}^{(k)}\left(\gamma_{33}^{(k)}\right)^{-1}\longrightarrow 0$. Therefore
				\begin{equation}\label{eq_thm_descomposicion_noconmutativo_iii_conv_b_2}
				\gamma_k\longrightarrow\left[\begin{array}{ccc}
				0 & \gamma_{12} & \gamma_{13} \\
				0 & \gamma_{22} & \gamma_{23}\\
				0 & 0 & 1\\ 
				\end{array}\right],
				\text{ where }
				\begin{array}{l}
				\gamma_{12}^{(k)}\left(\gamma_{33}^{(k)}\right)^{-1}\rightarrow\gamma_{12}\in\mathbb{C}\\
				\gamma_{13}^{(k)}\left(\gamma_{33}^{(k)}\right)^{-1}\rightarrow\gamma_{13}\in\mathbb{C}
				\end{array}
				\end{equation}
				If $\gamma_{12}^{(k)}\left(\gamma_{33}^{(k)}\right)^{-1}\rightarrow\infty$ or $\gamma_{13}^{(k)}\left(\gamma_{33}^{(k)}\right)\rightarrow\infty$, we consider the sequence $\{\gamma_k^{-1}\}$ instead. Considering a conjugation of $\tau$ by an adequate element of $\text{PSL}\left(3,\mathbb{C}\right)$, we can assume that $\gamma_{12}=\gamma_{13}=0$. Using Lemma \ref{lem_control_nodiscreto_descartar_grL_vacio}, we have the following possibilities for the Greenberg limit set of $\Sigma$:
				
				\begin{multicols}{2}
			    \begin{itemize}
        		\item[{\scriptsize(\textbf{LS1})}] $\Lambda_{Gr}(\Sigma)=\mathbb{S}^1$.
				\item[{\scriptsize(\textbf{LS2})}] $\Lambda_{Gr}(\Sigma)=\mathbb{CP}^1$. 
				\item[{\scriptsize(\textbf{LS3})}] $|\Lambda_{Gr}(\Sigma)|=1$.
				\item[{\scriptsize(\textbf{LS4})}] $|\Lambda_{Gr}(\Sigma)|=2$.
    			\end{itemize}
    			\end{multicols}	
    			
    			Now we prove that neither of this possibilities can happen. In the case {\footnotesize(\textbf{LS1})} there are four possibilities for the convergence of $\{\Pi(\gamma_k)\}$:
				\begin{itemize}
				\item[{\scriptsize(\textbf{LS1.1})}] The sequence $\{\Pi(\gamma_k)\}$ converges to $\mathit{id}\in \Sigma$.
				\item[{\scriptsize(\textbf{LS1.2})}] The sequence $\{\Pi(\gamma_k)\}$ converges to a elliptic element in $\Sigma$.
				\item[{\scriptsize(\textbf{LS1.3})}] The sequence $\{\Pi(\gamma_k)\}$ converges to a parabolic element in $\Sigma$.
				\item[{\scriptsize(\textbf{LS1.4})}] The sequence $\{\Pi(\gamma_k)\}$ converges to a loxodromic element in $\Sigma$.\\
				\end{itemize}					
				
				In Case {\footnotesize(\textbf{LS1.1})} we have $\gamma_{22}=1$ and $\gamma_{23}=0$, denote the quasi-projective limit of $\{\gamma_k\}$ by $\tau_{id}=\text{Diag}(0, 1, 1)$. Since $\Lambda_{Gr}(\Sigma)=\mathbb{S}^1$, we can take $g=\left[g_{ij}\right]\in\Gamma$ such that $\Pi(g)$ is parabolic (see the proof of Proposition \ref{prop_HC_conmutativo_control_no_discreto}). Then $g_{22}=g_{33}=\lambda\in\mathbb{C}^\ast$, $g_{23}\neq 0$ and $g_{11}=\lambda^{-2}$. A straightforward calculation shows that
				$$\gamma_k g\gamma_k^{-1}\rightarrow \left[\begin{array}{ccc}
				\lambda^{-2} & 0 & 0 \\
				0 & \lambda & g_{23}\\
				0 & 0 & \lambda\\ 
				\end{array}\right]\in\text{PSL}\left(3,\mathbb{C}\right),$$
				contradicting that $\Gamma$ is discrete, unless the sequence $\{\gamma_k g\gamma_k^{-1}\}$ is eventually constant. Therefore we can assume that this sequence is constant, yielding $\gamma_{22}^{(k)}=\gamma_{33}^{(k)}=\xi_k\in\mathbb{C}^\ast$. Corollary \ref{cor_diagonal_es_1} implies that $\xi_k=1$ for all $k$, contradicting that $\gamma_k\rightarrow\tau_{id}$. Then sub-case {\footnotesize(\textbf{LS1.1})} cannot happen.\\
				
				Case {\footnotesize(\textbf{LS1.2})} cannot occur. Otherwise, $\Pi(\gamma_k)$ would converge to an elliptic element $\text{Diag}\left(e^{2\pi i \theta},1\right)\in\Sigma$. After an appropriate conjugation, we can assume that
				\begin{equation}\label{eq_thm_descomposicion_caso_noconmutativo_indices_6}
				\gamma_k\rightarrow\tau_\theta=\left[\begin{array}{ccc}
				0 & 0 & 0 \\
				0 & e^{2\pi i \theta} & 0\\
				0 & 0 & 1\\ 
				\end{array}\right].				
				\end{equation}
				A straightforward calculation shows that $\gamma_k^n\rightarrow\tau_\theta^n$ for any $n\in\mathbb{N}$. If $\theta\in\mathbb{Q}$, then there exist $p\in\mathbb{Z}$ such that $\gamma_k^p\rightarrow\tau_{\theta}^p=\tau_{\mathit{id}}$; this cannot happen, as we have already proven. If $\theta\in\mathbb{R}\setminus\mathbb{Q}$, there is a subsequence $\{e^{2\pi i n_j \theta}\}$ such that $e^{2\pi i n_j \theta}\rightarrow 1$. Then $\tau_\theta^{n_j}\rightarrow \tau_\mathit{id}$ as $j\rightarrow\infty$. Consider the diagonal sequence $\{\gamma_k^{n_k}\}\subset\Gamma$, then $\gamma_k^{n_k}\rightarrow \tau_\mathit{id}$ which cannot happen.\\
				
				Now we dismiss Case {\footnotesize(\textbf{LS1.3})}. Assuming it occurs, 
				\begin{equation}\label{eq_thm_descomposicion_caso_noconmutativo_indices_10}
				\gamma_k\rightarrow\tau_b = \left[\begin{array}{ccc}
				0 & 0 & 0\\
				0 & 1 & b\\
				0 & 0 & 1				
				\end{array}\right].
				\end{equation}
				Since the sequence $\{\Pi(\gamma_k)\}$ converges to a parabolic element and $\Sigma$ is not discrete, it follows from Lemma \ref{lem_parte_parab_no_discreta} that there exists a sequence $\{h_k\}\subset\Gamma$ such that $\Pi(h_k)$ is parabolic for all $k$, and 
				$$\Pi(h_k)=\left[\begin{array}{cc}
				1 & \varepsilon_k\\
				0 & 1\\
				\end{array}\right]\rightarrow \mathit{id},$$
				for some sequence $\varepsilon_k\rightarrow 0$. Then using the same reasoning as in the case {\footnotesize(\textbf{LS1.1})}, we know that 
				$$h_k=\left[\begin{array}{ccc}
				1 & h_{12} & h_{13}\\
				0 & 1 & \varepsilon_k\\
				0 & 0 & 1				
				\end{array}\right].$$
				A direct calculation shows that $\gamma_k h_k \gamma_k^{-1} h_k^{-1}\rightarrow g_{h_{12},0}\in\text{PSL}\left(3,\mathbb{C}\right)$ contradicting that $\Gamma$ is discrete. \\
				
				Finally, Case {\footnotesize(\textbf{LS1.4})} is dismissed in a similar way to Case {\footnotesize(\textbf{LS1.3})}. We proved that Case {\footnotesize(\textbf{LS1})} cannot occur. Case {\footnotesize(\textbf{LS2})} cannot happen either, otherwise the action of $\Gamma$ would be nowhere proper and discontinuous. Case {\footnotesize(\textbf{LS3})} canoot occur, otherwise $\Sigma$ is conjugate to a subgroup of $\text{Epa}\left(\mathbb{C}\right)$ (Theorem 2.14 of \cite{cs2014}). Let $h\in\Gamma$ be a loxo-parabolic element such that $\Pi(h)$ is parabolic. A straightforward computation shows that the sequence of distinct elements $\{\eta_k:=h\gamma_k h^{-1} \gamma_k^{-1}\}\subset \Gamma$ converges to an element of $\text{PSL}\left(3,\mathbb{C}\right)$, contradicting that $\Gamma$ is discrete. \\				
				
				Now, finally, consider the case {\footnotesize(\textbf{LS4})}. Then $\Sigma$ is conjugate to a subgroup of $\text{Aut}\left(\mathbb{C}^\ast\right)$ (Theorem 2.14 of \cite{cs2014}). We can take two loxodromic elements $g_1,g_2\in\Gamma$ such that $\left[g_1,g_2\right]\neq\mathit{id}$ and such that the sequence $\{h \gamma_k h^{-1}\gamma_k^{-1}\}\in\Gamma$ converges to an element of $\text{PSL}\left(3,\mathbb{C}\right)$. This would contradict that $\Gamma$ is discrete. \\			
					 
			\item $\gamma_{33}^{(k)}\rightarrow 0$. Considering the sequence $\{\gamma_k^{-1}\}$ instead, this case is the same as case (b).\\			
			\end{enumerate}
				
		\item[{\scriptsize \textbf{[Div]}}] The sequence $\{\Pi(\gamma_k)\}\subset\Sigma$ diverges. Since $\Gamma$ is not discrete, there are 7 possibilites: 
			\begin{itemize}
			\item $\overline{\Sigma}=\text{SO}(3)$ or $\overline{\Sigma}=\text{PSL}\left(2,\mathbb{C}\right)$. These cases cannot happen since $\Sigma$ is solvable, but $\text{SO}(3)$ and $\text{PSL}\left(2,\mathbb{C}\right)$ are not solvable.
			\item $\overline{\Sigma}=\text{Rot}_\infty$ or $\overline{\Sigma}=\text{Dih}_\infty$. These cases cannot happen either, otherwise $\Lambda_{\text{Gr}}(\Sigma)=\emptyset$, contradicting Lemma \ref{lem_control_nodiscreto_descartar_grL_vacio}.
			\item[{\scriptsize \textbf{[Div-1]}}] The group $\overline{\Sigma}$ is a subgroup of the affine group $\text{Epa}\left(\mathbb{C}\right)$. There cannot be elliptic elements in $\Sigma$, otherwise $\lambda_{23}(\Gamma)$ would have a torsion element or $\Gamma$ would contain an irrational screw and therefore, it would be commutative (Corollary \ref{cor_IS_conmutativo}). Then every element of $\Sigma$ has the form
				$$\Pi(\gamma)=\left[\begin{array}{cc}
				1 & a\\
				0 & 1\\
				\end{array}\right],$$
			where $a\in A$, for some additive group $A\subset(\mathbb{C},+)$. Since $\Sigma$ is not discrete, $A$ is not discrete. Then either $\overline{A}\cong \mathbb{R}$ or $\overline{A}\cong \mathbb{R}\oplus\mathbb{Z}$ (Proposition \ref{prop_forma_subgrupos_aditivos}). Let us define the following union of pencils of lines passing through $e_1$,		
					$$\Lambda=\{e_1\}\cup\left(\bigcup_{p\in\overline{A}}\pi(p)^{-1}\right),$$
				and $\Omega=\mathbb{CP}^2\setminus \Lambda$. Using similar arguments as before, the action of $\Gamma$ on each connected component of $\Omega$ is proper and discontinuous. Observe that each connected component of $\Omega$ is simply connected.\\
				
			\item[{\scriptsize \textbf{[Div-2]}}] The group $\overline{\Sigma}$ is a subgroup of the group $\text{Aut}\left(\mathbb{C}^\ast\right)$. Then $\Sigma$ is a purely loxodromic group and then, up to conjugation, each element has the form $\Pi(\gamma)=\text{Diag}\left(\alpha, 1\right)$ for some $|\alpha|\neq 1$. Let $G=\lambda_{23}(\Gamma)$, then $G\cong\Sigma$, hence, $G$ is not discrete. Let us write each $\alpha\in G$ as $\alpha = re^{i\theta}$, $r\in A$ and $\theta\in B$, for some multiplicative group $A$ and some additive group $B$. Then $G\cong A\times B$, and since $G$ is not discrete, then either $A$ is not discrete or $B$ is not discrete (it is not possible that both $A$ and $B$ are not discrete, otherwise the action of $\Gamma$ would be nowhere proper and discontinuous). Let us examine each of the two possibilities:
			\begin{itemize}
			\item $A$ is discrete, and $B$ is not discrete. This case cannot happen, otherwise, since $A$ is a discrete, then $A=\left\{r^n\,\middle|\,n\in\mathbb{Z}\right\}$ for some $r\in\mathbb{C}^\ast$. Hence, $\mathbb{S}^1\subset\overline{G}$. This means that there exists an element $\gamma\in\Gamma$ such that $\lambda_{23}(\gamma)=e^{i\theta}$ for some $\theta\in\left[0,2\pi\right)$. Then either $\lambda_{23}(\Gamma)$ contains a non-trivial torsion element or $\Gamma$ contains an irrational screw (contradicting Corollary \ref{cor_IS_conmutativo}). 
			\item $A$ is not discrete, and $B$ is discrete. Since $B$ is discrete, then $B$ is finite. Then $B\cong\mathbb{Z}_k$ for some $k\in\mathbb{N}$. Regarding $\overleftrightarrow{e_2,e_3}$ homeomorphic to $\hat{\mathbb{C}}$ and for the sake of simplifying notation, we will indistinctly denote a point $x\in\overline{G}\subset\mathbb{C}$ and its corresponding image in $ \overleftrightarrow{e_2,e_3}$. Let $\Lambda\subset\mathbb{CP}^2$ be given by
					$$\Lambda=\{e_1\}\cup\left(\bigcup_{p\in\overline{G}}\pi(p)^{-1}\right),$$
				and $\Omega=\mathbb{CP}^2\setminus\Lambda$. Using a similar argument as in Case {\small \textbf{[con]}}, we deduce that the action of $\Gamma$ on $\Omega$ is proper and discontinuous. Besides, each connected component of $\Omega$ is simply connected.\\
			\end{itemize}			 
			\item[{\scriptsize \textbf{[Div-3]}}] The group $\overline{\Sigma}$ is a subgroup of $\text{PSL}\left(2,\mathbb{R}\right)$. Then $\Lambda_{\text{Gr}}(\Sigma)\cong\hat{\mathbb{R}}$ (Theorem 2.14 of \cite{cs2014}). Therefore, up to conjugation, the orbits of compact subsets of $\mathbb{CP}^1\setminus \Lambda_{\text{Gr}}(\Sigma)$ accumulate on $\hat{\mathbb{R}}$. Regarding the points $\pi(e_2)$ and $\pi(e_3)$ as the points $0$ and $\infty$ in this euclidian circle. We define the pencil of lines passing through $e_1$,
				$$\Lambda=\{e_1\}\cup\left(\bigcup_{p\in\hat{\mathbb{R}}}\pi(p)^{-1}\right).$$ 
			Again, the action of $\Gamma$ on each of these connected components of $\Omega=\mathbb{CP}^2\setminus\Lambda$ is proper and discontinuous. Each of the connected components of $\Omega$ are simply connected.				
			\end{itemize}
		\end{itemize}
	\end{enumerate}
\end{proof}

\subsection{Consequences of the theorem of decomposition of non-commutative groups}
\label{subsec_consequences}

In this subsection we prove some consequences of Theorems \ref{thm_descomposicion_caso_noconmutativo} and \ref{thm_descomposicion_caso_noconmutativo2}. In Corollary \ref{cor_descomposicion_bloques_noconm}, we give a simplified decomposition of a group in terms of copies of $\mathbb{Z}^d$. Theorem \ref{thm_descomposicion_caso_noconmutativo} gives a decomposition of the group $\Gamma$ in four layers. In Corollary \ref{cor_lox_nivel34}, we prove that the first two layers are made of parabolic elements and the last two layers are made of loxodromic elements. The description of these four layers are summarized in Table \ref{fig_noconmutativo_capas}. Before proving Corollaries \ref{cor_descomposicion_bloques_noconm} and \ref{cor_lox_nivel34}, we need to prove some technical results.

\begin{table}
\begin{center}
  \begin{tabular}{  c  c  c  c  }
    
    \multicolumn{2}{c}{Parabolic} & \multicolumn{2}{c}{Loxodromic}\\
    \multicolumn{2}{c}{$\overbrace{\hspace{25ex}}$} & \multicolumn{2}{c}{$\overbrace{\hspace{45ex}}$}\\
    $\left[\begin{array}{ccc}
	1 & x & y\\
	0 & 1 & 0 \\
	0 & 0 & 1\\
	\end{array}\right]$ & $\left[\begin{array}{ccc}
	1 & x & y\\
	0 & 1 & z\\
	0 & 0 & 1\\
	\end{array}\right]$ & $\left[\begin{array}{ccc}
	\alpha & x & y\\
	0 & \beta & z\\
	0 & 0 & \beta\\
	\end{array}\right]$ & $\left[\begin{array}{ccc}
	\alpha & x & y\\
	0 & \beta & z\\
	0 & 0 & \gamma\\
	\end{array}\right]$\\ 
     & $z\neq 0$ & $\alpha\neq\beta$, $z\neq 0$ & $\beta\neq\gamma$     \\
     & & \begin{tabular}{c}
     \small{Loxo-parabolic} \\
     \end{tabular} & \begin{tabular}{c}
     \small{Loxo-parabolic} \\
     \small{Complex homothety} \\
     \small{Strongly loxodromic} \\
     \end{tabular} \\  
    $\text{Core}(\Gamma)$ & $A\setminus\text{Ker}(\Gamma)$ & $\text{Ker}(\lambda_{23})\setminus A$ & \\
    
  \end{tabular}
\end{center}
\caption{\small The decomposition of a non-commutative subgroup of $U_+$ in four layers.}
\label{fig_noconmutativo_capas}
\end{table}

Let $F_1,F_2,F_3\subset U_+$ be the pairwise disjoint subsets defined as
	\begin{align*}
	F_1&=\left\{\alpha=\left[\begin{array}{ccc}
	\alpha_{11} & 0 & \alpha_{13}\\
	0 & \alpha_{22} & 0\\
	0 & 0 & \alpha_{33}
	\end{array}\right]\,\middle|\,\alpha\in U_+\right\} \\
	F_2&=\left\{\alpha=\left[\begin{array}{ccc}
	\alpha_{11} & \alpha_{12} & \alpha_{13}\\
	0 & \alpha_{22} & 0\\
	0 & 0 & \alpha_{33}
	\end{array}\right]\,\middle|\,\alpha_{12}\neq 0\text{, }\alpha\in U_+\right\} \\
	F_3&=\left\{\alpha=\left[\begin{array}{ccc}
	\alpha_{11} & 0 & \alpha_{13}\\
	0 & \alpha_{22} & \alpha_{23}\\
	0 & 0 & \alpha_{33}
	\end{array}\right]\,\middle|\,\alpha_{23}\neq 0\text{, }\alpha\in U_+\right\}	\\
	F_4&=\left\{\alpha=\left[\begin{array}{ccc}
	\alpha_{11} & \alpha_{12} & \alpha_{13}\\
	0 & \alpha_{22} & \alpha_{23}\\
	0 & 0 & \alpha_{33}
	\end{array}\right]\,\middle|\,\alpha_{12},\alpha_{23}\neq 0,\alpha\in U_+\right\}	
	\end{align*}	

These four subsets classify the elements of $U_+$ depending on whether they have zeroes in entries (1,2) and (2,3). We need this classification because, as we will see in Proposition \ref{prop_formas_conmutativas}, a necessary condition for two elements of $U_+$ to commute is that they both have the same \emph{form} given by the sets $F_1,F_2,F_3$. This is equivalent to say that the two $2\times 2$ diagonal sub-blocks of one element share the same fixed points with the corresponding sub-block of the other element. This argument will be key to prove Corollary \ref{cor_descomposicion_bloques_noconm}.

\begin{prop}\label{prop_formas_conmutativas}
Let $\Gamma\subset U_+$ be a subgroup and let $\alpha=\left[\alpha_{ij}\right],\beta=\left[\beta_{ij}\right]\in\Gamma\setminus\{\mathit{id}\}$. If $\left[\alpha,\beta\right]=\mathit{id}$ then $\alpha,\beta\in F_i$ for some $i=1,2,3,4$.
\end{prop}

\begin{proof}
Let $\alpha=\left[\alpha_{ij}\right]$ and $\beta=\left[\beta_{ij}\right]$ be two elements in $\Gamma$ and suppose that they commute. As in Lemma \ref{lem_bloques_conmutativos}, let us denote by $\alpha_1$ and $\alpha_2$ (resp. $\beta_1$ and $\beta_2$) the upper left and bottom right $2\times 2$ blocks of $\alpha$ (resp. $\beta$). Since $\alpha$ and $\beta$ commute, a direct calculation shows that $\alpha_i$ and $\beta_i$ commute, for $i=1,2$. If we consider $\alpha_i$ and $\beta_i$ as elements of $\text{PSL}\left(2,\mathbb{C}\right)$, Lemma \ref{lem_bloques_conmutativos} states that 
	$$
	\text{Fix}(\alpha_1)=\{e_1,\left[\alpha_{12}:\alpha_{22}-\alpha_{11}\right]\},\;\;\;\text{Fix}(\alpha_2)=\{e_1,\left[\alpha_{23}:\alpha_{33}-\alpha_{22}\right]\}.
	$$
Similar expressions hold for $\text{Fix}(\beta_i)$. Also, $\left[\alpha_1,\beta_1\right]=\mathit{id} \Leftrightarrow \text{Fix}(\alpha_1)=\text{Fix}(\beta_1)$. This allows us to conclude that, if $\alpha$ and $\beta$ commute then $\alpha_{12}$ and $\beta_{12}$ are zero or non-zero simultaneously, the same holds for $\alpha_{23}$ and $\beta_{23}$. This proves the proposition.
\end{proof}

\begin{prop}\label{prop_transitividad_conmutatividad}
Let $\alpha,\beta\in F_i$, $i=1,2,3$. Then the relation defined by $\alpha\sim\beta$ if and only if $\left[\alpha,\beta\right]=\mathit{id}$, is an equivalence relation on each subset $F_i$, $i=1,2,3$. 
\end{prop}

\begin{proof}
It is straightforward to verify that the relation is reflexive and symmetric. We now verify that it is also transitive, let $\alpha,\beta,\gamma\in F_i$ for some $i=1,2,3$ such that $\left[\alpha,\beta\right]=\mathit{id}$ and $\left[\beta,\gamma\right]=\mathit{id}$. Denote $\alpha=\left[\alpha_{ij}\right]$, $\beta=\left[\beta_{ij}\right]$ and $\gamma=\left[\gamma_{ij}\right]$. Then $\left[\alpha,\beta\right]=\mathit{id}$ if and only if 
	\begin{equation}\label{eq_prop_transitividad_conmutatividad_1}	
	\beta_{13}(\alpha_{11}-\alpha_{33})+\alpha_{13}(\beta_{11}-\beta_{33})=\alpha_{23}\beta_{12}-\alpha_{12}\beta_{23}.
	\end{equation}
If both $\alpha,\beta\in F_i$, $i=1,2,3$, then $\alpha_{12}=\beta_{12}=0$ or $\alpha_{23}=\beta_{23}=0$, which implies that the right side of (\ref{eq_prop_transitividad_conmutatividad_1}) is zero and then $\beta_{13}(\alpha_{11}-\alpha_{33})+\alpha_{13}(\beta_{11}-\beta_{33})=0$. From this, analogously to the proof of Proposition \ref{prop_formas_conmutativas}, it follows that $\text{Fix}(\alpha)=\text{Fix}(\beta)$. Analogously, $\left[\beta,\gamma\right]=\mathit{id}$ implies that $\text{Fix}(\beta)=\text{Fix}(\gamma)$, then $\text{Fix}(\alpha)=\text{Fix}(\gamma)$ and this implies that  $\left[\alpha,\gamma\right]=\mathit{id}$. This proves that the relation is transitive and thus, it is an equivalence relation on $F_i$, $i=1,2,3$.
\end{proof}

In \cite{tesisvanessa}, the author proves a version of Corollary \ref{cor_descomposicion_bloques_noconm} for groups with real entries. In \cite{ppar}, a version of this corollary for purely parabolic groups is proven. Corollary \ref{cor_descomposicion_bloques_noconm} is a generalization of these previous results.\\

Propositions \ref{prop_formas_conmutativas} and \ref{prop_transitividad_conmutatividad} state that, unlike the purely parabolic case (see Lemma 7.9 of \cite{ppar}), commutativity does not define an equivalence relation on $U_+$. This equivalence relation occurs separately on $F_1$, $F_2$ and $F_3$. On $F_4$, commutativity does not define an equivalence relation.\\

Corollary \ref{cor_descomposicion_bloques_noconm} simplifies the decomposition described in Theorem \ref{thm_descomposicion_caso_noconmutativo}. We will need this proposition to prove the corollary.


\begin{cor}\label{cor_descomposicion_bloques_noconm}
Under the same hypothesis and notation of Theorem \ref{thm_descomposicion_caso_noconmutativo}, the group $\Gamma$ can be written as
	$$\Gamma\cong\mathbb{Z}^{r_0}\rtimes ...\rtimes \mathbb{Z}^{r_m},$$
for integers $r_0,...,r_m\geq 1$ satisfying $r_0+\cdots+r_m\leq 4$.
\end{cor}

\begin{proof}
Using the notation of Theorem \ref{thm_descomposicion_caso_noconmutativo}, we know that the group $A=\text{Core}(\Gamma)\rtimes\langle \xi_1\rangle\rtimes...\rtimes \langle\xi_r\rangle$ is purely parabolic and therefore, by Theorem 7.11 of \cite{ppar} we can write
	\begin{equation}\label{eq_dem_cor_descomposicion_bloques_noconm_1}
	A\cong \mathbb{Z}^{k_0}\rtimes ...\rtimes \mathbb{Z}^{k_{n_1}}.
	\end{equation}
For some integers $k_0,...,k_{n_1}$ such that $k_0+...+k_{n_1}\leq 4$. We denote 
	\begin{equation}\label{eq_dem_cor_descomposicion_bloques_noconm_5}
	r_i=k_i,\;\;\;\text{for }i=0,...,n_1.
	\end{equation}

Let us re-order the elements $\{\eta_1,...,\eta_m\}$ in the third layer in such way that if $i<j$ , then $\eta_i\in F_{s_1}$, $\eta_j\in F_{s_2}$ with $s_1\leq s_2$. We re-order the elements $\{\gamma_1,...,\gamma_n\}$ in the same way. Define the relation in 
	$$\Gamma_1 := \langle\eta_1,...,\eta_m\rangle\cap\left(F_1\cup F_2\cup F_3\right)$$
given by $\alpha\sim \beta$ if and only if $\left[\alpha,\beta\right]=\mathit{id}$. This is an equivalence relation (Proposition \ref{prop_transitividad_conmutatividad}). Denote by $A_1,...,A_{n_2}$ the equivalence classes in $\Gamma_1$. Let $B_i=\langle A_i\rangle$, clearly $B_i$ is a commutative and torsion-free group. Denoting $p_i=\text{rank}(B_i)$, we have $B_i\cong \mathbb{Z}^{p_i}$. Then 
	\begin{equation}\label{eq_dem_cor_descomposicion_bloques_noconm_2}
	\Gamma_1\cong \mathbb{Z}^{p_1}\rtimes ...\rtimes \mathbb{Z}^{p_{n_2}}.
	\end{equation}
Denote by $\eta_{\tilde{p}_i}$ the remaining elements of the third layer. That is, 
	$$\langle\eta_1,...,\eta_m\rangle\cap F_4=\{\eta_{\tilde{p}_1},...,\eta_{\tilde{p}_{n_3}}\}.$$
Then it follows from (\ref{eq_dem_cor_descomposicion_bloques_noconm_2})
	\begin{equation}\label{eq_dem_cor_descomposicion_bloques_noconm_3}
	\langle \eta_1,...,\eta_m\rangle\cong \mathbb{Z}^{p_1}\rtimes ...\rtimes \mathbb{Z}^{p_{n_2}}\rtimes\langle\eta_{\tilde{p}_1}\rangle\rtimes...\rtimes \langle\eta_{\tilde{p}_{n_3}}\rangle.
	\end{equation}
Let us denote
	\begin{align}
	r_{n_1+i}=p_i &, \;\;\;\text{for }i=1,...,n_2. \label{eq_dem_cor_descomposicion_bloques_noconm_6} \\
	r_{n_1+n_2+i}= 1 &,\;\;\;\text{for }i=1,...,n_3. \nonumber
	\end{align}
	
Applying the same argument to the elements of the fourth layer $\{\gamma_1,...,\gamma_n\}$ we have  
	\begin{equation}\label{eq_dem_cor_descomposicion_bloques_noconm_4}
	\langle\gamma_1,...,\gamma_n\rangle\cong \mathbb{Z}^{q_1}\rtimes ...\rtimes \mathbb{Z}^{q_{n_4}}\rtimes\langle\gamma_{\tilde{q}_1}\rangle\rtimes...\rtimes \langle\gamma_{\tilde{q}_{n_5}}\rangle.
	\end{equation}
Again, we denote
\begin{align}
	r_{n_1+n_2+n_3+i}=q_i &, \;\;\;\text{for }i=1,...,n_4. \label{eq_dem_cor_descomposicion_bloques_noconm_7} \\
	r_{n_1+n_2+n_3+n_4+i}= 1 &,\;\;\;\text{for }i=1,...,n_5. \nonumber
	\end{align}
Putting together (\ref{eq_dem_cor_descomposicion_bloques_noconm_1}), (\ref{eq_dem_cor_descomposicion_bloques_noconm_3}) and (\ref{eq_dem_cor_descomposicion_bloques_noconm_4}) we prove the corollary. The indices $r_0,...,r_m$ are given by (\ref{eq_dem_cor_descomposicion_bloques_noconm_5}), (\ref{eq_dem_cor_descomposicion_bloques_noconm_6}) and (\ref{eq_dem_cor_descomposicion_bloques_noconm_7}) and $m=n_1+...+n_5$.
\end{proof}

The following corollary describes the type of elements found in each layer of the decomposition.

\begin{cor}\label{cor_lox_nivel34}
Let $\Gamma\subset U_+$ be a non-commutative discrete subgroup. Consider the decomposition in four layers described in the proof of Theorem \ref{thm_descomposicion_caso_noconmutativo} and summarized in Table \ref{fig_noconmutativo_capas}.\\
The first two layers $\text{Core}(\Gamma)$ and $A\setminus\text{Core}(\Gamma)$ are purely parabolic and the last two layers $\text{Ker}(\lambda_{23})\setminus A$ and $\Gamma\setminus\text{Ker}(\lambda_{23})$ are made up entirely of loxodromic elements. Furthermore,
	\begin{enumerate}[(i)]
	\item The third layer $\text{Ker}(\lambda_{23})\setminus A$ contains only loxo-parabolic elements.
	\item The fourth layer $\Gamma\setminus\text{Ker}(\lambda_{23})$ contains only loxo-parabolic and strongly loxodromic elements or complex homotheties of the form $\text{Diag}(\lambda,\lambda^{-2},\lambda)$. 
	\end{enumerate}
\end{cor}

\begin{proof}
From the definition of $\text{Core}(\Gamma)$ and $A=\text{Ker}(\lambda_{12})\cap\text{Ker}(\lambda_{23})$ is clear that this two subgroups are purely parabolic. Now we deal with the third layer $\text{Ker}(\lambda_{23})\setminus A$. If there were a elliptic element $\gamma$ in this layer, we have two cases:
	\begin{itemize}
	\item If $\gamma$ has infinite order, then $\Gamma$ cannot be discrete. 
	\item If $\gamma$ has finite order $p>0$, then $\Gamma$ cannot be torsion-free.
	\end{itemize}

If there were a parabolic element $\gamma$ in this layer, then it must have exactly two repeated eigenvalues (if it had 3, then $\gamma\in A$). Furthermore, all of its eigenvalues must be unitary (they cannot be 1, otherwise $\gamma\in A$). Then
	$$\gamma=\left[\begin{array}{ccc}
	e^{-4\pi i \theta} & x & y\\
	0 & e^{2\pi i \theta} & z\\
	0 & 0 & e^{2\pi i \theta}\\
	\end{array}\right]$$
with $z\neq 0$ and $\theta\in\mathbb{R}\setminus\mathbb{Q}$ (since $\lambda_{12}(\Gamma)$ is a torsion-free group). This means that $\gamma$ is an irrational ellipto-parabolic element, and by Proposition \ref{prop_EPI_conmutativo}, $\Gamma$ would be commutative. All of these arguments prove that the third layer $\text{Ker}(\lambda_{23})\setminus A$ is purely loxodromic. Finally, since $\Gamma$ is not commutative, it cannot contain a complex homothety as a consequence of Corollary \ref{cor_HC_no_hay_en_no_conmutativos}.\\

Now we deal with the fourth layer $\Gamma\setminus\text{Ker}(\lambda_{23})$. Using the same argument as in the third layer, there cannot be elliptic elements. Now assume that there is a parabolic element $\gamma\in\Gamma\setminus\text{Ker}(\lambda_{23})$. In the same way as before, $\gamma$ must have exactly two distinct eigenvalues and none of them are equal to 1. Since $\gamma\not\in\text{Ker}(\lambda_{23})$, then
	$$\gamma=\left[\begin{array}{ccc}
	e^{2\pi i \theta} & x & y\\
	0 & e^{2\pi i \theta} & z\\
	0 & 0 & e^{-4\pi i \theta}\\
	\end{array}\right]$$
with $x\neq 0$ and $\theta\in\mathbb{R}\setminus\mathbb{Q}$. Then $\gamma$ is an irrational ellipto-parabolic element, and by Proposition \ref{prop_EPI_conmutativo}, $\Gamma$ is commutative. Then the fourth layer $\Gamma\setminus\text{Ker}(\lambda_{23})$ is purely loxodromic. Inspecting the form of these elements, they can be strongly loxodromic or complex homotheties of the form $\text{Diag}(\lambda,\lambda^{-2},\lambda)$ (Corollary \ref{cor_HC_no_hay_en_no_conmutativos}).
\end{proof}

\section{Commutative triangular groups}\label{sec_commutative_triangular}

In this section we describe the commutative triangular groups of $\text{PSL}\left(3,\mathbb{C}\right)$. We describe how these groups should look in order to be commutative and discrete. We also describe the Kulkarni limit sets and obtain a bound for their rank.\\ 

Recall the definitions of the group morphisms $\Pi$, $\lambda_{12}$, $\lambda_{13}$ and $\lambda_{23}$. Let us define the group morphism $\Pi^\ast:U_+\rightarrow\text{M\"ob}\left(\hat{\mathbb{C}}\right)$ given by
	$$\Pi^\ast\left(\left[\alpha_{ij}\right]\right)(z)=\alpha_{11}\alpha_{22}^{-1}z+\alpha_{12}\alpha_{22}^{-1}.$$
We will also define the projections $\pi_{kl}\left(\left[\alpha_{ij}\right]\right)=\alpha_{kl}$.\\

Whenever we have a discrete subgroup $\Gamma\subset U_+$, we have a finite index torsion-free subgroup $\Gamma'\subset\Gamma$ such that $\Pi^\ast(\Gamma)$, $\lambda_{12}(\Gamma')$, $\lambda_{13}(\Gamma)$ and $\lambda_{23}(\Gamma')$ are torsion-free groups as well (Lemma 5.8 of \cite{ppar}). This finite index subgroup and the original group satisfy $\Lambda_{\text{Kul}}(\Gamma)=\Lambda_{\text{Kul}}(\Gamma')$ (Proposition 3.6 of \cite{bcn16}). Therefore we can assume for the rest of this work that all discrete subgroups $\Gamma\subset U_+$ are torsion-free.\\

The following lemma describes the form of the upper triangular commutative subgroups of $\text{PSL}\left(3,\mathbb{C}\right)$. The proof is taken from \cite{ppar}.

\begin{lem}[Lemma 5.13 of \cite{ppar}]\label{lem_7casos_conmutativos}
Let $\Gamma\subset U_+$ be a commutative group, then $\Gamma$ is conjugate in $\text{PSL}\left(3,\mathbb{C}\right)$ to a subgroup of one of the following Abelian Lie Groups:
	\begin{multicols}{2}
	\begin{enumerate}
	\item
		$$C_1=\left\{\left(\begin{array}{ccc}
		\alpha^{-2} & 0 & 0 \\
		0 & \alpha & \beta \\
		0 & 0 & \alpha
		\end{array}\right)\,\middle|\,\alpha\in\mathbb{C}^\ast,\beta\in\mathbb{C}\right\}.$$		
	\item 
		$$C_2=\left\{\text{Diag}\left(\alpha,\beta,\alpha^{-1}\beta^{-1}\right)\,\middle|\,\alpha,\beta\in\mathbb{C}^\ast\right\}.$$
	\item
		$$C_3=\left\{\left(\begin{array}{ccc}
		1 & 0 & \beta \\
		0 & 1 & \gamma \\
		0 & 0 & 1
		\end{array}\right)\,\middle|\,\beta,\gamma\in\mathbb{C}\right\}.$$
	\item
		$$C_4=\left\{\left(\begin{array}{ccc}
		1 & \beta & \gamma \\
		0 & 1 & 0 \\
		0 & 0 & 1
		\end{array}\right)\,\middle|\,\beta,\gamma\in\mathbb{C}\right\}.$$
	\item
		$$C_5=\left\{\left(\begin{array}{ccc}
		1 & \beta & \gamma \\
		0 & 1 & \beta \\
		0 & 0 & 1
		\end{array}\right)\,\middle|\,\beta,\gamma\in\mathbb{C}\right\}.$$
	\end{enumerate}
	\end{multicols}
\end{lem} 

\begin{proof}
Since $\Gamma$ is commutative, $\Pi(\Gamma)$ and $\Pi^\ast(\Gamma)$ are commutative. Let us consider the following 4 possible cases:

\begin{enumerate}
\item Both groups $\Pi(\Gamma)$ and $\Pi^\ast(\Gamma)$ contain a parabolic element. Since they are Abelian groups, it is direct to verify that they are purely parabolic, and then $\Gamma\subset \text{Ker}(\lambda_{12})\cap\text{Ker}(\lambda_{13})$. Let $h\in \Gamma$ such that $\Pi(g)$ and $\Pi^\ast(g)$ are parabolic, then
	$$h=\left[
	\begin{array}{ccc}
	1 & a & b \\
	0 & 1 & c \\
	0 & 0 & 1 \\	
	\end{array}		
	\right],$$
where $ac\neq 0$. Let us define $h_0\in\text{PSL}\left(3,\mathbb{C}\right)$ by $h_0 = \text{Diag}(a^{-1},1,c)$, then for every $g=\left[g_{ij}\right]\in h_0 \Gamma h_0^{-1}$,
	$$\left[h_0 h h_0^{-1},g\right] = \left[
	\begin{array}{ccc}
	1 & 0 & -g_{12}+g_{23} \\
	0 & 1 & 0 \\
	0 & 0 & 1 \\	
	\end{array}		
	\right].$$
Since $\Gamma$ is commutative, it follows that $g_{12}=g_{13}$. This means that $\Gamma$ is conjugate, via $h_0$, to a subgroup of $C_5$.

\item The group $\Pi(\Gamma)$ does not contain a parabolic element, but $\Pi^\ast(\Gamma)$ does. Under this assumptions, we deduce that $\Pi^\ast(\Gamma)$ is purely parabolic, hence $\Gamma\subset\text{Ker}(\lambda_{12})$. Since $\Pi(\Gamma)$ does not contain parabolic elements, there exists $w\in\mathbb{C}$ such that $\Pi(\Gamma) w = w$. We define
	$$h=\left[
	\begin{array}{ccc}
	1 & 0 & 0 \\
	0 & 1 & w \\
	0 & 0 & 1 \\	
	\end{array}\right],$$
it is straightforward to verify that, for every $g=\left[g_{ij}\right]\in\Gamma$, there exists $c_g\in\mathbb{C}$ such that
	$$hgh^{-1}=\left[
	\begin{array}{ccc}
	g_{11} & g_{12} & c_g \\
	0 & g_{11} & 0 \\
	0 & 0 & g_{11}^{-2} \\	
	\end{array}\right].$$
Therefore $\Gamma_1=h\Gamma h^{-1}$ leaves $\overleftrightarrow{e_1,e_3}$ invariant. Then $\Pi_1:\Gamma_1\rightarrow\text{M\"ob}\left(\hat{\mathbb{C}}\right)$ given by $\Pi_1\left(\left[g_{ij}\right]\right)=g_{11}g_{33}^{-1}z+g_{13}g_{33}^{-1}$ is a well defined group morphism. Now, we have two sub-cases, depending on whether $\Pi_1(\Gamma_1)$ has a parabolic element or not.

\begin{itemize}
\item The group $\Pi_1(\Gamma_1)$ contains a parabolic element. Then $\Pi_1(\Gamma_1)$ is purely parabolic and then $\Gamma_1\subset\text{Ker}(\lambda_{13})$. From this, we deduce that $\Gamma$ is conjugate, via $h$, to the group $\Gamma_1$ of the form 
	$$hgh^{-1}=\left[
	\begin{array}{ccc}
	1 & g_{12} & c_g \\
	0 & 1 & 0 \\
	0 & 0 & 1 \\	
	\end{array}\right],$$
corresponding to a subgroup of $C_4$.

\item The group $\Pi_1(\Gamma_1)$ does not contains parabolic elements. Then there exists $p\in\mathbb{C}$ such that $\Pi_1(\Gamma_1)p=p$. We define
	$$h_1 = \left[
	\begin{array}{ccc}
	1 & 0 & p \\
	0 & 1 & 0 \\
	0 & 0 & 1 \\	
	\end{array}\right],$$
then for every $g=\left[g_{ij}\right]\in\Gamma_1$ it holds
	$$(h_1 h) g (h_1 h)^{-1} = \left[
	\begin{array}{ccc}
	g_{11} & g_{12} & 0 \\
	0 & g_{11} & 0 \\
	0 & 0 & g_{11}^{-2} \\	
	\end{array}\right].$$
This means that $\Gamma$ is conjugate to a subgroup of $C_1$.
\end{itemize}

\item The group $\Pi(\Gamma)$ contains a parabolic element but $\Pi^\ast(\Gamma)$ does not. This case is similar to the last case, but the roles of $\Pi(\Gamma)$ and $\Pi^\ast(\Gamma)$ are reversed. In this case, $\Gamma$ leaves $\overleftrightarrow{e_1,e_3}$ invariant again and therefore, we consider the group morphism $\Pi_2=\Pi_{e_2,\overleftrightarrow{e_1,e_3}}$. The group $\Pi_2(\Gamma)$ is Abelian and it might contain parabolic elements or not. In the former case, $\Gamma$ is conjugate to a subgroup of $C_3$; in the latter case, it is conjugate to a subgroup of $C_1$.

\item Neither $\Pi(\Gamma)$ nor $\Pi^\ast(\Gamma)$ contains a parabolic element. Therefore there are $z,w\in\mathbb{C}$ such that $\Pi^{\ast}(\Gamma) z = z$ and $\Pi^{\ast}(\Gamma) w = w$. Define 
	$$h = \left[
	\begin{array}{ccc}
	1 & z & 0 \\
	0 & 1 & w \\
	0 & 0 & 1 \\	
	\end{array}\right].$$
Then for every $g=\left[g_{ij}\right]\in\Gamma$ there exists $c_g\in\mathbb{C}$ such that
	$$hgh^{-1}=\left[
	\begin{array}{ccc}
	g_{11} & 0 & c_g \\
	0 & g_{22} & 0 \\
	0 & 0 & g_{33} \\	
	\end{array}\right].$$
Using the same arguments as in the previous cases and considering the group morphism $\Pi_2$ of Case (3) and the Abelian group $\Pi_2(\Gamma)$, we have the same two possibilities: either $\Pi_2(\Gamma)$ contains a parabolic element or not. In the former case, $\Gamma\subset\text{Ker}(\lambda_13)$ and therefore it is conjugate to a subgroup of $C_1$. In the latter case, there exists $p\in\mathbb{C}$ such that $\Pi_2(\Gamma) p = p$. Let 
	$$h_1 = \left[
	\begin{array}{ccc}
	1 & 0 & p \\
	0 & 1 & 0 \\
	0 & 0 & 1 \\	
	\end{array}\right].$$
Then the group $(h_1h)\Gamma(h_1 h)^{-1}$ contains only diagonal elements. Hence $\Gamma$ is conjugate to a subgroup of $C_2$.
\end{enumerate}
 
\end{proof}

Observe that $C_3$, $C_4$ and $C_5$ contain only parabolic elements, and their subgroups have already been studied in \cite{ppar}. Groups $C_1$ and $C_2$ can be purely parabolic or they can have loxodromic elements, therefore we will only study discrete subgroups of $C_1$ and $C_2$ containing loxodromic elements. We will call this two cases Cases 1 and 2 respectively. We study each case in each of the following two subsections.

\subsection{Case 1}

This first proposition describes the form of subgroups of $C_1$.

\begin{prop}\label{prop_conmutativo_c1_descripcion}
Let $\Gamma\subset U_+$ be a commutative subgroup such that each element of $\Gamma$ has the form given by (1) of Lemma \ref{lem_7casos_conmutativos}. Then there exists an additive subgroup $W\subset(\mathbb{C},+)$, and a group morphism $\mu:(W,+)\rightarrow(\mathbb{C}^{\ast},\cdot)$ such that
	$$
	\Gamma=\Gamma_{W,\mu}=\left\{\left[\begin{array}{ccc}
	\mu(w)^{-2} & 0 & 0 \\
	0 & \mu(w) & w\mu(w) \\
	0 & 0 & \mu(w)
	\end{array}\right]\,\middle|\,w\in W\right\}.	
	$$
\end{prop}

\begin{proof}
Let $\zeta:(\Gamma,\cdot)\rightarrow(\mathbb{C},+)$ be the group homomorphism given by $\left[\alpha_{ij}\right]\overset{\zeta}{\mapsto}\alpha_{23}\alpha^{-1}_{33}$, clearly $\text{Ker}(\zeta)=\{\mathit{id}\}$. Thus we can define the group homomorphism $\mu:(\zeta(\Gamma),+)\rightarrow(\mathbb{C}^{\ast},\cdot)$ as $x\overset{\mu}{\mapsto}\pi_{22}\left(\zeta^{-1}(x)\right)$. Define the additive group $W=\zeta\left(\Gamma\right)$. It is straightforward to verify that 
	$$
	\Gamma=\left\{\left[\begin{array}{ccc}
	\mu(w)^{-2} & 0 & 0 \\
	0 & \mu(w) & w\mu(w) \\
	0 & 0 & \mu(w)
	\end{array}\right]\,\middle|\,w\in W\right\}.	\qedhere
	$$
\end{proof}

For $w\in W$, we will denote 
	$$\gamma_w=\left[\begin{array}{ccc}
	\mu(w)^{-3} & 0 & 0 \\
	0 & 1 & w \\
	0 & 0 & 1
	\end{array}\right]\in \Gamma_{W,\mu}.$$

\begin{prop}\label{prop_conmutativo_caso1_rangoG_rangoW}
Let $\Gamma=\Gamma_{W,\mu}\subset U_+$ be a commutative subgroup of the form given by Proposition \ref{prop_conmutativo_c1_descripcion}. If $\text{rank}(\Gamma)=r$, then $\text{rank}(W)=r$.
\end{prop}

\begin{proof}
Let $r=\text{rank}(\Gamma)$ and $\gamma_1,...,\gamma_r\in\Gamma$ be elements such that $\Gamma=\langle\gamma_1,...,\gamma_r\rangle$. Let $w_1,...,w_r\in W$ such that $\gamma_j=\gamma_{w_j}$ for $j=1,...,r$. Let $w\in W$ and consider $\gamma_w\in \Gamma$, then there exist $n_1,...,n_r\in\mathbb{Z}$ such that $\gamma_w=\gamma_1^{n_1}\cdots \gamma_r^{n_r}$. Comparing entries (2,3) of both sides yields $w=n_1 w_1+...+n_r w_r$. This means that $W=\langle w_1,...,w_r\rangle$, and therefore $\text{rank}(W)\leq r$. To prove that $\text{rank}(W)= r$, assume without loss of generality that $w_r=n_1 w_1 +...+n_{r-1} w_{r-1}$. Then $\gamma_r=\gamma_1^{n_1}\cdots \gamma_r^{n_{r-1}}$, contradicting that $\text{rank}(\Gamma)=r$. This completes the proof.
\end{proof}

Now, we show examples of non-discrete additive subgroups $W$ such that $\Gamma_{W,\mu}$ is discrete. 

\begin{example}\label{lem_span_no_discreto}
If $\alpha$ and $\beta$ are two rationally independent real numbers, then $W=\text{Span}_{\mathbb{Z}}(\alpha,\beta)$ is a non-discrete additive subgroup of $\mathbb{C}$.
\end{example}

\begin{proof}
Let $h:(W,+)\rightarrow(\mathbb{S}^1,\cdot)$ be a group homomorphism given by $x\mapsto e^{2\pi i \frac{x}{\alpha}}$. Observe that $\{1,\frac{\beta}{\alpha}\}$ are rationally independent, then $h(W)$ is a sequence of distinct elements in $\mathbb{S}^1$. Therefore there is a subsequence, denoted by $\{g_n\}\subset h(W)$, such that $g_n=e^{2\pi i q_n \frac{\beta}{\alpha}}\rightarrow \xi \in \mathbb{S}^1$ for some $\{q_n\}\subset\mathbb{Z}$ and some $\xi \in \mathbb{S}^1$. Define the sequence $\{h_n\}\subset \mathbb{S}^1$ by $h_n=g_n g_{n+1}^{-1}$, then $h_n\rightarrow 1$. Denoting $h_n=e^{2\pi i r_n \frac{\beta}{\alpha}}$ for some $\{r_n\}\subset\mathbb{Z}$, and taking the logarithm of the sequence we have
	\begin{equation}\label{eq_lem_span_no_discreto_1}
	2\pi i r_n \frac{\beta}{\alpha} + 2\pi i s_n \rightarrow 0
	\end{equation}		
for some logarithm branch defined by $\{s_n\}\subset\mathbb{Z}$. As a consequence of (\ref{eq_lem_span_no_discreto_1}), $r_n \beta + s_n \alpha\rightarrow 0$ and therefore $W$ is not discrete.
\end{proof}

Proposition \ref{prop_conmutativo_c1_rangoW} will provide the full description of discrete commutative subgroups of $U_+$ of the case 1. In order to prove this proposition, we first need to determine the equicontinuity region for these groups (Proposition \ref{prop_eq_c1}). To do this, consider Table \ref{fig_c1_casos_qp}, in which we list all possible quasi-projective limits $\tau$ of sequences of distinct elements in these groups along with the condition under which they occur.

\begin{table}[H]
\begin{center}
  \begin{tabular}{ | l | c | c | c | c | }
    \hline
    Case & $\tau$ & Conditions & Ker($\tau$) & Im($\tau$) \\ \hline
    (i) & $\text{Diag}\left(1,0,0\right)$ & \begin{tabular}{c}
	\small $w_n\rightarrow b\in\mathbb{C}$ and $\mu(w_n)\rightarrow 0$ \\
	\small or \\
	\small $w_n\rightarrow \infty$, $\mu(w_n)\rightarrow 0$ and $w_n \mu(w_n)^3\rightarrow 0$\\
	\end{tabular}	
	 & $\overleftrightarrow{e_2,e_3}$ & $\{e_1\}$  \\ \hline
    (ii) & $\left[\begin{array}{ccc}
	0 & 0 & 0\\
	0 & 1 & b \\
	0 & 0 & 1\\
	\end{array}\right]$ & \small $w_n\rightarrow b\in\mathbb{C}$ and $\mu(w_n)\rightarrow \infty$ & $\{e_1\}$ & $\overleftrightarrow{e_2,e_3}$  \\ \hline
    (iii) & $\left[\begin{array}{ccc}
	0 & 0 & 0\\
	0 & 0 & 1 \\
	0 & 0 & 0\\
	\end{array}\right]$ & \begin{tabular}{c}
\small $w_n\rightarrow \infty$ and $\mu(w_n)\rightarrow \infty$ \\
\small or \\
\small $w_n\rightarrow \infty$ and $\mu(w_n)\rightarrow a\in\mathbb{C}^{\ast}$\\
\small or \\
\small $w_n\rightarrow \infty$, $\mu(w_n)\rightarrow 0$ and $w_n \mu(w_n)^3\rightarrow\infty$\\
	\end{tabular}
   & $\overleftrightarrow{e_1,e_2}$ & $\{e_2\}$  \\ \hline
    (iv) & $\left[\begin{array}{ccc}
	1 & 0 & 0\\
	0 & 0 & b \\
	0 & 0 & 0\\
	\end{array}\right]$ & \small $w_n\rightarrow \infty$, $\mu(w_n)\rightarrow 0$ and $w_n \mu(w_n)^3\rightarrow b\in\mathbb{C}^\ast$ & $\{e_2\}$ & $\overleftrightarrow{e_1,e_2}$ \\
    \hline
  \end{tabular}
\caption{Quasi-projective limits $\tau$ of sequences $\{w_n\}$ of distinct elements in $\Gamma$.}
\label{fig_c1_casos_qp}
\end{center}
\end{table}

\begin{prop}\label{prop_eq_c1}
Let $\Gamma\subset\text{PSL}\left(3,\mathbb{C}\right)$ be a commutative discrete group of the form given in Propositions \ref{prop_conmutativo_c1_descripcion}. If we assume that $\Gamma$ contains loxodromic elements, then $Eq(\Gamma)=\mathbb{CP}^{2}\setminus\left(\overleftrightarrow{e_1,e_2}\cup\overleftrightarrow{e_2,e_3}\right)$.
\end{prop}

\begin{proof}
We use Proposition \ref{prop_descripcion_Eq} to determine $\text{Eq}(\Gamma)$. Let $\gamma_w\in\Gamma$ be a loxodromic element, then $|\mu(w)|\neq 1$. Let us suppose, without loss of generality, that $|\mu(w)|>1$. Consider the sequence $\{\gamma_w^n\}_{n\in\mathbb{N}}\subset\Gamma$, then
	\begin{equation}\label{eq_prop_eq_c1_1}
	\gamma_w^n\rightarrow\tau_1=\left[\begin{array}{ccc}
	0 & 0 & 0 \\
	0 & 0 & 1 \\
	0 & 0 & 0 \\
	\end{array}\right],\;\;\;\text{ with Ker}(\tau_1)=\overleftrightarrow{e_1,e_2}.
	\end{equation}
Considering the sequence $\{\gamma_w^{-n}\}_{n\in\mathbb{N}}\subset\Gamma$ instead, we have $\gamma_w^{-n}\rightarrow\tau_2=\text{Diag}\left(1,0,0\right)$ with $\text{Ker}(\tau_2)=\overleftrightarrow{e_2,e_3}$. This, together with Proposition \ref{prop_descripcion_Eq} and (\ref{eq_prop_eq_c1_1}) imply that 
	\begin{equation}\label{eq_prop_eq_c1_3}
	\mathbb{CP}^{2}\setminus\left(\overleftrightarrow{e_1,e_2}\cup\overleftrightarrow{e_2,e_3}\right)\subset\text{Eq}(\Gamma).
	\end{equation} 
Proposition \ref{prop_descripcion_Eq} and Table \ref{fig_c1_casos_qp} imply that $\text{Eq}(\Gamma)\subset \mathbb{CP}^{2}\setminus\left(\overleftrightarrow{e_1,e_2}\cup\overleftrightarrow{e_2,e_3}\right)$. This together with (\ref{eq_prop_eq_c1_3}) prove the proposition.
\end{proof}

The following observation is important for the proof of Theorem \ref{thm_case1_kulkarni}.

\begin{obs}\label{obs_caso1_llenando_las_2_lineas_ya}
If $\Gamma\subset C_1$ is a complex Kleinian group and $\overleftrightarrow{e_1,e_2}\cup\overleftrightarrow{e_2,e_3}\subset \Lambda_{\text{Kul}}(\Gamma)$, then Propositions \ref{prop_eq_in_kuld} and \ref{prop_eq_c1} imply that 
	$$\Lambda_{\text{Kul}}(\Gamma)=\overleftrightarrow{e_1,e_2}\cup\overleftrightarrow{e_2,e_3}.$$
\end{obs}

The following proposition characterize discrete subgroups belonging to this first case.

\begin{prop}\label{prop_conmutativo_c1_rangoW}
Let $\Gamma=\Gamma_{W,\mu}\subset U_+$ be a group as described in Proposition \ref{prop_conmutativo_c1_descripcion}. $\Gamma$ is discrete if and only if $\text{rank}(W)\leq 3$ and the morphism $\mu$ satisfies the following condition:
\begin{enumerate}[(C)]
\item Whenever we have a sequence $\{w_k\}\in W$ of distinct elements such that $w_k\rightarrow 0$, either $\mu(w_k)\rightarrow 0$ or $\mu(w_k)\rightarrow \infty$. 
\end{enumerate}
\end{prop} 

\begin{proof}
First, assume that $\text{rank}(W)\leq 3$ and that the group morphism $\mu$ satisfies condition (C). If $W$ is discrete, then it is straightforward to see that $\Gamma$ is discrete. Now, assume that $W$ is not discrete. Let $\{\gamma_k\}\subset\Gamma$ be a sequence of distinct elements such that $\gamma_k\rightarrow\mathit{id}$, denote $\gamma_k=\gamma_{w_k}$. Then $\mu(w_k)$ converges to some cubic root of the unity and $w_k\rightarrow 0$. Since $\mu$ satisfies condition (C), then $\mu(w_k)\rightarrow 0$ or $\mu(w_k)\rightarrow \infty$ contradicting that $\mu(w_k)$ converges to some cubic root of the unity. This contradiction proves that $\Gamma$ is discrete.\\

Now assume that $\Gamma$ is discrete. Propositions \ref{prop_eq_in_kuld} and \ref{prop_eq_c1} imply that $\Gamma$ acts properly and discontinuously on $\text{Eq}(\Gamma)=\mathbb{CP}^2\setminus\left(\overleftrightarrow{e_1,e_2}\cup \overleftrightarrow{e_2,e_3}\right)\cong \mathbb{C}\times\mathbb{C}^{\ast}$. Consider the universal covering $\pi=(\mathit{id},\text{exp}):\mathbb{C}\times\mathbb{C}\rightarrow\mathbb{C}\times\mathbb{C}^\ast$, where $\text{exp}(z)=e^z$ for $z\in\mathbb{C}$. The group $\Gamma$ can be written as $\Gamma\cong \Gamma_1\times\Gamma_2$ with
	$$
	\Gamma_1 = \left\{\left[\begin{array}{cc}
	\mu(w)^{-3} & 0\\
	0 & 1 \end{array}\right]\,\middle|\,w\in W\right\},\;\;\;
	\Gamma_2 = \left\{\left[\begin{array}{cc}
	1 & w\\
	0 & 1 \end{array}\right]\,\middle|\,w\in W\right\}
	$$
with the multiplicative group $\Gamma_1$ acting on $\mathbb{C}^\ast$, and the additive group $\Gamma_2$ acting on $\mathbb{C}$. Let $\tilde{\Gamma}$, $\tilde{\Gamma}_2$ be covering groups of $\Gamma$ and $\Gamma_2$ respectively (Theorem 9.1 of \cite{bredon}). Let $\tilde{\Gamma}_1=\Gamma_1$. There is a group morphism, induced by $\pi$ and still denoted by $\pi$, given by
	\begin{align*}
	\pi=(\mathit{id},\text{exp}):\tilde{\Gamma}\cong\tilde{\Gamma}_1\times\tilde{\Gamma}_2 &\rightarrow \Gamma_1\times\Gamma_2\\
	(\alpha,\beta) &\mapsto (\alpha,e^\beta) 
	\end{align*}

Observe that $\Gamma_1\cong\mathbb{Z}^{k_1}$ and $\Gamma_2\cong\mathbb{Z}^{k_2}$ with $k_1=\text{rank}(\Gamma_1)$ and $k_2=\text{rank}(\Gamma_2)$. Then $\Gamma \cong \mathbb{Z}^{k_1}\times\mathbb{Z}^{k_2}=\mathbb{Z}^k$ with $k=k_1+k_2=\text{rank}(\Gamma)$. Since $\text{Ker}(\pi)=\text{Ker}(\mathit{id})\times\ker(\text{exp})\cong \mathbb{Z}$, then $\tilde{\Gamma}\cong \text{Ker}(\pi)\times \Gamma\cong \mathbb{Z}\times \mathbb{Z}^k$. Therefore,
	\begin{equation}\label{eq_dem_prop_conmutativo_c1_rangoW_1}
	\text{rank}\left(\Gamma\right)=k+1.	
	\end{equation}	 
On the other hand, since $\Gamma$ acts properly and discontinuously on $\mathbb{C}\times\mathbb{C}^\ast$, then $\tilde{\Gamma}$ acts properly and discontinuously on $\mathbb{C}\times\mathbb{C}$, which is simply connected. Then Theorem \ref{thm_obdim_2} implies $\text{rank}\left(\tilde{\Gamma}\right)\leq 4$. This, together with (\ref{eq_dem_prop_conmutativo_c1_rangoW_1}) yields $\text{rank}(\Gamma)\leq 3$. Using Proposition \ref{prop_conmutativo_caso1_rangoG_rangoW} we conclude that $\text{rank}(W)\leq 3$.\\

Now we will verify that $\mu$ satisfies the condition (C). Let $\{w_k\}\subset W$ be a sequence of distinct elements such that $w_k\rightarrow 0$. Consider the sequence $\{\mu(w_k)\}\subset\mathbb{C}^\ast$, and assume that it does not converge to $0$ or $\infty$, then there are open neighbourhoods $U_0$ and $U_\infty$ of $0$ and $\infty$ respectively such that $\{\mu(w_k)\}\subset \mathbb{CP}^1\setminus\left(U_0\cup U_\infty\right)$. Since $\mathbb{CP}^1$ is compact and $\mathbb{CP}^1\setminus\left(U_0\cup U_\infty\right)$ is a closed subset of $\mathbb{CP}^1$, then $\mathbb{CP}^1\setminus\left(U_0\cup U_\infty\right)$ is compact, and therefore there is a converging subsequence of $\{\mu(w_k)\}$, still denoted the same way. Let $z\in\mathbb{C}^\ast$ such that $\mu(w_k)\rightarrow z$, then $\gamma_k\rightarrow \text{Diag}\left(z^{-3}, 1,1\right)$, contradicting that $\Gamma$ is discrete. This proves that $\mu$ satisfies the condition (C). 
\end{proof}

In order to describe the Kulkarni limit set, we will divide all groups $\Gamma = \Gamma_{W,\mu}$ of this case into the following sub-cases:

\begin{center}
\begin{tabular}{l|p{10cm}}
Case & Conditions \\ \hline
C1.1 & $\mu(W)$ has rational rotations and $W$ is discrete. \\
C1.2 & $\mu(W)$ has rational rotations and $W$ is not discrete. \\
C1.3 & $\mu(W)$ has no rational rotations but has irrational rotations, and $W$ is discrete.\\
C1.4 & $\mu(W)$ has no rational or irrational rotations, and $W$ is discrete.\\
C1.5 & $\mu(W)$ has no rational rotations but has irrational rotations, and $W$ is not discrete.\\
C1.6 & $\mu(W)$ has no rational or irrational rotations, and $W$ is not discrete.\\
\end{tabular}
\end{center}

We say that the commutative group $\Gamma=\Gamma_{W,\mu}$ satisfy the condition {\small \textbf{(F)}} if there is a sequence $\{w_k\}\subset W$ such that $w_k\rightarrow\infty$, $\mu(w_k)\rightarrow 0$ and $w_k\mu(w_k)^3\rightarrow b\in\mathbb{C}^\ast$.

\begin{obs}
If $\Gamma=\Gamma_{W,\mu}$ is cyclic, then $W=\langle w\rangle$ (Proposition \ref{prop_conmutativo_caso1_rangoG_rangoW}). Then the group $\Gamma$ is either generated by a loxo-parabolic or an ellipto-parabolic element depending on whether $|\mu(w)|\neq 1$ or $|\mu(w)|=1$. According to Propositions 4.2.10 and 4.2.19 of \cite{ckg_libro}, $\Lambda_{\text{Kul}}(\Gamma)=\overleftrightarrow{e_1,e_2}\cup\overleftrightarrow{e_2,e_3}$ and $\Lambda_{\text{Kul}}(\Gamma)=\overleftrightarrow{e_1,e_2}$ respectively. Therefore we will assume that the group $\Gamma$ is not cyclic.
\end{obs}

Now, we describe the Kulkarni limit sets.
 
\begin{thm}\label{thm_case1_kulkarni}
Let $\Gamma\subset\text{PSL}\left(3,\mathbb{C}\right)$ be a commutative discrete group having the form given in Proposition \ref{prop_conmutativo_c1_rangoW}, then
	$$
	\Lambda_{\text{Kul}}(\Gamma) = \begin{cases}
	\{e_1,e_2\}, & \text{Cases C1.3 or C1.4 with condition \small \textbf{(F)} not holding.} \\
	\overleftrightarrow{e_1,e_2}, & \begin{cases}
						 \text{Cases C1.3 or C1.4, satisfying condition \small \textbf{(F)}} &\\
						 \text{Case C1.1} & \\	
						 \end{cases}	  \\
	\{e_1\}\cup\overleftrightarrow{e_2,e_3}, & \text{Cases C1.5 or C1.6 with condition \small \textbf{(F)} not holding}.\\
	\overleftrightarrow{e_1,e_2}\cup\overleftrightarrow{e_2,e_3}, & \begin{cases}
						 \text{Cases C1.5 or C1.6, satisfying condition \small \textbf{(F)}} &\\
						 \text{Case C1.2} & \\	
						 \end{cases}	  \\
	\end{cases}
	$$
\end{thm}

For the sake of clarity, we break down the proof of Theorem \ref{thm_case1_kulkarni} into Lemmas \ref{lem_casoc1_L0}-\ref{lem_Kul_casoc15c16}. Lemma \ref{lem_casoc1_L0} determines the set $L_0(\Gamma)$ for all cases. Lemma \ref{lem_Kul_casoc11} describe Case C1.1, Lemma \ref{lem_Kul_casoc12} describes Case C1.2, Lemma \ref{lem_Kul_casoc13c14} describes Cases C1.3 and C1.4. Finally, Lemma \ref{lem_Kul_casoc15c16} describes Cases C1.5 and C1.6.

\begin{lem}\label{lem_casoc1_L0}
Let $\Gamma\subset U_+$ be a commutative discrete group of the form given in Proposition \ref{prop_conmutativo_c1_rangoW}, then
	$$L_0(\Gamma)=\begin{cases}
	\{e_1,e_2\},& \text{If }\mu(W)\text{ does not contain rational rotations}\\	
	\overleftrightarrow{e_1,e_2},& \text{If }\mu(W)\text{ contains rational rotations}	
	\end{cases}	
	$$
\end{lem}

\begin{proof}
Since $e_1$ and $e_2$ are global fixed points of $\Gamma$, in both cases we have $\{e_1,e_2\}\subset L_0(\Gamma)$. A direct computation shows that if $X\in L_0(\Gamma)$, then $X\in\overleftrightarrow{e_1,e_2}$, that is, $L_0(\Gamma)\subset \overleftrightarrow{e_1,e_2}$.\\

It is straightforward to verify that, if $\mu(W)$ contains rational rotations, then every point in $\overleftrightarrow{e_1,e_2}$ is in $L_0(\Gamma)$, and therefore $L_0(\Gamma)=\overleftrightarrow{e_1,e_2}$. Also, if $\mu(W)$ doesn't contain rational rotations, then $L_0(\Gamma)=\{e_1,e_2\}$. 
\end{proof}

\begin{lem}\label{lem_Kul_casoc11}
Let $\Gamma=\Gamma_{W,\mu}\subset U_+$ be a commutative discrete group of the form given in Proposition \ref{prop_conmutativo_c1_rangoW}. If $W$ is discrete and $\mu(W)$ contains rational rotations, then $\Lambda_{\text{Kul}}(\Gamma)=\overleftrightarrow{e_1,e_2}$.
\end{lem}

\begin{proof}
Since $\mu(W)$ contains rational rotations, $L_0(\Gamma)=\overleftrightarrow{e_1,e_2}$ (Lemma \ref{lem_casoc1_L0}). Since $W$ is discrete, then $\Gamma$ has no sequences with quasi-projective limit of the form (ii) in Table \ref{fig_c1_casos_qp}, and this implies $\Lambda_{\text{Kul}}(\Gamma)\subset\overleftrightarrow{e_1,e_2}$ (Proposition \ref{prop_eq_in_kuld}). All of this yields $\overleftrightarrow{e_1,e_2}=L_0(\Gamma)\subset \Lambda_{\text{Kul}}(\Gamma)\subset\overleftrightarrow{e_1,e_2}$. 
\end{proof}

\begin{lem}\label{lem_Kul_casoc12}
Let $\Gamma=\Gamma_{W,\mu}\subset U_+$ be a commutative discrete group of the form given in Proposition \ref{prop_conmutativo_c1_rangoW}. If $\mu(W)$ contains rational rotations and $W$ is not discrete, then $\Lambda_{\text{Kul}}(\Gamma)=\overleftrightarrow{e_1,e_2}\cup\overleftrightarrow{e_2,e_3}$.
\end{lem}

\begin{proof}
Since $W$ is not discrete, there is a sequence $\{w_n\}\subset W$ such that $w_n\rightarrow 0$, and therefore either $\mu(w_n)\rightarrow \infty$ or $\mu(w_n)\rightarrow 0$ (Proposition \ref{prop_conmutativo_c1_rangoW}). We can assume without loss of generality that the former happens (otherwise, consider the sequence $\{-w_n\}$). Then the quasi-projective limit of the sequence $\{\gamma_{w_n}\}$ is $\tau=\text{Diag}(0, 1, 1)$. Let $z\in \mathbb{CP}^2\setminus L_0(\Gamma)$, then $z\not\in\text{Ker}(\tau)=\{e_1\}$ (Lemma \ref{lem_casoc1_L0}). Therefore the set of accumulation points of points in $\mathbb{CP}^2\setminus L_0(\Gamma)$ is $\text{Im}(\tau)=\overleftrightarrow{e_2,e_3}$. Then $\overleftrightarrow{e_1,e_2}\cup\overleftrightarrow{e_2,e_3}\subset L_0(\Gamma)\cup L_1(\Gamma)$. Using Observation \ref{obs_caso1_llenando_las_2_lineas_ya} we conclude that $\Lambda_{\text{Kul}}(\Gamma)=\overleftrightarrow{e_1,e_2}\cup\overleftrightarrow{e_2,e_3}$ whenever $W$ contains rational rotations and is not discrete.
\end{proof}

\begin{lem}\label{lem_Kul_casoc13c14}
Let $\Gamma=\Gamma_{W,\mu}\subset U_+$ be a commutative discrete group of the form given in Proposition \ref{prop_conmutativo_c1_rangoW}. If $\mu(W)$ contains no rational rotations and $W$ is discrete, then
	$$\Lambda_{\text{Kul}}(\Gamma)=\begin{cases}
	\overleftrightarrow{e_1,e_2},& \text{ if }\Gamma\text{ satisfy condition }\text{\small \textbf{(F)}}\\
	\{e_1,e_2\},& \text{ any other case}
	\end{cases}.$$
\end{lem}

\begin{proof}
If $W$ is discrete, using Table \ref{fig_c1_casos_qp} we conclude that $\overleftrightarrow{e_2,e_3}$ cannot be contained in $\Lambda_{\text{Kul}}(\Gamma)$. Then, if $\mu(W)$ contains no rational rotations, either $\Lambda_{\text{Kul}}(\Gamma)=\{e_1,e_2\}$ or $\Lambda_{\text{Kul}}(\Gamma)=\overleftrightarrow{e_1,e_2}$. This argument, together with Table \ref{fig_c1_casos_qp} concludes the proof.
\end{proof}

\begin{lem}\label{lem_Kul_casoc15c16}
Let $\Gamma=\Gamma_{W,\mu}\subset U_+$ be a commutative discrete group of the form given in Proposition \ref{prop_conmutativo_c1_rangoW}. If $\mu(W)$ contains no rational rotations and $W$ is not discrete, then
	$$\Lambda_{\text{Kul}}(\Gamma)=\begin{cases}
	\overleftrightarrow{e_1,e_2}\cup\overleftrightarrow{e_2,e_3},& \text{ if }\Gamma\text{ satisfy condition }\text{\small \textbf{(F)}}\\
	\{e_1\}\cup\overleftrightarrow{e_2,e_3},& \text{ any other case}
	\end{cases}.$$
\end{lem}

\begin{proof}
Since $\mu(W)$ contains no rational rotations, $L_0(\Gamma)=\{e_1,e_2\}$ (Lemma \ref{lem_casoc1_L0}). Since $W$ is not discrete, then there is a sequence $\{w_k\}\subset W$ such that $w_k\rightarrow 0$, and then $\mu(w_k)\rightarrow \infty$ or $\mu(w_k)\rightarrow 0$ (Proposition\ref{prop_conmutativo_c1_rangoW}). In the former case, we can conclude using Table \ref{fig_c1_casos_qp}, that $\overleftrightarrow{e_2,e_3}\subset\Lambda_{\text{Kul}}(\Gamma)$. In the latter case we can consider the sequence $\{-w_k\}$ which satisfies $\mu(-w_k)\rightarrow \infty$, and we conclude again that $\overleftrightarrow{e_2,e_3}\subset\Lambda_{\text{Kul}}(\Gamma)$.\\
If the group $\Gamma$ satisfies condition {\small \textbf{(F)}}, using Table \ref{fig_c1_casos_qp} it follows that $\overleftrightarrow{e_1,e_2}\subset\Lambda_{\text{Kul}}(\Gamma)$, and using Observation \ref{obs_caso1_llenando_las_2_lineas_ya} we conclude that $\Lambda_{\text{Kul}}(\Gamma)= \overleftrightarrow{e_1,e_2}\cup\overleftrightarrow{e_2,e_3}$. Analogously, if $\Gamma$ doesn't satisfy condition {\small \textbf{(F)}}, $\Lambda_{\text{Kul}}(\Gamma)= \{e_1\}\cup\overleftrightarrow{e_2,e_3}$.
\end{proof}
 
In Example \ref{ej_falso_hopf} we show a group belonging to Case C1.6. This example is worth noting because it was conjectured that only fundamental groups of Hopf surfaces had a Kulkarni limit set consisting of a line and a point. This is an example of a group with this limit set, which is not a fundamental group of a Hopf surface (since such groups are cyclic, \cite{kato1975topology}).  

\begin{ejem}\label{ej_falso_hopf}
Let $W=\text{Span}_{\mathbb{Z}}\{1,\sqrt{2}\}$ and let $\mu:(W,+)\rightarrow(\mathbb{C}^{\ast},\cdot)$ be the group homomorphism given by $\mu(1)=e^{-1}$ and $\mu(\sqrt{2})=e^{\sqrt{2}}$. Consider the commutative group $\Gamma=\Gamma_{W,\mu}$. Let $\{w_n\}\subset W$ be a sequence of distinct elements such that $w_n=p_n+q_n\sqrt{2}\rightarrow 0$. Assuming without loss of generality that $p_n,-q_n\rightarrow \infty$, we have $\mu(w_n)=\mu(1)^{p_n}\mu(\sqrt{2})^{q_n}=e^{-p_n+q_n\sqrt{2}}\rightarrow 0$. Therefore $\Gamma$ is discrete (Proposition \ref{prop_conmutativo_c1_rangoW}).\\
Observe that $\mu(W)$ contains neither rational nor irrational rotations. Otherwise, there would be $x\in\mu(W)\setminus\{1\}$ such that $|x|=1$, and then $p+q\sqrt{2}\in \mathbb{Z}$ for some $p,q\in\mathbb{Z}$. This would contradict that $\{1,\sqrt{2}\}$ are rationally independent. This verifies that $\Gamma$ belongs to the case C1.6. Using Theorem \ref{thm_case1_kulkarni}, if follows that $\Lambda_{\text{Kul}}(\Gamma)=\{e_1\}\cup\overleftrightarrow{e_2,e_3}$ unless condition {\small \textbf{(F)}} is satisfied. That is, unless there is a sequence $\{w_n\}\subset W$ such that $w_n\rightarrow \infty$, $\mu(w_n)\rightarrow 0$ and $w_n\mu(w_n)^3\rightarrow b\in\mathbb{C}^{\ast}$. Assume that this happens. Since $w_n\rightarrow \infty$, there are the following possibilities for the sequences $\{p_n\}\text{, }\{q_n\}\subset\mathbb{Z}$:
\begin{enumerate}
\item If both sequences $\{p_n\},\{q_n\}$ are bounded, then there exists $R>0$ such that $|p_n|,|q_n|<R$ and then $|p_n+q_n\sqrt{2}|<R(\sqrt{2}+1)$. This contradicts that $w_n\rightarrow\infty$.
\item If $\mu(w_n)\rightarrow 0$ with $p_n\rightarrow\infty$ and $\{q_n\}$ is bounded, then
	$$w_n\mu(w_n)^3 = \underbrace{p_n e^{-3p_n}}_{\rightarrow 0}\underbrace{e^{3q_n\sqrt{2}}}_{\text{bounded}} + \underbrace{q_n\sqrt{2}}_{\text{bounded}} \underbrace{e^{-3p_n+3q_n\sqrt{2}}}_{\rightarrow 0} \rightarrow 0 $$

\item If $\mu(w_n)\rightarrow 0$ with $q_n\rightarrow-\infty$ and $\{p_n\}$ is bounded, then
	$$w_n\mu(w_n)^3 = \underbrace{p_n}_{\text{bounded}}\underbrace{e^{-3p_n+3q_n\sqrt{2}}}_{\rightarrow 0} + \underbrace{q_n\sqrt{2}e^{3q_n\sqrt{2}}}_{\rightarrow 0} \underbrace{e^{-3p_n}}_{\text{bounded}} \rightarrow 0.$$
\end{enumerate}
Then condition {\small \textbf{(F)}} doesn't hold. Therefore $\Lambda_{\text{Kul}}(\Gamma)=\{e_1\}\cup\overleftrightarrow{e_2,e_3}$.
\end{ejem}

\subsection{Case 2}

In this subsection, we study the case of discrete subgroups of $U_+$ conjugate to a diagonal group. We start by describing the form of these groups.

\begin{prop}\label{prop_conmutativo_c3_descripcion}
Let $\Gamma\subset U_+$ be a commutative subgroup such that each element of $\Gamma$ has the form $\text{Diag}(\alpha,\beta,\alpha^{-1}\beta^{-1})$. Then there exist two multiplicative subgroups $W_1,W_2\subset(\mathbb{C}^\ast,\cdot)$ such that
	\begin{equation}\label{eq_prop_conmutativo_c3_descripcion_2}
	\Gamma=\Gamma_{W_1,W_2}=\left\{\text{Diag}(w_1,w_2,1)\,\middle|\,w_1\in W_1,\;w_2\in W_2\right\}.	
	\end{equation}
\end{prop}

\begin{proof}
Let $\Gamma\subset U_+$ be a commutative group as in the hypothesis. Let $\gamma=\text{Diag}(\alpha,\beta,\alpha^{-1}\beta^{-1})\in\Gamma$, then $\gamma=\text{Diag}(\alpha^2\beta,\alpha\beta^2,1)$. Let $W_1,W_2\subset \mathbb{C}^\ast$ be the two multiplicative groups given by $W_1=\lambda_{13}(\Gamma)$ and $W_2=\lambda_{23}(\Gamma)$, then $\gamma=\text{Diag}(w_1,w_2,1)$, where $w_1 := \alpha^2\beta = \lambda_{13}(\gamma)\in W_1$ and $w_2 := \alpha\beta^2 = \lambda_{23}(\gamma)\in W_2$. Then $\Gamma$ has the form given by (\ref{eq_prop_conmutativo_c3_descripcion_2}).
\end{proof}

\begin{prop}\label{prop_caso_diagonal_rango}
Let $\Gamma\subset U_+$ be a diagonal discrete group such that every element has the form $\gamma=\text{Diag}(w_1,w_2,1)$. Then $\text{rank}(\Gamma)\leq 2$.	
\end{prop}

\begin{proof}
Recall the group morphisms $\lambda_{ij}$ defined in the beginning of Section \ref{sec_commutative_triangular}. Let $\mu:\Gamma\rightarrow\mathbb{R}^2$ given by $\mu(\gamma)=\left(\log |\lambda_{13}(\gamma)|,\log |\lambda_{23}(\gamma)|\right)$. Clearly $\mu$ is well defined, and it is a group homomorphism between $\Gamma$ and the additive subgroup $\mu(\Gamma)\subset\mathbb{R}^2$. Furthermore $\text{Ker}(\mu)=\{\mathit{id}\}$, and then $\mu:\Gamma\rightarrow\mu(\Gamma)$ is a group isomorphism. Since $\Gamma$ is discrete, then $\mu(\Gamma)$ is discrete and therefore $\text{rank}\left(\mu(\Gamma)\right)\leq 2$, then $\text{rank}(\Gamma)\leq 2$.	
\end{proof}

Proposition \ref{prop_caso_diagonal_rango} implies that $\text{rank}(W_1)+\text{rank}(W_2)\leq 2$. If $\text{rank}(W_1)=1$, $\text{rank}(W_2)=0$ or $\text{rank}(W_1)=0$, $\text{rank}(W_2)=1$, then $\Gamma$ is cyclic, and its Kulkarni limit set is described in Section 4.2 of \cite{ckg_libro}. The cases $\text{rank}(W_1)=2$, $\text{rank}(W_2)=0$ and $\text{rank}(W_1)=0$, $\text{rank}(W_2)=2$ imply that $\Gamma$ is not discrete. Therefore we just have to describe the case:
	$$\Gamma:=\Gamma_{\alpha,\beta}=\left\{\text{Diag}\left(\alpha^{n},\beta^{m},1\right)\,\middle|\,n,m\in\mathbb{Z}\right\}.$$
for some $\alpha,\beta\in \mathbb{C}^\ast$ such that $|\alpha|\neq 1$ or $|\beta|\neq 1$.\\

Consider a sequence of distinct elements $\{\gamma_k\}\subset \Gamma$ given by $\gamma_k=\text{Diag}(\alpha^{n_k},\beta^{m_k},1)$. In Table \ref{fig_cdiag_casos_qp} we show all the possible quasi-projective limits of the sequences $\{\gamma_k\}$.\\

\begin{table}
\begin{center}
  \begin{tabular}{ | l | c | c | c | c | }
    \hline
    Case & $\tau$ & Conditions & Ker($\tau$) & Im($\tau$) \\ \hline
    (i) & $\text{Diag}(1,0,0)$ & \begin{tabular}{c}
	\small $\alpha^{n_k}\rightarrow\infty$, $\beta^{m_k}\rightarrow \infty$ and $\alpha^{n_k}\beta^{-m_k}\rightarrow \infty$ \\
	\small $\alpha^{n_k}\rightarrow\infty$ and $\beta^{m_k}\rightarrow b\in\mathbb{C}$	
	\end{tabular}
		
	 & $\overleftrightarrow{e_2,e_3}$ & $\{e_1\}$  \\ \hline
    (ii) & $\text{Diag}(0,1,0)$ & \begin{tabular}{c}
	\small $\alpha^{n_k}\rightarrow\infty$, $\beta^{m_k}\rightarrow \infty$ and $\alpha^{-n_k}\beta^{m_k}\rightarrow \infty$ \\
	\small $\alpha^{n_k}\rightarrow 0$ and $\beta^{m_k}\rightarrow \infty$	
	\end{tabular} & $\overleftrightarrow{e_1,e_3}$ & $\{e_2\}$  \\ \hline
    (iii) & $\text{Diag}(0,0,1)$ & \small $\alpha^{n_k}\rightarrow 0$ and $\beta^{m_k}\rightarrow 0$
   & $\overleftrightarrow{e_1,e_2}$ & $\{e_3\}$  \\ \hline
    (iv) & $\text{Diag}(b,0,1)$ & $\alpha^{n_k}\rightarrow b\in\mathbb{C}^\ast$ and $\beta^{m_k}\rightarrow 0$ & $\{e_2\}$ & $\overleftrightarrow{e_1,e_3}$ \\
    \hline
    (v) & $\text{Diag}(b,1,0)$ & $\alpha^{n_k}\rightarrow \infty$, $\beta^{m_k}\rightarrow \infty$ and $\alpha^{n_k}\beta^{-m_k}\rightarrow b\in\mathbb{C}^\ast$ & $\{e_3\}$ & $\overleftrightarrow{e_1,e_2}$ \\
    \hline
    (vi) & $\text{Diag}(0,b,1)$ & $\alpha^{n_k}\rightarrow 0$ and $\beta^{m_k}\rightarrow b\in\mathbb{C}^\ast$ & $\{e_1\}$ & $\overleftrightarrow{e_2,e_3}$ \\
    \hline
  \end{tabular}
\caption{Quasi-projective limits in the diagonal case.}
\label{fig_cdiag_casos_qp}
\end{center}
\end{table}	

Lemmas \ref{lem_caso_diagonal_L0} and \ref{lem_caso_diagonal_L1} determine the sets $L_0$ and $L_1$ for groups $\Gamma_{\alpha,\beta}$ and they will be used to determine the Kulkarni limit sets in Theorem \ref{thm_kulkarni_diagonales}.

\begin{lem}\label{lem_caso_diagonal_L0}
Let $\Gamma_{\alpha,\beta}\subset U_+$ be a discrete group containing loxodromic elements then
$$	
L_0(\Gamma)=\begin{cases}
\overleftrightarrow{e_1,e_2}\cup\{e_3\}, & \text{if }\alpha^n=\beta^m\text{ for some }n,m\in\mathbb{Z} \\
\{e_1,e_2,e_3\}, & \text{if there are no }n,m\in\mathbb{Z} \text{ such that }\alpha^n=\beta^m.
\end{cases}
$$
\end{lem}

\begin{proof}
Let $\Gamma=\Gamma_{\alpha,\beta}$, suppose that $\alpha^n=\beta^m$ for some $n,m\in\mathbb{Z}$, and let $z=\left[z_1:z_2:z_3\right]\in L_0(\Gamma)$. Then $\left[\alpha^p z_1: \beta^q z_2: z_3\right]= \left[z_1:z_2:z_3\right]$ for an infinite number of $p,q\in\mathbb{Z}$.\\
If $z_3\neq 0$, then either $\alpha$ and $\beta$ are rational rotations or $z_1=z_2=0$. If the former happens, then $\Gamma$ contains no loxodromic elements, contradicting the hypothesis. If the latter happens, then $z=e_3$. If $z_3=0$, we can assume without loss of generality that $z_1\neq 0$. If $z_2=0$, then $z=e_2$. If $z_2\neq 0$, then $\alpha^p=\beta^q$ for an infinite number of integers $p,q$, and this implies that $\alpha^{jn}=\beta^{jm}$ for any $j\in\mathbb{Z}$. Then $z\in \overleftrightarrow{e_1,e_2}$.\\
If there are no $n,m\in\mathbb{Z}$ such that $\alpha^n=\beta^m$, then no point in $\overleftrightarrow{e_1,e_2}\setminus\{e_1,e_2\}$ satisfies that $\left[\alpha^p z_1: \beta^q z_2: z_3\right] = \left[z_1:z_2:z_3\right]$ for an infinite number of $p,q\in\mathbb{Z}$.
\end{proof}	

\begin{lem}\label{lem_caso_diagonal_L1}
Let $\Gamma_{\alpha,\beta}\subset U_+$ be a discrete group containing loxodromic elements, then
	$$
	L_1(\Gamma)=
	\begin{cases}
	\{e_1,e_2\}, & \text{if }|\alpha|>1>|\beta|\text{ or }|\alpha|<1<|\beta| \\
	\{e_1,e_3\}, & \text{if }|\alpha|>|\beta|>1\text{ or }|\alpha|<|\beta|<1 \\
	\{e_1\}\cup\overleftrightarrow{e_2,e_3}, & \text{if }\beta\text{ is an irrational rotation}.
	\end{cases}		
	$$	 
\end{lem}

\begin{proof}
Let $\alpha,\beta\in\mathbb{C}^\ast$, since $\Gamma=\Gamma_{\alpha,\beta}$ contains loxodromic elements, we can assume without loss of generality that $|\alpha|\neq 1$. Let $\{\gamma_k\}\subset\Gamma$ be a sequence of distinct elements given by $\gamma_k=\text{Diag}(\alpha^k,\beta^k,1)$. Let $z=\left[z_1:z_2:z_3\right]\in \mathbb{CP}^2\setminus L_0(\Gamma)$, then $\gamma_k z=\left[\alpha^k z_1:\beta^k z_2:z_3\right]$. We have two essentially different sequences in $\Gamma$, $\{\gamma^k\}$ and $\{\gamma^{-k}\}$. Since $|\alpha|\neq 1$, we can assume, without loss of generality that $|\alpha|>1$, then $\alpha^k\rightarrow\infty$ as $k\rightarrow \infty$. We have three cases:
		\begin{enumerate}[(i)]
		\item $|\beta|<1$, then $\beta^k\rightarrow 0$ as $k\rightarrow \infty$, and $\gamma_k z\rightarrow e_1$ as $k\rightarrow \infty$. Analogously, $\gamma_k z\rightarrow e_2$ as $k\rightarrow -\infty$.
		\item $|\beta|=1$, $\beta$ cannot be a rational rotation because $\Gamma$ is torsion-free, then $\beta$ is an irrational rotation. Then there are subsequences (still denoted by $\{\gamma_k\}$) converging to any $b\in\mathbb{C}$, $|b|=1$. Then $\gamma_k z\rightarrow e_1$ as $k\rightarrow \infty$ and $\gamma_k z\rightarrow \left[0:b^{-1}z_2:z_3\right]$ as $k\rightarrow -\infty$. 
		\item $|\beta|>1$, then $\beta^k\rightarrow \infty$ as $k\rightarrow \infty$. There are two possibilities:
			\begin{itemize}
			\item $\alpha^k\beta^{-k}\rightarrow\infty$ as $k\rightarrow \infty$, then $\gamma_k z\rightarrow e_1$ as $k\rightarrow \infty$. Also, $\gamma_k z\rightarrow e_3$ as $k\rightarrow -\infty$.
			\item $\alpha^k\beta^{-k}\rightarrow 0$ as $k\rightarrow \infty$, $\gamma_k z\rightarrow e_2$ as $k\rightarrow \infty$ and $\gamma_k z\rightarrow e_2$ as $k\rightarrow \infty$, and $\gamma_k z\rightarrow e_3$ as $k\rightarrow -\infty$.
			\end{itemize}
		\end{enumerate}		 
\end{proof}

If $\Gamma=\Gamma_{\alpha,\beta}\subset U_+$ is a discrete group containing loxodromic elements, then putting together Lemmas \ref{lem_caso_diagonal_L0} and \ref{lem_caso_diagonal_L1} we have the following cases:\\

\noindent If $\alpha^n=\beta^m$ for some $n,m\in\mathbb{Z}$:
	\begin{itemize}
	\item[{\scriptsize \textbf{[D1]}}] $L_0(\Gamma)\cup L_1(\Gamma)=\overleftrightarrow{e_1,e_2}\cup\{e_3\}$, if $|\alpha|>1>|\beta|$ or $|\alpha|<1<|\beta|$.
	\item[{\scriptsize \textbf{[D2]}}] $L_0(\Gamma)\cup L_1(\Gamma)=\overleftrightarrow{e_1,e_2}\cup\{e_3\}$, if $|\alpha|>|\beta|>1$ or $|\alpha|<|\beta|<1$.
	\end{itemize}
If there are no integers $n,m$ such that $\alpha^n=\beta^m$:
	\begin{itemize}
	\item[{\scriptsize \textbf{[D3]}}] $L_0(\Gamma)\cup L_1(\Gamma)=\{e_1,e_2,e_3\}$, if $|\alpha|>1>|\beta|$ or $|\alpha|<1<|\beta|$.
	\item[{\scriptsize \textbf{[D4]}}] $L_0(\Gamma)\cup L_1(\Gamma)=\{e_1,e_2,e_3\}$, if $|\alpha|>|\beta|>1$ or $|\alpha|<|\beta|<1$.
	\item[{\scriptsize \textbf{[D5]}}] $L_0(\Gamma)\cup L_1(\Gamma)=\overleftrightarrow{e_1,e_2}\cup\overleftrightarrow{e_2,e_3}$, if $\beta$ is an irrational rotation.
	\end{itemize}

The following theorem gives the full description of the Kulkarni limit set for groups in this case, we will use the notation described in the previous paragraph.

\begin{thm}\label{thm_kulkarni_diagonales}
Let $\Gamma_{\alpha,\beta}\subset U_+$ be a discrete group containing loxodromic elements, then
	\begin{enumerate}[(i)]
	\item $\Lambda_{\text{Kul}}(\Gamma)=\overleftrightarrow{e_1,e_2}\cup\{e_3\}$ in Cases {\scriptsize \textbf{[D1]}} and {\scriptsize \textbf{[D2]}}.
	\item $\Lambda_{\text{Kul}}(\Gamma)=\{e_1,e_2,e_3\}$ in Cases {\scriptsize \textbf{[D3]}} and {\scriptsize \textbf{[D4]}}.
	\item $\Lambda_{\text{Kul}}(\Gamma)=\overleftrightarrow{e_1,e_2}\cup\overleftrightarrow{e_2,e_3}$ in Case {\scriptsize \textbf{[D5]}}.
	\end{enumerate}
\end{thm}

\begin{proof}
We consider a sequence $\{\gamma^k\}\subset\Gamma$. We will determine the quasi-projective limits of this sequence using Table \ref{fig_cdiag_casos_qp}, and then using Proposition \ref{prop_descripcion_Eq}, we determine the set $L_2(\Gamma)$, and therefore $\Lambda_{\text{Kul}}(\Gamma)$.
	\begin{itemize}
	\item In Case {\scriptsize \textbf{[D1]}}, $\alpha^k\rightarrow\infty$, $\beta^k\rightarrow 0$, and the quasi-projective limit is given by (i) of Table \ref{fig_cdiag_casos_qp}. Therefore the orbits of compact subsets of $\mathbb{CP}\setminus L_0(\Gamma)\cup L_1(\Gamma)$ accumulate on $\{e_1\}$ and $\{e_2\}$ under the sequences $\{\gamma^k\}$ and $\{\gamma^{-k}\}$ respectively. Then $\Lambda_{\text{Kul}}(\Gamma)=\overleftrightarrow{e_1,e_2}\cup\{e_3\}$.
	\item In Case {\scriptsize \textbf{[D2]}}, $\alpha^k\rightarrow\infty$, $\beta^k\rightarrow \infty$ and $\alpha^k\beta^{-k}\rightarrow\infty$, then the quasi-projective limit is also given by (i) of Table \ref{fig_cdiag_casos_qp}. We have again, $\Lambda_{\text{Kul}}(\Gamma)=\overleftrightarrow{e_1,e_2}\cup\{e_3\}$.
	\item In Cases {\scriptsize \textbf{[D3]}} and {\scriptsize \textbf{[D4]}}, the orbits of compact subsets of $\mathbb{CP}\setminus L_0(\Gamma)\cup L_1(\Gamma)$ accumulate on $\{e_1\}$ (and accumulate on $\{e_1\}$ under the sequence of $\{\gamma^{-k}\}$). Then $\Lambda_{\text{Kul}}(\Gamma)=\{e_1,e_2,e_3\}$.
	\item In Case {\scriptsize \textbf{[D5]}}, we have $|\beta|=1$, and we can assume without loss of generality that $|\alpha|>1$. Then the quasi-projective limit of the sequence $\{\gamma_k\}$ is $\tau=\text{Diag}(1,0,0)$ and therefore the orbits of compact subsets of $\mathbb{CP}^2\setminus \text{Ker}(\tau)=\mathbb{CP}^2\setminus\overleftrightarrow{e_2,e_3}$ accumulate on $\text{Im}(\tau)=\{e_1\}$. If $K\subset\mathbb{CP}^2\setminus\left(\overleftrightarrow{e_1,e_2}\cup\overleftrightarrow{e_2,e_3}\right)$, then $K\subset\mathbb{CP}^2\setminus\overleftrightarrow{e_2,e_3}$. Therefore $L_2(\Gamma)=\{e_1\}$, we finally conclude that
		$$\Lambda_{\text{Kul}}(\Gamma)=\overleftrightarrow{e_1,e_2}\cup\overleftrightarrow{e_2,e_3}.\qedhere$$
	\end{itemize}
\end{proof}

\section{Proof of the Main Theorems}
\label{sec_proof_main_thm}
In this section we prove Theorem \ref{thm_main_solvable} for discrete subgroups of $U_{+}$ and Theorem \ref{thm_main_gen}. 

\begin{proof}[Proof of Theorem \ref{thm_main_solvable}]
Let $\Gamma\subset U_+$ be a complex Kleinian group such that its Kulkarni limit set does not consist of exactly four lines in general position. We can assume that $\Gamma$ contains loxodromic elements, otherwise it is described in detail in \cite{ppar}.\\

If $\Gamma$ is non-commutative, then in each sub-case of the proof of Theorem \ref{thm_descomposicion_caso_noconmutativo2} we have constructed an open subset $\Omega_\Gamma\subset\mathbb{CP}^2$ such that the orbits of every compact set $K\subset \Omega_\Gamma$ accumulate on $\mathbb{CP}^2\setminus \Omega_\Gamma$. Thus we can define a limit set for the action of $\Gamma$ by $\Lambda_\Gamma:=\Omega_\Gamma$.  This limit set describes the dynamics of $\Gamma$, and the open region $\Omega_\Gamma$ satisfies (i) and (ii) by construction. $\Gamma$ is finitely generated (see \cite{auslander}) and we have proven in Theorem \ref{thm_descomposicion_caso_noconmutativo2} that $\text{rank}(\Gamma)\leq 4$. This verifies (iii). Conclusion (iv) follows immediately from Theorem \ref{thm_descomposicion_caso_noconmutativo} and Corollary \ref{cor_lox_nivel34}.\\

If $\Gamma$ is commutative, it is conjugate to a sugroup of the Lie groups $C_1$ or $C_2$ (see Lemma \ref{lem_7casos_conmutativos} and the subsequent arguments). In this setting, the region $\Omega_{\text{Kul}}(\Gamma)$ satisfies conclusions (i) and (ii) as a consequence of Theorems \ref{thm_case1_kulkarni}, and \ref{thm_kulkarni_diagonales}.  Again, $\Gamma$ is finitely generated \cite{auslander} and $\text{rank}(\Gamma)\leq 4$ (Propositions \ref{prop_conmutativo_c1_rangoW} and \ref{prop_caso_diagonal_rango}). This proves conclusion (iii). On the other hand, $\Gamma\cong\mathbb{Z}^r$ with $r=\text{rank}(\Gamma)$, and then we can write $\Gamma$ as a trivial semidirect product of copies of $\mathbb{Z}$, thus verifying conclusion (iv).\\  

In both cases, conclusion (v) follows from Theorem \ref{teo_solvable_triangularizables}.
\end{proof}

As it was stated in the introduction, our ultimate goal is to prove Theorem \ref{thm_main_solvable} for solvable discrete subgroups of $\text{PSL}\left(3,\mathbb{C}\right)$ satisfying the hypothesis of the theorem. Using the ideas we have developed in this paper, we are able to prove a first generalization of Theorem \ref{thm_main_solvable}, stated in Theorem \ref{thm_main_gen}. 

\begin{proof}[Proof of Theorem \ref{thm_main_gen}]
Let $\Gamma$ be a solvable subgroup of $\text{PSL}\left(3,\mathbb{C}\right)$ as in the hypothesis of Theorem \ref{thm_main_gen}. Theorem \ref{teo_solvable_triangularizables} guaranties the existence of a virtually triangularizable finite index subgroup $\Gamma_0\subset\Gamma$. According to the hypothesis, $\Gamma_0$ is commutative and therefore, the non-empty region $\Omega_{\text{Kul}}(\Gamma_0)$ satisfies conclusions (i) and (ii) (Theorem \ref{thm_main_solvable}). Proposition 3.6 of \cite{bcn16} states that $\Omega_{\text{Kul}}(\Gamma)=\Omega_{\text{Kul}}(\Gamma_0)$, thus proving conclusions (i) and (ii). Finally, conclusion (iii) follows from Theorem \ref{teo_solvable_triangularizables}. 
\end{proof}

In order to prove the remaining conclusions for solvable subgroups of $\text{PSL}\left(3,\mathbb{C}\right)$ we need to study the extensions of upper triangular discrete groups of $U_+$. This will be done in subsequent works.

\section*{Acknowledgments}
The author would like to thank to the IMUNAM Cuernavaca, CIMAT, FAMAT UADY and their people for its hospitality and kindness during the writing of this paper. Finally, the author is specially thankful to \'Angel Cano, M\'onica Moreno, Waldemar Barrera, Juan Pablo Navarrete, Carlos Cabrera and Manuel Cruz for many valuable and helpful conversations.

\bibliographystyle{abbrv}
\bibliography{Bibliografia}

\end{document}